%% file: main_arXiv.tex
\documentclass[11pt,letterpaper]{article}

\input{_macros}
\definecolor{OliveGreen}{rgb}{0,0.6,0}
\usepackage{custom}
\usepackage[normalem]{ulem}

\makeatletter
\renewcommand{\paragraph}{%
  \@startsection{paragraph}{4}%
  {\z@}{1.25ex \@plus 1ex \@minus .2ex}{-1em}%
  {\normalfont\normalsize\bfseries}%
}
\makeatother

\begin{document}

\title{Analysis of Langevin midpoint methods using an anticipative Girsanov theorem}

 \author{
Matthew S.\ Zhang\thanks{
  Department of Computer Science at
  University of Toronto, and Vector Institute, \texttt{matthew.zhang@mail.utoronto.ca}
}
}

\maketitle

\begin{abstract}
    We introduce a new method for analyzing midpoint discretizations of stochastic differential equations (SDEs), which are frequently used in Markov chain Monte Carlo (MCMC) methods for sampling from a target measure $\pi \propto \exp(-V)$. Borrowing techniques from Malliavin calculus, we compute estimates for the Radon--Nikodym derivative for processes on $L^2([0, T); \R^d)$ which may \emph{anticipate} the Brownian motion, in the sense that they may not be adapted to the filtration at the same time. Applying these to various popular midpoint discretizations, we are able to improve the regularity and cross-regularity results in the literature on sampling methods. We also obtain a query complexity bound of $\Otilde(\frac{\kappa^{5/4} d^{1/4}}{\varepsilon^{1/2}})$ for obtaining a $\varepsilon^2$-accurate sample in $\KL$ divergence, under log-concavity and strong smoothness assumptions for $\nabla^2 V$.
\end{abstract}

\tableofcontents

\section{Introduction}

At their core, many algorithms for sampling over distributions in $\R^d$ boil down to approximating It\^o SDEs:
\[
    \D X_t = b(t, X_t) \, \D t + \sigma(t, X_t) \, \D B_t\,,
\]
where $(B_t)_{t \geq 0}$ is a Brownian process, and $b: \R_+ \times \R^d \to \R^d$, $\sigma: \R_+ \times \R^d \to \R^{d\times d}$ are suitably regular coefficients. 

A model example arises with the Langevin diffusion~\cite{smoluchowski1918try}, given by
\begin{align}\label{eq:langevin_diff}\tag{LD}
    \D X_t = -\nabla V(X_t) \, \D t + \sqrt{2} \, \D B_t\,,
\end{align}
for a continuously differentiable $V: \R^d \to \R$. \eqref{eq:langevin_diff} is the canonical algorithm for sampling from $\pi$.
A generalization~\cite{langevin1908theorie} of this diffusion, which incorporates a momentum component, is given by
\begin{align}\label{eq:underdamped_langevin_diff}\tag{ULD}
\begin{aligned}
    \D X_t &= P_t \, \D t \,, \\
    \D P_t &= \bigl\{-\gamma P_t -\nabla V(X_t)\bigr\}\, \D t + \sqrt{2\gamma} \, \D B_t\,, 
\end{aligned}
\end{align}
where $\gamma > 0$ is a friction coefficient. Under mild conditions, this converges to $\bs \pi(x,p) \propto \exp(-V(x) -\frac{\norm{p}^2}{2})$.

In this work, we consider the simulation of such equations. Generally, one hopes to quantify the number of queries made by the algorithm to $\nabla V$ or $V$ that are needed to guarantee that the $\KL(\operatorname{law}(X_T) \mmid \pi) \leq \varepsilon^2$, with particular focus on asymptotic dependence on quantities as the ambient dimension $d$ or the desired error $\varepsilon$. Exact simulation of the equation is generally impossible due to the non-linearity $\nabla V$ in the dynamics. Instead, one adopts methods for approximating the dynamics on small enough timescales. The simplest such method is the Euler--Maruyama discretization, which fixes a finite set of interpolant times $0 \leq t_1 \leq t_2 \leq \ldots \leq t_N$. On each of these intervals, we replace the nonlinearity $\nabla V(X_t)$ by $\nabla V(X_{t_k})$, which then gives a sequence of linear SDEs whose laws are easily simulated. 
Error estimates for such schemes have been extensively studied in both the SDE and sampling literature; within sampling, the monograph~\cite{chewisamplingbook} is an excellent reference. 

We can expect that a more refined scheme may outperform this simple interpolation. By analogy, to simulate an ordinary differential equation, the gradient estimator $\nabla V(X_{t_k})$ may be replaced with an expression which involves midpoints such as $\frac{1}{2} (X_{t_k} + X_{t_{k+1/2}})$, or other more complex schemes, which gives rise to a rich family of midpoint or Runge--Kutta schemes.

In this spirit, the work of~\cite{shen2019randomized} introduced the randomized midpoint scheme for simulating~\eqref{eq:ULD}, which introduced an additional random interpolant timestep in order to debias the gradient estimate. This was extended to~\eqref{eq:langevin_diff} in~\cite{he2020ergodicity, yu2023langevin}. To be precise, let us just look at this scheme on a single discretized step $[0,h]$, with $h > 0$ a step-size. Fix some interpolant time $\tau$, and write
\begin{align}\label{eq:rmp-std}
\begin{aligned}
    X^+ &= X_0 - \tau \nabla V(X_0) + \sqrt{2} B_\tau \\
    X_h &= X_0 - h \nabla V(X^+) + \sqrt{2} B_h\,.
\end{aligned}
\end{align}
After repeating this process $N$ times, $X_{Nh}$ is taken to be the approximate sample from $\pi$.
As shown in~\cite{shen2019randomized}, this scheme enjoys better error guarantees compared to the standard discretization, although the analysis in it and subsequent works~\cite{he2020ergodicity} are conducted solely in Wasserstein distance using coupling techniques. See also the works~\cite{kandasamy2024poisson, srinivasan2025poisson}.
Building on these approaches, the separate line of work~\cite{scr3, scr4} proposed a method which was able to strengthen these guarantees to bounds on the Kullback--Leibler ($\KL$) divergence between the terminal marginal $\operatorname{law}(X_T)$ and $\pi$. While these works have led to significant improvements in our understanding of these methods, they are unable to obtain process-level bounds between~\eqref{eq:langevin_diff} and~\eqref{eq:rmp-std} directly.

We resolve this question in the current work. We highlight an existing framework for Girsanov analysis of non-adapted processes. In general, consider some possibly anticipating transformation of the Brownian motion, $\msf T: C_0([0,T]; \R^d) \to C_0([0, T); \R^d)$ (where $C_0$ are the set of continuous paths $\omega: [0, T] \to \R^d$ with $\omega_0 = 0$),
\begin{align}\label{eq:transformation-generic}
    \msf T_t(\omega) = \omega_t + \int_0^t \bar \Mu_s (\omega)\, \D s\,,
\end{align}
for a process $(\bar \Mu_s)_{s \in [0,T]}$ in $\R^d$ which may anticipate the Brownian motion, in the sense that $\bar \Mu_s$ is not necessarily a measurable random variable with respect to the filtration $\ms F_s$ at time $s$, and $\omega = (\omega_t)_{t \in [0, T)} \in C_0([0,T]; \R^d)$ corresponds to a particular realization of the Brownian motion trajectory. Given a measure $\mathbf P$ over $\Omega \equiv C_0([0,T]; \R^d)$, one then produces $\mathbf Q = \msf T_{\#} \mathbf P$. It is reasonable to ask if $\mathbf Q \ll \mathbf P$, or to inquire about the magnitude of $\frac{\D \mathbf Q}{\D \mathbf P}$.

Although such processes $(\msf T_t(\omega))_{t \in [0, T)}$ lose many of the beneficial features of the It\^o calculus, it is indeed still possible~\cite{nualart2006malliavin} to quantify the Radon--Nikodym derivative between $\mathbf P$ and $\mathbf Q$ under sufficient regularity conditions of the transform $\msf T$, using techniques from Malliavin calculus. By closely investigating these results, it is possible to leverage them to obtain guarantees for such processes in information theoretic divergences such as $\msf{KL}$ or R\'enyi. 

\subsection{Related work}
\paragraph{Sampling and midpoint methods.}
The theory of sampling methods has seen extensive development in recent years. The continuous-time dynamics in~\eqref{eq:langevin_diff} and~\eqref{eq:underdamped_langevin_diff} have an elegant interpretation as gradient flows of $\KL(\cdot \mmid \pi)$~\cite{jordan1998variational, villani2008optimal, ambrosio2008gradient}. More relevant to us is the history of SDE discretizations, in particular ``unadjusted'' schemes which do not add a Metropolis--Hastings filter~\cite{hastings1970monte}. The Euler--Maruyama discretization was first introduced for the simulation of SDEs in~\cite{maruyama1955continuous}. Numerous refinements of this method have since been analyzed in the literature, which would be impossible to fully survey. See~\cite{kloeden2012numerical} for an introduction. 

Analysis within the sampling literature is concerned with guarantees for the law of the terminal iterate $X_T$, typically in a metric such as total variation or Wasserstein distance. In this context the Euler--Maruyama discretization of~\eqref{eq:langevin_diff} was first analyzed in~\cite{dalalyan2012sparse}, which showed the first non-asymptotic guarantees for this algorithm in $W_2^2$ distance. Numerous refinements have since been made, holding in ``metrics'' as strong as the $q$-R\'enyi divergence~\cite{dalalyan2017theoretical, durmus2019analysis, chewi2024analysis}. 

The randomized midpoint scheme was first introduced by~\cite{shen2019randomized}, and its analysis was refined by~\cite{he2020ergodicity, cao2020complexity, yu2023langevin}. Other Runge--Kutta type schemes have also been considered in the literature~\cite{li2019stochastic, wu2024stochastic}. We comment on~\cite{sanz2021wasserstein, paulin2024correction}, which introduces a similar midpoint method as~\eqref{eq:DM--ULMC-multistep} (defined in \S\ref{scn:dm-ulmc}) but uses only a singular midpoint. All of the aforementioned works only obtain guarantees in $W_2^2$, and do not provide techniques for obtaining guarantees in $\KL$ or $\Renyi_q$, as the resultant process will be non-adapted. While the analyses of~\cite{shen2019randomized, sanz2021wasserstein} were sufficient to achieve better complexity bounds for single midpoint discretizations of~\eqref{eq:ULD} in the $W_2^2$ metric, it is not possible at the moment to show similar guarantees in $\KL$ with only a single midpoint. Instead, we generalize their analysis to a scheme with two midpoints, which recovers their rates.
Finally, although not considered in this work, we also mention the existence of implicit discretizations of Langevin diffusions, which can also be seen as anticipating processes~\cite{wibisono2019proximal, salim2019stochastic, hodgkinson2021implicit}.

We also mention that there exists an army of methods for discretizing (randomized) Hamiltonian Monte Carlo dynamics~\cite{monmarche2021high, bou2023mixing, bou2023nonlinear, bou2025randomized}. These methods also provide guarantees in terms of closeness of the final iterate to $\bs \pi$, but cannot be seen as discretizations of~\eqref{eq:ULD} in the same way as the methods considered in this work.

\paragraph{Shifted composition.}
The parallel line of work~\cite{scr1, scr2, scr3, scr4} proposes a different method for analyzing midpoint schemes, known as the \textbf{shifted composition} rule. Unlike the present work, the shifted composition method is able to obtain sharper marginal guarantees by differentiating between the \emph{weak} and \emph{strong} errors of the algorithm. The former measures $\norm{\E X^\alg - \E X}$, while the latter measures $\norm{X^\alg - X}_{L^2_{\mbf P}}$, where $X^\alg, X$ are obtained from the algorithm and true process started at a single point. Secondly, the conditions in shifted composition are more user-friendly, as there is no need to compute the determinant of a Malliavin derivative, which introduces significant technical challenges in our work.

In light of this, we should justify why anticipating Girsanov theorems are necessary at all. Firstly, the $\KL$ bounds that we obtain are genuinely process-valued, which allows us to obtain stronger guarantees if one is interested in functionals of the entire stochastic process (for instance, auto-correlations). Secondly, the key assumption needed in~\cite{scr3, scr4} is \emph{cross-regularity}, which measures $\KL$ (or $\Renyi_q$) between the laws of $X^\alg$ and $Y$, where $X^\alg$ is obtained from an algorithm starting at $x$ and $Y$ is obtained from the reference process started at $y$. This requires a bound on $\KL(\operatorname{law}(X^\alg) \mmid \operatorname{law}(Y))$ in terms of $\norm{x-y}^2$ and other quantities. The shifted composition framework is unable to provide tools for directly analyzing this quantity, and the bounds obtained thus far have been largely on a case-by-case basis. On the other hand, the anticipating Girsanov theorem presented in this work can be used to obtain such bounds directly, and thus provides a complementary contribution to the literature. 

\paragraph{Anticipating stochastic calculus.} We would be remiss not to touch upon the history of anticipating stochastic calculus. The study of non-adapted processes has nearly as long a history as the standard It\^o calculus, and the same techniques used to define the anticipating Girsanov theorem were also used to place the Stratonovich calculus on a rigorous mathematical foundation~\cite{ocone1989generalized}. The differential and integral formulations of Malliavin calculus were developed independently by Malliavin~\cite{malliavin1978stochastic} and Skorohod~\cite{skorohod1976generalization} respectively. The application to non-adapted processes was developed in~\cite{nualart1988stochastic, buckdahn1991linear, obata2006white}, among many other works. See~\cite{pardoux2006applications} or the bibliography in~\cite{nualart2006malliavin} for a more comprehensive survey.

The original, non-anticipating Girsanov theorem is from~\cite{girsanov1960transformation}. Its anticipating variant is due to~\cite{buckdahn1994anticipative}, based on other generalizations of Girsanov's theorem~\cite{ramer1974nonlinear, kusuoka1974non}. It has been further generalized to other settings~\cite{ustunel1992transformation, enchev1993anticipative, ustunel2013transformation}, and has also seen applications in a large body of recent scholarship~\cite{kuo2013generalization, kuo2013ito, hwang2019anticipating}.

\subsection{Approach}

We briefly summarize the technical approach which leads us to our bounds. Broadly speaking, if we approximate our process $\bar \Mu_s$ by an elementary process $\sum_{k=0}^{T/\eta - 1} f_k((B_{[j\eta, (j+1)\eta)})_{j = 0, 1, \dotsc, T/\eta-1}) \one_{[k\eta, (k+1)\eta)}$ (which is constant on intervals $[k\eta, (k+1)\eta)$ for $\eta$ sufficiently small, and depends on $B_{[0, T)}$ only through the finite collection of increments $(B_{[k\eta, (k+1)\eta)})_{k \in [T/\eta]}$), then it becomes simple to write a change of measure so long as~\eqref{eq:transformation-generic} is invertible. The techniques of Malliavin calculus allow us to discuss the validity of this limit as we improve the granularity of the approximation. Given $\mathbf Q, \mathbf P$ which correspond to the laws of the ``paths'' $X_{[0, T)}$ following~\eqref{eq:langevin_diff} and a suitable interpolation of~\eqref{eq:rmp-std}, we can use this result to obtain bounds on the $\KL(\mathbf P \mmid \mathbf Q)$, or $\Renyi_q$ as desired.

The key to this computation lies in bounding the determinant of the additional quantity $\MD_s \Mu_t$, which for elementary processes is simply the matrix of derivatives of $f_k$ with respect to the $j$-th Brownian increment. To do this, we take the following procedure: (1) approximate all continuous processes with a suitably fine, piecewise constant interpolation; (2) expand the log-determinant in terms of traces of matrix powers using Lemma~\ref{lem:log-det}, resolving only the lowest-order terms. These contribute negligible factors to our bounds. We make these notions precise in the following sections.

In isolation, these Girsanov bounds can lead to some guarantees, by combining with the weak triangle inequality (Lemma~\ref{lem:renyi-triangle}) and the convergence of~\eqref{eq:langevin_diff} or~\eqref{eq:underdamped_langevin_diff} in R\'enyi divergence (\S\ref{scn:cvg}). However, our bounds work best when complimenting the results in~\cite{scr3, scr4}, which can witness the weak error of the various numeric schemes but requires an additional \emph{cross-regularity} ingredient which cannot be proved using their framework alone. We go into more detail in the subsequent section.

\subsection{Contributions}
We establish direct $\KL$ and $\Renyi_q$ bounds between various midpoint interpolations. We summarize our primary contributions as follows.
\begin{itemize}
    \item We state a variant of Girsanov's theorem for anticipating processes (Theorem~\ref{thm:anticipating-girsanov}), and show how the quantities within this theorem can be effectively computed for discretizations of interest. While this result is not new, it has not thus far been effectively utilized within the sampling community, and we hope that our exposition will serve as an interpretable blueprint for future works aiming to analyze similar methods.

    \item We introduce a deterministic double midpoint method for simulating~\eqref{eq:ULD} called~\eqref{eq:DM--ULMC-multistep}, new to this work. In our analysis, we show that under a strong Hessian smoothness assumption (Assumption~\ref{as:chaos-tail}), we can obtain $\KL(\bs \mu_{Nh} \mmid \bs \pi) \leq \varepsilon^2$ for some computationally tractable measure $\bs \mu_{Nh}$, with no more than $\widetilde{O}\Bigl(\frac{d^{1/4}}{\varepsilon^{1/2}}\Bigr)$ queries. See Theorem~\ref{thm:dmd-final} for a more precise statement. This is the first such result in $\KL$ divergence.

    \item We apply Theorem~\ref{thm:anticipating-girsanov} for several midpoint methods used in sampling. Namely, we compute the Radon--Nikodym derivative between $\mathbf P, \mathbf Q$ for the single midpoint method for the overdamped Langevin~\eqref{eq:M-LMC}, as well as the double midpoint method for the underdamped Langevin~\eqref{eq:DM--ULMC-multistep}. These can be used to derive cross-regularity for these schemes, as we discuss in \S\ref{scn:discussion}. 
    
    In~\cite[Lemma 6.2]{scr3}, cross-regularity was proven for~\eqref{eq:M-LMC} in an ad-hoc way, leading to an extra factor of $\log \frac{1}{\beta h}$. On the other hand,~\cite{scr4} was unable to prove sufficiently strong cross-regularity bounds for~\eqref{eq:DM--ULMC-multistep}, and so could only state guarantees if the exponential Euler discretization was used in the final step. We dispense with both of these caveats in this work through Corollary~\ref{cor:ormd-error} and Lemma~\ref{lem:renyi-dmd-coarse} respectively. We also note that our bounds hold in $q$-R\'enyi divergence.
    
    \item For~\eqref{eq:DM--ULMC-multistep}, the approaches used in~\cite{scr3, scr4} are insufficient for a $\Otilde(d^{1/4})$ query complexity bound. We emphasize that the present work is neither derivative of nor detracts from the results in~\cite{scr4}. Instead, we see our approach as complementary; we can provide genuine $\KL$ bounds on the process laws, which can be combined with the weak error through the framework of~\cite{scr4}. \textbf{The combination of both approaches} is necessary to obtain our final guarantees.
\end{itemize}

\subsection{Organization}

In \S\ref{scn:prelim}, we provide some general preliminary notions. In \S\ref{scn:malliavin}, we analyze a simple algorithm~\eqref{eq:m-lmc}, and provide an introductory treatment of Malliavin calculus. We also state the anticipating Girsanov theorem in Theorem~\ref{thm:anticipating-girsanov}, which is the primary tool of this paper. \S\ref{scn:applications} then provides the remaining detailed analysis of several common midpoint algorithms, deriving bounds in $\KL$ and $q$-R\'enyi divergences between $\mathbf P$, $\mathbf Q$. Finally, \S\ref{scn:discussion} harmonizes these approaches with the local error framework~\cite{scr3, scr4}, and shows how the novel $\Otilde(d^{1/4} \varepsilon^{-1/2})$ query complexity guarantee can be obtained under a strong smoothness assumption.

\section{Preliminaries}\label{scn:prelim}
\paragraph{Notation.}
In the sequel, we will consider the triple $(\Omega, \ms F, \mathbf P)$, induced by a Wiener measure $\mathbf P$ (most naturally thought of as the measure which induces the ``path law'' of Brownian motions). As a result, we will let $L^2_\mathbf P(\R^d)$ be the set of functions $f:\Omega \to \R^d$ such that $\E_{\mathbf P}\norm{f}^2 < \infty$. As a special case, $L^2_{\mathbf P}(\R)$ corresponds to real-valued functions $f: \Omega \to \R$ such that $\E_{\mathbf P} \abs{f}^2 < \infty$.

Likewise, we will use $L^2([0, T); \R^d)$ to denote functions $\Mh: [0, T) \to \R^d$ such that $\int_0^T \norm{\Mh(t)}^2 \, \D t < \infty$. $(B_t)_{t \in [0, T)}$ will be the standard $\mathbf P$-Brownian motion on $\R^d$; where necessary, we will denote $(\ms F_t)_{t \in [0, T)}$ as its corresponding filtration. Almost surely, $(B_t)_{t \in [0, T)}$ takes values on $C_0([0,T]; \R^d)$, which denotes the set of random processes $(\Mu_t)_{t \in [0, T)}: [0, T) \to \R^d$ with $\Mu_0 = 0$ almost surely.
$L^2_{\mathbf P}([0, T); \R^d)$ will denote random processes $(\Mu_t)_{t \in [0, T)}$ with $\E[\norm{\Mu}_{\Hb}^2] < \infty$. A fortiori, $\Mu_t \in L^2_{\mathbf P}(\R^d)$ for almost every $t \in [0, T)$.

We say that a process $(\Mu_t)_{t \in [0, T)}$ is $\mc F$-adapted if $\Mu_t$ is measurable with respect to $\ms F_t$ (the filtration corresponding to the $\mathbf P$-Brownian motion) for all $t \in [0, T)$. 
For some simple examples of processes which are not adapted, one can consider the constant-in-time functions $t \mapsto B_\tau$ for $\tau \in (0, T]$. 

If a matrix $A \in \R^{(d_1 + d_2)\times (d_1 + d_2)}$ can be partitioned into blocks $\R^{d_1 \times d_1}, \R^{d_1 \times d_2}, \R^{d_2 \times d_1}, \R^{d_2 \times d_2}$, we refer to these as $A_{[1,1]}$, $A_{[1,2]}$, etc.\ respectively. This is not to be confused with the notation $A_{i,j}$, which indicates the $(i,j)$-th entry of $A$. For instance, $(A_{[2,2]})_{i,j} = A_{d_1+i, d_1+j}$. We use the notation $\lesssim, \gtrsim , \asymp$ to denote bounds up to an implied absolute constant, and these are equivalent to $O(\cdot), \Omega(\cdot), \Theta(\cdot)$. $a = \Otilde(b)$ denotes $a = O(b \operatorname{polylog} b)$, and likewise for $\widetilde \Theta, \widetilde \Omega$.

\paragraph{General notions}

To measure closeness between measures on arbitrary domains, we introduce the following information divergences and metrics between probability measures.
\begin{definition}
    Let $\mu \ll \nu$ be two distributions on $\R^d$. Then, we can define the Kullback--Leibler ($\KL$) divergence between then as
    \begin{align*}
        \KL(\mu \mmid \nu) = \E_\mu \log \frac{\D \mu}{\D \nu}\,.
    \end{align*}
    Likewise, we can define the $q$-R\'enyi divergence for $q > 1$ as
    \begin{align*}
        \Renyi_q(\mu \mmid \nu) = \frac{1}{q-1} \log \int \Big\lvert\frac{\D \mu}{\D \nu}\Big\rvert^q \, \D \nu\,, 
    \end{align*}
    where these definitions are to be interpreted as equalling $+\infty$ when $\mu \not\ll \nu$. 
    The above generalizes to two measures on an infinite dimensional space such as $\Omega$, if we consider $\frac{\D \mu}{\D \nu}$ as the Radon--Nikodym derivative of $\mu$ with respect to $\nu$. We note that the $q$-R\'enyi divergence is monotonically increasing in $q$ and recovers the $\KL$-divergence when $q \to 1$. $\KL$ divergence is also sufficient to upper bound the squared total variation distance via Pinsker's inequality, with
    \begin{align*}
        \KL(\mu \mmid \nu) \geq 2 \norm{\mu - \nu}_{\TV}^2 \deq 2\sup_{B \in \ms F} \abs{\mu(B) - \nu(B)}^2\,.
    \end{align*}
    Finally, we introduce the Wasserstein distance as, for $\Gamma(\mu, \nu) \deq \{\rho \in \mc P(\R^{2d}): \int_{\R^d} \rho(x, \D y) = \mu(x), \int_{\R^d} \rho(\D x, y) = \nu(y)\}$ the set of couplings with marginals $\mu, \nu$ respectively,
    \begin{align*}
        W_2^2(\mu, \nu) \deq \inf_{\rho \in \Gamma(\mu, \nu)} \E_{(x,y) \sim \rho}\Bigl[\norm{x-y}^2 \Bigr]\,.
    \end{align*}
\end{definition}

It is known that exponential convergence of~\eqref{eq:langevin_diff} in the divergences mentioned above can be obtained if we assume that the target measure satisfies a relevant functional inequality.
\begin{definition}[Log-Sobolev inequality]\label{def:lsi}
    A measure $\pi$ on $\R^d$ satisfies a log-Sobolev inequality with parameter $1/\alpha$ if for all smooth and compactly supported functions $f: \R^d \to \R_+$, we have
    \begin{align}\label{eq:lsi}\tag{LSI}
        \ent_\pi f \deq \E_\pi \Bigl[f \log\frac{f}{\E_\pi f}\Bigr] \leq \frac{1}{2\alpha} \E_\pi [\norm{\nabla f}^2]\,.
    \end{align}
\end{definition}
Notationally, for~\eqref{eq:underdamped_langevin_diff}, consider $\bs \pi(x,p) \propto \exp(-V(x) - \frac{\norm{p}^2}{2})$ and an initial measure $\bs \mu_0 \in \mc P(\R^{2d})$. We will consider the following triplet of processes: $(\bs \mu_t, \tilde{\bs \mu}_t, \bs \pi)_{t \in [0, T)}$, which correspond respectively to the evolutions of $\bs \mu_0$ along the algorithm, $\bs \mu_0$ along~\eqref{eq:underdamped_langevin_diff}, and $\bs \pi$ under~\eqref{eq:underdamped_langevin_diff} (noting that $\bs \pi$ is stationary). Similarly, we let $(\mu_t, \tilde \mu_t, \pi)$ be the analogous versions along~\eqref{eq:langevin_diff}.

\section{Malliavin calculus}\label{scn:malliavin}
\subsection{Warm-up: overdamped midpoint discretization}\label{scn:warmup-m-lmc}
We now proceed with a brief warm-up. Before going into details, we state an important assumption which will be used throughout our work.
\begin{assumption}[Regularity]\label{as:regularity}
    Let $V: \R^d \to \R$ be twice continuously differentiable with $-\beta I_d \preceq \nabla^2 V \preceq \beta I_d$. Moreover, let $\nabla^2 V$ be uniformly continuous.
\end{assumption}
\begin{remark}
The uniform continuity assumption is only needed for convergence of the certain approximation schemes when obtaining expressions for the Radon--Nikodym, and does not feature quantitatively in any of our results on $\KL$ or $\Renyi_q$. As a result, it can be effectively ignored using the standard technique of weakly convergent approximation combined with lower semi-continuity of information divergences.
\end{remark}

We consider the overdamped midpoint discretization, which was introduced at the start of this paper. Here, we do not consider a random interpolant point; instead, suppose that the choice of interpolation times $\tau$ is fixed and inhabiting $[0, h)$. 
\begin{align*}
    X_t = X_0 - \int_0^t \nabla V(X_s) \, \D s + \sqrt{2}\tilde B_t\,,
\end{align*}
and
\begin{align}\label{eq:m-lmc}\tag{M--LMC}
    X_t = X_0 - \int_0^t \nabla V(X^+) \, \D s + \sqrt{2} B_t\,,
\end{align}
where $(B_t)_{t\in [0, h)}, (\tilde B_t)_{t \in [0, h)}$ are $\mathbf P, \mathbf Q$ Brownian motions respectively, and we define the interpolant
\begin{align*}
    X^+ = X_0 - \tau \nabla V(X_0) + \sqrt{2} B_\tau\,.
\end{align*}
The reason behind this formalism is that given $X_0$, we can essentially think of $X_t$ as deterministic functions of either $(B_t)_{t \in [0, h)}$ or $(\tilde B_t)_{t \in [0, h)}$, and if $(B_t)_{t \in [0, h)}$ has the distribution of a Brownian motion, then $X_h$ will now exactly have the marginal law of a single iteration of the midpoint discretization (see~\cite{shen2019randomized}). On the other hand, if $(\tilde B_t)_{t \in [0, h)}$ has a Brownian law, then $X_h$ will instead be the marginal distribution of~\eqref{eq:langevin_diff}. As a result, if we can compute the Radon--Nikodym and $\KL$ divergence of the laws of $(B_t)_{t \in [0, h)}, (\tilde B_t)_{t \in [0, h)}$, then we can immediately upper bound the $\KL$ divergence between the marginals of the two sampling processes.

\paragraph{Part (A): Discretization.}
To formulate this change of measure principle, it is instructive to first introduce a coarse discretization. Consider two sequences of standard normal random variables $\bs \xi = [\xi_1, \dotsc, \xi_{h/\eta}]$, $\tilde{\bs \xi} = [\tilde \xi_1, \dotsc, \tilde \xi_{h/\eta}]$ for $\eta \leq h$ evenly dividing $h$, each inhabiting $\R^d$, with $\sqrt{\eta} \xi_i$, $\sqrt{\eta} \tilde \xi_i$ representing the $i$-th Brownian increment $B_{i\eta} - B_{(i-1)\eta}$ and $\tilde B_{i \eta} - \tilde B_{(i-1)\eta}$ respectively. Now, define the sequence of marginals by, for $k \in [h/\eta]$
\begin{align*}
    \hat X_{k\eta}^\eta &= \hat X_{(k-1)\eta}^\eta - \eta \nabla V(\hat X_{(k-1)\eta}^\eta) + \sqrt{2\eta} \tilde \xi_k\,, \\
    \hat X_{k\eta}^\eta &= \hat X_{(k-1)\eta}^\eta - \eta \nabla V(X^+) + \sqrt{2\eta} \xi_k\,,
\end{align*}
and as $\eta$ will be fixed for the moment, we omit it from the superscript. The definition of $X^+$ does not change so long as $\eta$ also subdivides $\tau$, which we can assume without much loss of generality, and is given by
\begin{align*}
    X^+ = X_0 - \tau \nabla V(X_0) + \sqrt{2 \eta} \sum_{i=1}^{\tau/\eta} \xi_k\,.
\end{align*}
We call the processes above ``elementary'' as they depend only on a finite number of Brownian increments. Now, consider the sequence
\begin{align}\label{eq:transform-m-lmc}
    \tilde \xi_i = \xi_i + \sqrt{\frac{\eta}{2}} \Bigl(\nabla V(\hat X_{(i-1)\eta}) - \nabla V(X^+) \Bigr)\,,
\end{align}
and define $\msf T(\xi_1, \dotsc, \xi_{h/\eta}) = (\tilde \xi_1, \dotsc, \tilde \xi_{h/\eta})$. 
Then, (noting that both $X^+$, $\hat X^\eta_{k\eta}$ are explicit functions of $(\xi_i)_{i \in [h/\eta]}$), this gives a deterministic function which defines each $\tilde \xi_i$ in terms of the entire sequence $(\xi_i)_{i \in [h/\eta]}$.

Now, let us compute the Radon--Nikodym between the two sequences $(\xi_i)_{i \in [h/\eta]}$, $(\tilde \xi_i)_{i \in [h/\eta]}$. Define the functions
\begin{align}\label{eq:psi-ld}
    \Mu_i(\xi_1, \dotsc, \xi_{h/\eta}) =  \sqrt{\frac{\eta}{2}} \Bigl(\nabla V(\hat X_{(i-1)\eta}) - \nabla V(X^+) \Bigr)\,.
\end{align}
Consider the ``matrix'' $\MD \Mu$ which has entries $(\MD \Mu)_{i,j} = \partial_{\xi_j} \Mu_i$. To disambiguate the gradient in $\xi$ from the standard gradient e.g., in $\nabla V$, we write $\MD$ to denote the former. Note that each entry of $\MD \Mu$ is really matrix-valued, since both $\xi_j$ and $\Mu_i$ are vectors in $\R^d$, so that $\MD \Mu \in \R^{d \times d \times (h/\eta) \times (h/\eta)}$; however, we will typically view it as a matrix of size $(h/\eta) \times (h/\eta)$ (temporal indices) with entries $d \times d$ (spatial indices). 

\paragraph{Change of measure.} 
Now, define the random factor
\begin{align*}
    \mbf M \deq \abs{\det(I + \MD \Mu)} \exp\Bigl(-\sum_{i=1}^{h/\eta} \inner{\Mu_i, \xi_i} - \frac{1}{2} \sum_{i=1}^{h/\eta} \nnorm{\Mu_i}^2 \Bigr)\,,
\end{align*}
omitting the dependence of $\Mu_i$ on $\bs \xi$. Here, let $\nnorm{\Mu}^2$ be the $2$-norm (Hilbert--Schmidt norm) on $\R^{d \times (h/\eta)}$. As shall be seen later, it is instead more convenient to consider the quantities:
\begin{align*}
    \mbf M \deq \Bigl\{\abs{\det(I_d + \MD \Mu)} \exp(-\tr \MD \Mu)\Bigr\} \exp\Bigl(-\bigl(\sum_{i=1}^{h/\eta} \inner{\Mu_i, \xi_i} - \tr \MD \Mu\bigr) - \frac{1}{2} \sum_{i=1}^{h/\eta} \nnorm{\Mu_i}^2 \Bigr)\,,
\end{align*}
Here, the trace is both temporal and spatial, i.e., $\tr \MD \Mu = \sum_{i=1}^{h/\eta} \tr((\MD \Mu)_{i,i})$. $\sum_{i \in [h/\eta]}\inner{\Mu_i, \xi_i} - \tr \MD \Mu$ has the pleasing interpretation that it is mean zero under $\mathbf P$, via a Gaussian integration by parts formula. Now, consider a non-negative random variable of the form $Z = f(\bs \xi)$, and $f \in C_b^\infty(\R^{h/\eta}; \R_+)$. We have another random variable
\begin{align}\label{eq:transform-finite}
    Z \circ \msf T = f(\xi_1 + \Mu_1,\, \dotsc,\, \xi_{h/\eta} + \Mu_{h/\eta})\,.
\end{align}
Now we have
\begin{align*}
    \E_{\mathbf P}\Bigl[(Z \circ \msf T) \cdot \mathbf M \Bigr] &= \E_{\mathbf P}\biggl[\int_{\R^{h/\eta}} \abs{\det(I_d + \MD \Mu(\mathbf x))} \exp\Bigl(-\frac{1}{2}\sum_{j=1}^{h/\eta} (x_j + \Mu_j(\mathbf x))^2 \Bigr) \\
    &\qquad\quad \cdot f\Bigl(x_1 + \Mu_1(\mathbf x), \dotsc, x_{h/\eta} + \Mu_1(\mathbf x) \Bigr) \cdot (2\pi)^{-\frac{h/\eta}{2}} \, \D \mathbf x\biggr]\,,
\end{align*}
with $\mathbf x = (x_1, \dotsc, x_{h/\eta})$. Here, the extra factor of $-\frac{1}{2} \sum_{j=1}^{h/\eta} x_j^2$ in the exponent arises from $\bs \xi$ being Gaussian under $\mathbf P$. Applying the change-of-variables formula and as $\msf T$ is one-to-one,
\begin{align*}
    \E_{\mathbf P}\Bigl[(Z \circ \msf T) \mbf M \Bigr] &= \E_{\mathbf P}\biggl[\int_{\operatorname{Im}_{\operatorname{id} + \Mu}} \exp\Bigl(-\frac{1}{2} \norm{\mathbf y}^2 \Bigr) f(y_1, \dotsc, y_{h/\eta}) \cdot (2\pi)^{-\frac{h/\eta}{2}} \, \D \mathbf y\biggr]\,,
\end{align*}
Here, $\operatorname{Im}(\operatorname{id} + \Mu)$ is the image of $\R^{h/\eta}$ under the mapping $\mathbf x \mapsto \mathbf x + \Mu(\mathbf x)$. Replacing this with the integral over the whole space gives
\begin{align}\label{eq:change-of-measure-guaranteed}
     \E_{\mathbf P}\Bigl[(Z \circ \msf T) \cdot \mbf M \Bigr] \leq \E_{\mathbf P}[Z]\,,
\end{align}
with equality when $\operatorname{id} + \Mu$ is onto, e.g.\ when $\msf T$ is invertible (rather than merely locally invertible).

\paragraph{Part (B): Computation of relevant quantities.} We note that $\MD \nabla V(X^+)$ is
\begin{align*}
    \MD_i \{\nabla V(X^+)\} = \nabla^2 V(X^+) \cdot \sqrt{2\eta} \one_{i \leq \tau/\eta}\,. 
\end{align*}
On the other hand,
\begin{align}\label{eq:sde-mall-deriv}
\begin{aligned}
    \MD_i \nabla V(\hat X_{k\eta}) &= \nabla^2 V(\hat X_{k\eta}) \cdot \MD_i \Bigl(X_0 - k\eta \nabla V(X^+) + \sqrt{2\eta} \sum_{j=1}^k \xi_k \Bigr) \\
    &= \nabla^2 V(\hat X_{k\eta}) \cdot \Bigl(k\eta \nabla^2 V(X^+) \cdot \sqrt{2\eta} \one_{i \leq \tau/\eta} + \sqrt{2\eta} I_d \one_{i \leq k}\Bigr)\,,
\end{aligned}
\end{align}
using the formula for $\MD_i (\nabla V(X^+))$. Now, we can compute the determinant as needed. For later convenience, we will consider the determinant involving $q \Mu$ for some $q > 0$. The coefficients can be written in the following form:
\begin{align*}
    I + q\MD \Mu = \begin{bmatrix}
        I + (q \MD \Mu)_{[1,1]} & 0 \\
        (q \MD \Mu)_{[2,1]} & I + (q \MD \Mu)_{[2,2]}
    \end{bmatrix}\,,
\end{align*}
where the block $(q \MD \Mu)_{[1,1]}$ has size $(\tau/\eta) \times (\tau/\eta)$ (with each entry a block matrix in $\R^{d \times d}$, so that $\MD \Mu$ is really a $4$-tensor), and $I$ is an identity matrix of appropriate size.  As $(q \MD \Mu)_{[2,2]}$ involves only the components that are strictly below the diagonal, we have $\det(I+(q \MD \Mu)_{[2,2]}) = 1$. We can also neglect $(q\MD \Mu)_{[2,1]}$ as the matrix is block upper triangular. It remains to consider the determinant of the top left.

By abuse of notation, all our subsequent calculations will be \textbf{restricted to coordinates in this block}. We will continue to use the same notation for our matrices, noting that these are really the restrictions of the matrices to the $[1, 1]$ block. Now, we make some simplifications. First, we have the expansion
\begin{align*}
     \norm{\sqrt{\eta}\, \MD (\nabla V(\hat X))} \lesssim \norm{\msf A} + O(\beta^2 h^2)\,, \qquad \msf A_{i,j} = \sqrt{2} \eta \nabla^2 V(\hat X_{(j-1)\eta}) \one_{i \leq j-1}\,.
\end{align*}
The reason for the offset $j-1$ in the index arises from our definition of $\Mu_i$ in~\eqref{eq:psi-ld}.
We note that the terms arising from differentiating $\nabla V(X^+)$ within $\nabla V(\hat X)$ are of higher order and so can be ignored, which can be seen from~\eqref{eq:sde-mall-deriv} (multiplying both sides by $\sqrt{\eta}$), wherein the first term has eigenvalues of order $O(\beta^2 h^2)$, whereas the second gives rise to $\msf A$ and has eigenvalues of order $O(\beta h)$. Define $\bbv = [1, \dotsc, 1] \in \R^{h/\eta}$. Appealing to Lemma~\ref{lem:log-det}, we have for $h \lesssim \frac{1}{\beta q}$,
\begin{align*}
    &\log\abs{\det (I + q \MD \Mu)} - \tr\bigl(q\MD \Mu)\bigr)  \\
    &\qquad = - \frac{q^2}{2} \tr \Bigl((\MD \{\nabla V(X^+)\} \bbv^\top - \msf A )^2 \Bigr) + O(\beta^3 d h^3 q^3) \\
    &\qquad = -\frac{q^2}{2}\tr\Bigl(\msf A^2 - 2 \MD \{\nabla V(X^+)\} \bbv^\top  \cdot \msf A + (\MD \{\nabla V(X^+) \bbv^\top)^2\Bigr) + O(\beta^3 dh^3 q^3)\,.
\end{align*}

\paragraph{Validity of the discretization.} Finally, we can verify the invertibility from our Hessian computations involving $\msf T$. We can see that $q \MD \Mu$ has operator norm bounded by $O(\beta h q)$. So long as $h \lesssim \frac{1}{\beta q}$ for a sufficiently small implied constant, then invertibility follows if we can show properness of the transformation $\msf T$. Properness of $\msf T: \bs \xi \mapsto \bs \xi + q\Mu$ follows as, from~\eqref{eq:transform-m-lmc} and Lipschitzness of $\nabla V$, 
\begin{align}\label{eq:aligned}
\begin{aligned}
    \nnorm{\tilde{\bs \xi}}^2 &\geq \nnorm{\bs \xi}^2 - q \cdot  O\Bigl(\eta \nnorm{\nabla V(\hat X)}^2 + h\norm{\nabla V(X^+)}^2 \Bigr) \\
    &\geq \nnorm{\bs \xi}^2 -q^2 \cdot  O\Bigl(h \norm{\nabla V(X_0)}^2 +(h+\beta^2 h^3)\norm{\nabla V(X^+)}^2 + \beta^2 h^2\nnorm{\bs \xi}^2\Bigr)\\
    &\geq \nnorm{\bs \xi}^2 - q^2 \cdot O\Bigl((h+\beta^2 h^3 + \beta^4 h^5) \norm{\nabla V(X_0)}^2 + (\beta^2 h^2 + \beta^4 h^4)\nnorm{\bs \xi}^2\Bigr)\,.
\end{aligned}
\end{align}
and thus taking $\nnorm{\bs \xi} \to \infty$, we can ensure that $\nnorm{\tilde{\bs \xi}} \to \infty$ for $h \lesssim \frac{1}{\beta q}$. This implies~\eqref{eq:change-of-measure-guaranteed} is indeed an equality, and the factor $\mathbf M$ is a true Radon--Nikodym derivative.

\subsection{Details of Malliavin calculus}

It turns out that the computations of the preceding section can be considerably generalized, as the relevant quantities are stable in the limit as we take $\eta \to 0$. As we consider now arbitrary endpoints, we consider processes defined on $[0, T)$.\footnote{We exclude the right endpoint of $[0, T)$ for simplicity, although it can be included via a slight (but tedious) modification of the notation.}

In particular, we introduce the following definitions. Use the familiar shorthand $\xi_i^\eta = (B_{i\eta} - B_{(i-1)\eta})$, $\bs \xi^\eta = [\xi_1^\eta, \dotsc, \xi_{T/\eta}^\eta]$. Now, we define the Malliavin derivative of functions of $\bs \xi^\eta$.
\begin{definition}[Malliavin derivative]\label{defn:malliavin}
    The \emph{Malliavin derivatives} of functions of $\bs \xi^\eta \in \R^{d \times T/\eta}$, $\mbf F: \R^{d \times T/\eta} \to \R^d$ can be defined component-wise as
    \begin{align*}
         \MD_i \mbf F = \partial_i \mbf F(\bs \xi^\eta)\,.
    \end{align*}
\end{definition}

\begin{remark}
    Intuitively, functions $\mbf F$ can be considered a function of the entire Brownian path. Due to technical subtleties, we would prefer that these functions depend only on a finite number of nice statistics relating to the path; for instance, on a finite number of It\^o integrals against nice functions. Then, the Malliavin derivative gives a way of differentiating the function $\mbf F$ with respect to the underlying Brownian motion, such that $\mbf F$ satisfies nice continuity properties with respect to the topology of $\Omega$ (the product topology). The monographs~\cite{nualart2006malliavin, eldredge2016analysis, hairer2021introduction} contain more detailed treatments of these issues. Furthermore, the Malliavin derivative is typically not defined as a linear form, but rather as a process on $L^2_{\mathbf P}([0, T]; \R^d)$. This definition is consistent with the definition above up to some factors of $\eta$; while our definition bears greater resemblance to elementary calculus, it is somewhat less elegant when taking the limit $\eta \searrow 0$.
\end{remark}

\begin{definition}[Skorohod adjoint operator]\label{defn:skorohod}
We denote by $\Mdel$ the Skorohod adjoint operator to the Malliavin derivative. For any function $\mbf F: \R^{d \times T/\eta} \to \R^d$ and $\Mu \in L^2_{\mbf P}([0, T]; \R^d)$ a random process which is piecewise constant on each of the intervals $[(i-1)\eta, i\eta)$, $i \in [T/\eta]$, we have
\begin{align*}
    \E_{\mathbf P}[\mbf F \Mdel \Mu] = \E_{\mathbf P}\Bigl[\sum_{i=1}^{T/\tau} \inner{\MD_i \mbf F, \Mu_i} \Bigr]\,,
\end{align*}
and one can verify that this gives
\begin{align*}
    \Mdel \Mu = \sum_{i=1}^{T/\eta} \inner{\Mu_i, \xi_i} - \tr \MD \Mu\,.
\end{align*}
\end{definition}
\begin{remark}
This is sometimes referred to as the ``Skorohod integral'' of the element $\Mu$, since (i) it agrees with the It\^o integral when $\Mu$ is adapted to the filtration of the standard Brownian motion, (ii) aesthetically, it has the effect of ``integrating the Malliavin derivative by parts''. However, it is not clear how much this represents a bona fide stochastic integral in the non-adapted case.
\end{remark}

The following definition will be useful for defining determinants of kernels over general spaces, which appears in the anticipating Girsanov theorem in a vital way. 
\begin{definition}[Carleman--Fredholm determinant]
    For finite kernels $K = \R^{d \times d \times (T/\eta) \times (T/\eta)}$, the \emph{regularized Carleman--Fredholm determinant} is defined as
    \begin{align*}
        \Mdet(I + K) \deq \det(I+K) \exp(-\tr K)\,.
    \end{align*}
\end{definition}

Now, we can accurately state the anticipating Girsanov theorem. In the formulation of~\eqref{eq:transformation-generic}, we have $\msf T: B_{[0, T]} \mapsto B_{[0, T]} + (\int_0^\cdot \bar \Mu_s \, \D s)_{[0, T]}$. If we write in terms of standard Gaussians, and suppose that $\bar \Mu$ is piecewise constant over each interval the intervals $[(k-1)\eta, k\eta)$, $k \in [T/\eta]$, then this coincides with our definition in~\eqref{eq:transform-finite} if, for $t \in [(k-1)\eta, k\eta)$,
\begin{align}\label{eq:interpolant-drift}
    \bar \Mu_{t}^\eta = \frac{1}{\sqrt{\eta}} \Mu_k^\eta\,.
\end{align}
Then, we write, using the same notation 
\begin{align*}
    \log \Mdet(I + \MD \bar \Mu^\eta) \deq \log \Mdet(I + \MD \Mu^\eta)\,, \qquad \Mdel \bar \Mu^\eta \deq \Mdel \Mu\,.
\end{align*}
Furthermore, for processes which are not piecewise constant, we can define
\begin{align*}
    \log \Mdet(I + \MD \bar \Mu) \deq \lim_{\eta \searrow 0} \log \Mdet(I + \MD \bar \Mu^\eta)\,, \qquad \Mdel \bar \Mu \deq \lim_{\eta \searrow 0} \Mdel \bar \Mu^\eta\,.
\end{align*}
for some approximation $\bar \Mu^\eta$ to $\bar \Mu$. Note that $I + \MD \bar \Mu$ is not itself a well-defined object in our notation, although the $\log \Mdet$ is well-defined from the expressions above. It turns out~\cite{nualart2006malliavin} that these definitions are invariant with respect to the specific choice of approximation, as long as the limiting quantities are well-defined for the sequence of piecewise constant approximations, and the convergence $\bar \Mu^\eta \to \bar \Mu$ holds in $L^2_{\mbf P}([0, T]; \R^d)$.

Under this convention, it turns out that the preceding computations give rise to a generalization of Girsanov's theorem.
\begin{theorem}[{Anticipating Girsanov theorem, adapted from \cite[Section 4.1.4]{nualart2006malliavin}}]\label{thm:anticipating-girsanov}
    Suppose that in~\eqref{eq:transformation-generic}, $\msf T$ is bijective. Then, let $\mathbf Q$ be the measure under which $\msf T((B_t)_{t \in [0, T)})$ is Brownian, and $\mathbf P$ the measure under which $(B_t)_{t \in [0, T)}$ is Brownian. The Radon--Nikodym derivative is given by
    \begin{align*}
        \frac{\D \mathbf Q}{\D \mathbf P}(\omega) = \abs{\Mdet(I+\MD \bar \Mu(\omega))} \exp\Bigl(-(\Mdel \bar \Mu)(\omega) - \frac{1}{2} \int_0^T \norm{\bar \Mu_t(\omega)}^2 \, \D t\Bigr)\,,
    \end{align*}
    where $\bar \Mu(\omega)$ refers to the value of $\bar \Mu = (\bar \Mu_t)_{t \in [0, T)}$ given $(B_t)_{t \in [0, T)} = \omega \in C_0([0, T); \R^d)$, where $\bar \Mu(\omega)$ is deterministic by assumption. $(\MD \bar \Mu)(\omega),\ \Mdel \bar \Mu(\omega)$ are similarly interpreted.
\end{theorem}
This lemma was already proven for elementary processes in the preceding section. Whereas the extension to arbitrary processes can be formally done using some heavy analytic machinery~\cite{nualart2006malliavin}, this will not be necessary for computing $\Renyi_q$ between the processes, by using the weak convergence of the elementary approximations as well as the lower semicontinuity of $\Renyi_q$ in both arguments. See~\cite[Section 4.1]{nualart2006malliavin} for a complete proof.

For reference, we provide the original Girsanov theorem below. It will be easier to utilize than the anticipating Girsanov equation when dealing with transformations given by an SDE, i.e., without any anticipation in the drift.
\begin{theorem}[{Non-anticipating Girsanov; adapted from~\cite[Theorem 8.6.8]{oksendal2013stochastic}}]\label{thm:regular-girsanov}
    Suppose that in~\eqref{eq:transformation-generic}, $\Mu_s \in L^2_{\mathbf P, \operatorname{ad}}([0, T]; \R^d)$ is adapted. Then, if $\mathbf Q, \mathbf P$ are again the measures induced by $\msf T$ and the Wiener measure respectively as in Theorem~\ref{thm:anticipating-girsanov}, the Radon--Nikodym derivative is given by
    \begin{align*}
        \frac{\D \mathbf Q}{\D \mathbf P}(\omega) = \exp\Bigl( -\int_0^T \inner{\bar \Mu_t, \D B_t}(\omega) - \frac{1}{2} \int_0^T \norm{\bar \Mu_t(\omega)}^2 \, \D t \Bigr)\,,
    \end{align*}
    where again $\int_0^T \inner{\bar \Mu_t, \D B_t}(\omega)$ denotes the It\^o integral of $\Mu_t$ (evaluated at the Brownian path realization $(B_t)_{t \in [0, T]} = \omega$). This requires Novikov's condition to be satisfied:
    \begin{align}\label{eq:novikov}
        \E_{\mathbf P}\Bigl[\exp\Bigl(\frac{1}{2} \int_0^T \norm{\bar \Mu_t(\omega)}^2 \, \D t\Bigr)\Bigr] < \infty\,.
    \end{align}
\end{theorem}
Theorem~\ref{thm:regular-girsanov} can be recovered as a special case of the anticipating Girsanov theorem. In this case, $\MD \bar \Mu$ is a quasi-nilpotent kernel (all its eigenvalues are zero), $\Mdet(I+ \MD \bar \Mu) = 1$. Secondly, the Skorohod adjoint reduces to the It\^o integral, which then gives the theorem as desired. 

\paragraph{User's guide to the anticipating Girsanov theorem}

The following procedure provides a step-by-step guide to the machinery developed in the sections above. Here, we suppose we want to compare two discretizations of the forms
\begin{align}\label{eq:equation-generic}
\begin{aligned}
    X_t &= X_0 - \int_0^t \msf G_s(X_{[0,T]}) \, \D s + \sqrt{2} \, \D \tilde B_t \\
    X_t &= X_0 - \int_0^t \msf F_s(X_{[0,T]}, B_{[0, T]}) \, \D s + \sqrt{2} \, \D B_t\,,
\end{aligned}
\end{align}
where for all time, $\msf F, \msf G$ depend on $X_{[0, T]}$ only through finitely many statistics of the form $\int_{t_0}^{t_1} f(t) X_t \, \D t$ for $t_0 < t_1$ and some function $f$, and likewise depend only on finitely many statistics of $B_{[0, T]}$. Furthermore, $\msf F, \msf G$ should be suitably continuous in time. Certain generalizations of this set-up may also be permitted, but we do not cover such generality in this text.
\begin{enumerate}[label=(\Alph*)]
    \item Discretize the two formulations above using piecewise constant schemes (which resemble the ``Euler--Maruyama'' schemes), defined for $t \in [0, T]$, with $\xi_k$ a sequence of i.i.d.\ standard $\mbf P$-Gaussians which correspond to the increments $B_{k\eta} - B_{(k-1)\eta}$, and with $\tilde \xi_k$ defined as
    \begin{align*}
        \hat X_{k\eta}^\eta &= \hat X_{(k-1)\eta}^\eta - \eta \msf G_{(k-1)\eta}(\hat X^\eta_{[0, T]}) + \sqrt{2 \eta} \tilde \xi_k\,, \\
        \hat X_{k\eta}^\eta &= \hat X_{(k-1)\eta}^\eta  - \eta \msf F_{(k-1)\eta}(\hat X^\eta_{[0, T]}, B_{[0, T]}) + \sqrt{2\eta} \xi_k\,.
    \end{align*}
    Note that the scheme above is possibly implicit, in that the vector $\hat {\bs X}^\eta \in \R^{d \times T/\eta}$ solves an equation like $\hat {\bs X}^\eta = \bs \psi(\hat {\bs X}^\eta, \bs \xi)$ for some non-linear function $\bs \psi: \R^{d \times T/\eta} \to \R^{d\times T/\eta}$. Where necessary, one needs to ensure that this implicit equation has a unique and well-defined solution for $\eta$ small enough, and that it converges to the solution of~\eqref{eq:equation-generic} in $L^2_{\mbf P}([0, T]; \R^d)$. These properties are verified using standard analysis techniques, which is usually best done at the end (as it requires computations from the successive parts).
    
    \item One then considers \begin{align*}
        \Mu^\eta_k = \sqrt{\frac{\eta}{2}} \Bigl(\msf G_{(k-1)\eta}(\hat X^\eta_{[0, T]}) - \msf F_{(k-1)\eta}(\hat X^\eta_{[0, T]}, B_{[0, T]}) \Bigr)\,,
    \end{align*}
    and computes the relevant Malliavin objects $\log\Mdet(I+ \MD \Mu^\eta)$, $\Mdel \Mu^\eta$ using standard finite-dimensional calculus.

    \item Then, assuming that the expressions $\log\Mdet(I+ \MD \Mu^\eta)$, $\Mdel \Mu^\eta$ remain almost surely finite in the limit as $\eta \searrow 0$, one formally assigns these limits as the values of $\log\Mdet(I+ \MD \bar \Mu)$, $\Mdel \bar \Mu$ respectively.

    \item Substitute these expressions into Theorem~\ref{thm:anticipating-girsanov} in order to obtain a simple expression for the Radon--Nikodym derivative. Use this to compute the $\KL$ divergence between path measures, or whatever other statistic that may be desired. Here, some extra work is usually needed to control $\norm{\bar \Mu}_{L^2_{\mbf P}([0, T]; \R^d)}$, which is somewhat orthogonal to the anticipating Girsanov machinery; instead, it is usually similar to existing discretization analyses in sampling, see~\cite{chewisamplingbook}.
\end{enumerate}
\begin{remark}
    Note that it does not matter precisely how one discretizes~\eqref{eq:equation-generic}, as long as the discretization scheme is well-defined and converges in $L^2_{\mbf P}([0, T]; \R^d)$.
\end{remark}

\section{Applications}\label{scn:applications}

\subsection{M--LMC}

We conclude our presentation of the analysis for~\eqref{eq:m-lmc} according to our blueprint above.

\paragraph{Part (C): Limiting quantities for M--LMC.}
Combining all our calculations yields the formula for the determinant, once we take $\eta \searrow 0$. First, we give the results for the traces of the various objects.
\begin{lemma}[Trace computations for~\eqref{eq:m-lmc}]\label{lem:trace-M-LMC}
    We have in the limit as $\eta \to 0$, if $\tau \ll \beta^{-1}$, and referring only to the $[1,1]$-blocks of the matrices as in Lemma~\ref{lem:MD-one-step},
    \begin{align*}
        \tr (\msf A^2) &\to 0 \\
        \tr \Bigl(\MD \{\nabla V(X^+) \bbv^\top \} \cdot {\msf A}\Bigr) &\to 2 \int_0^{\tau} t\tr\Bigl(\nabla^2 V(X^+) \nabla^2 V(X_t) \Bigr) \, \D t\\
        \tr \Bigl(({\MD \{\nabla V(X^+)\} \bbv})^2\Bigr) &\to 2 \tau\tr\Bigl((\nabla^2 V(X^+))^2\Bigr)\,.
    \end{align*}
\end{lemma}
\begin{proof}
    The first identity is obvious as $\msf A$ is strictly lower triangular in its temporal indices. As for the second, we have, since ${\MD \{\nabla V(X^+)\} \bbv^\top }$ is (temporally) rank-one,
    \begin{align*}
        \Bigl({\MD \{\nabla V(X^+)}\} \bbv^\top \cdot \msf A \Bigr)_{\ell,\ell} = 2\eta^2 (\ell - 1) \nabla^2 V(X^+) \nabla^2 V(\hat X_{\ell \eta})\,.
    \end{align*}
    Finally, ${\MD \{\nabla V(X^+)\} \bbv^\top }$ being temporally rank-one allows us to derive
    \begin{align*}
        (({\MD \{\nabla V(X^+)\}} \bbv)^2)_{\ell,\ell} = 2\eta^2 (\frac{\tau}{\eta} - 1) \nabla^2 V(X^+) \cdot  \nabla^2 V(X^+)\,.
    \end{align*}
    These easily give the identities in the limit, using the uniform continuity in Assumption~\ref{as:regularity}.
\end{proof}

Combining this with our earlier analysis gives the following one-step bound. First, from~\eqref{eq:m-lmc}, we obtain
\begin{align*}
    \bar \Mu = \frac{1}{\sqrt{2}} \Bigl(\nabla V(X_t) - \nabla V(X^+) \Bigr)\,.
\end{align*}
\begin{lemma}\label{lem:MD-one-step}
    Let $\Mu^\eta$ be defined as in~\eqref{eq:psi-ld} with $N = 1$. Under Assumption~\ref{as:regularity},
    \begin{align*}
        \Mdel \bar \Mu &= \frac{1}{\sqrt{2}}\int_0^h \inner{\nabla V(X_t), \,\D B_t}-\frac{1}{\sqrt{2}}\inner{\nabla V(X^+), B_h} + \frac{\tau}{\sqrt{2}} \tr(\nabla^2 V(X^+)) \\
        &\qquad + \int_0^{\tau} t \tr(\nabla^2 V(X_t)\nabla^2 V(X^+))\, \D t\,.
    \end{align*}
    Furthermore,
    \begin{align*}
        \log \Mdet(\msf I + q \MD \bar \Mu) &= 
        -\sqrt{2}q^2 \int_0^{\tau} t \tr(\nabla^2 V(X_t) \nabla^2 V(X^+))\,\D t - \frac{ \tau^2 q^2}{\sqrt{2}} \tr((\nabla^2 V(X^+))^2) \\
        &\qquad+ O(\beta^3 dh^3 q^3)\,.
    \end{align*}
    This implies in particular that if $h \lesssim \frac{1}{\beta q}$ is sufficiently small, then the transformation in~\eqref{eq:transformation-generic} will be invertible.
\end{lemma}

We now define the multistep version of our midpoint method for~\eqref{eq:langevin_diff}, given below. As usual, consider $\tilde B_{[0, T]}, B_{[0, T]}$ as $\mbf Q, \mbf P$ Brownian motions respectively. Then, write
\begin{align}
    X_t = X_{kh} - \int_{kh}^t \nabla V(X_s) \,\D s + \sqrt{2}(\tilde B_t - \tilde B_{kh})\,,
\end{align}
and
\begin{align}\label{eq:M-LMC}\tag{M-LMC}
    X_t = X_{kh} - \int_{kh}^t \nabla V(X_k^+)\, \D s + \sqrt 2(B_t - B_{kh})\,,
\end{align}
for the sequence of interpolants, given interpolant times $\tau_k \in [0, h)$ for $k \in [N]$,
\begin{align*}
    X_k^+ = X_{kh} - \int_{kh}^{kh+\tau_k} \nabla V(X_{kh}) \, \D t + \sqrt{2} (B_{kh+\tau_k} - B_{kh})\,.
\end{align*}
Therefore we can define, for $t \in [(k-1)h, kh)$,
\begin{align*}
    \bar \Mu_t = \frac{1}{\sqrt{2}}\Bigl(\nabla V(X_t) - \nabla V(X^+_k)\Bigr)\,.
\end{align*}

\begin{lemma}\label{lem:MD-multistep}
    For the extension $\bar \Mu$ defined above, we have under the same conditions as Lemma~\ref{lem:MD-one-step},
    \begin{align*}
       \Mdel \bar \Mu &= \frac{1}{\sqrt{2}}\int_0^{Nh} \inner{\nabla V(X_t), \,\D B_t}-\sum_{k=1}^N \frac{1}{\sqrt{2}}\inner{\nabla V(X^+_k), B_{kh} - B_{(k-1)h}} + \sum_{k=1}^N \frac{\tau_k}{\sqrt{2}} \tr(\nabla^2 V(X^+_k)) \\
        &\qquad + \sum_{k=1}^N \int_{(k-1)h}^{(k-1)h + \tau_k} t \tr(\nabla^2 V(X_t)\nabla^2 V(X^+_k))\, \D t\,.
    \end{align*}
    Furthermore,
    \begin{align*}
        \log \Mdet(\msf I + q \MD \bar \Mu) &=
        -\sum_{k=0}^{N-1} \Biggl\{\sqrt 2q^2 \int_{kh}^{kh+\tau_k} t \tr(\nabla^2 V(X_t) \nabla^2 V(X^+_k))\,\D t -\frac{\tau_k^2 q^2 }{\sqrt{2}} \tr((\nabla^2 V(X^+_k))^2)\Biggr\} \\
        &\qquad + O(\beta^3 d h^3 q^3 N)\,.
    \end{align*}
    In particular, for $h \lesssim \frac{1}{\beta q}$ with a small enough implied constant, the determinant can almost surely be bounded below by a non-zero constant and also be bounded above, showing that the transform is almost surely locally invertible.
\end{lemma}
\begin{proof}
    The computation of the Skorohod adjoint is simple, as each step can be handled as in the one-step case via linearity.
    
    Consider the same approximation for $X_{[0, T]}$ as in Lemma~\ref{lem:MD-one-step} with some parameter $\eta$. If we consider the Carleman--Fredholm determinant, it is easy to see that for $i\eta, j\eta \in [kh, (k+1)h)$ for some $k$, we will have a very similar computation to that in Lemma~\ref{lem:MD-one-step} for $\MD_{i\eta} X_{j\eta}$. On the other hand, for $i\eta \in [k'h, (k'+1)h)$ and $j\eta \in [kh, (k+1)h)$ and $k' > k$, then $\MD_{i\eta} X_{j\eta} = 0$, $\MD_{i\eta} X_k^+ = 0$. Thus, for the matrix of Malliavin derivative coefficients $\MD \Mu^\eta$, when approximating $X_t$ with its Euler discretization $\hat X_t^\eta$ for some $\eta$ dividing $h$, has the form
    \begin{align*}
        I + q\MD \Mu^\eta = \begin{bmatrix}
            I + q(\MD \Mu^\eta)_{[1,1]} & 0 & 0 & \hdots & 0 \\
            q(\MD \Mu^\eta)_{[2,1]} & I+ q({\MD \Mu^\eta})_{[2,2]} & 0 & \hdots & 0 \\
            q({\MD \Mu^\eta})_{[3,1]} & q({\MD \Mu^\eta})_{[3,2]} & I+q({\MD \Mu^\eta})_{[3,3]}& \hdots & 0 \\
            \vdots & \vdots & \vdots & \ddots & \vdots \\
            q({\MD \Mu^\eta})_{[N,1]} & q({\MD \Mu^\eta})_{[N,2]}  & q(\MD \Mu^\eta)_{[N, 3]} & \hdots & I + q(\MD \Mu^\eta)_{[N, N]}
        \end{bmatrix}\,.
    \end{align*}
    where each $({\MD \Mu^\eta})_{[i,j]} \in \R^{d\times d \times h/\eta \times h/\eta}$ is similar to that of Lemma~\ref{lem:MD-one-step}, and the coefficients $\{({\MD \Mu^\eta})_{[i,j]}: i \in [N], i > j\}$ are irrelevant both in computing the trace and the determinant.

    We then have that the determinant of the total kernel is given by the product of determinants of the kernels on the diagonal above, i.e., the determinant of the Malliavin derivative matrix restricted to the interval $[(k-1)h, kh) \times [(k-1)h, kh)$. Each determinant can be computed as in Lemma~\ref{lem:MD-one-step}, and the logarithm of the total determinant is simply given by the sum. This is stable in the limit $\eta \searrow 0$, so we can easily conclude.
\end{proof}

\paragraph{Part (D): Computation of the $\Renyi_q$ divergence.}
At last, we can consider the $\Renyi_q$ divergence of this process, given in the lemma below. Before this, one notes that we have computed the objects needed for $\frac{\D \mathbf Q}{\D \mathbf P}$. This is no issue, however, as we simply invert this for the Radon-Nikodym of the reverse process, noting that the necessary regularity properties continue to be satisfied. 

\begin{lemma}\label{lem:renyi-md}
If $\mathbf Q, \mathbf P$ are the laws of $(\tilde B_t)_{t \in [0, Nh)}$, $(B_t)_{t \in [0, Nh)}$ respectively, under Assumption~\ref{as:regularity}, for $q \geq 2$, $h \lesssim \frac{1}{\beta q}$,
\begin{align*}
    \Renyi_q(\mathbf P \mmid \mathbf Q) &\lesssim \beta^2 dh^2 q N + \frac{1}{q} \max_{k \in [N]} \log \E_{\mathbf P} \exp\Bigl(\beta^2 q^2\int_{(k-1)h}^{kh} \norm{X_{t} - X_k^+}^2\,\D t \Bigr)\,.
\end{align*}
\end{lemma}
\begin{proof}
    Using Theorem~\ref{thm:anticipating-girsanov} and Lemma~\ref{lem:MD-multistep}, with our step-size condition being necessary for existence, and applying Cauchy--Schwarz,
    \begin{align*}
        \Renyi_q(\mathbf P \mmid \mathbf Q) &= \frac{1}{q-1}\log\E_{\mathbf Q} \Bigl(\frac{\D \mathbf P}{\D \mathbf Q}\Bigr)^q = \frac{1}{q-1}\log\E_{\mathbf P} \Bigl(\frac{\D \mathbf P}{\D \mathbf Q}\Bigr)^{q-1}\\
        &\leq \frac{1}{q-1} \log \E_{\bf P} \Bigl[\Bigl\lvert\Mdet\Bigl(I + \MD \bar \Mu \Bigr)\Bigr\rvert^{-(q-1)} \exp\Bigl(q-1) \Mdel \bar\Mu + \frac{(q-1)^2}{2} \norm{\bar \Mu}^2_{\Hb} \Bigr) \Bigr] \\
        &\leq \frac{1}{q-1} \log \E_{\bf P} \Bigl[\Bigl\lvert\Mdet\Bigl(I -2(q-1) \MD \bar\Mu \Bigr)\Bigr\rvert \exp\Bigl(2(q-1) \Mdel \bar \Mu - 2(q-1)^2 \norm{\bar \Mu}^2_{\Hb} \Bigr) \Bigr] \\
        &\qquad +\frac{1}{q-1} \log \E_{\bf P} \biggl[\exp\Bigl(\bigl((q-1) + 2(q-1)^2\bigr)\norm{\bar\Mu}^2_{\Hb} \Bigr) \frac{\lvert\Mdet(I + \MD \bar \Mu)\rvert^{-2(q-1)}}{\lvert\Mdet(I - 2(q-1) \MD \bar \Mu)\rvert}\biggr]\,.
    \end{align*}
    Via Corollary~\ref{cor:doleans-dade}, the first term has unit expectation if we consider the process $-2(q-1) \bar\Mu$. Such a process is still invertible for $h \lesssim \frac{1}{\beta q}$ sufficiently small as its determinant can be found from Lemma~\ref{lem:MD-multistep}.
    Now, we note that we had bounded
    \begin{align*}
         \log \lvert\Mdet (I + r \MD \bar \Mu )\rvert &\lesssim -\sqrt 2r^2\sum_{k=1}^N \int_{(k-1)h }^{(k-1)h +\tau_k} t \tr(\nabla^2 V(X_t) \nabla^2 V(X_k^+))\,\D t \\
         &\qquad + \sum_{k=1}^N \frac{\tau_k^2 r^2}{\sqrt{2}} \tr((\nabla^2 V(X^+_k))^2) + O(\beta^3 dh^3 \abs{r}^3 N) \lesssim \beta^2 d h^2 r^2 N\,,
    \end{align*}
    almost surely, for $h \lesssim \frac{1}{\beta q}$, and furthermore this holds for all $r \in \R$. If we consider the cases where $r = 1$ and $-(q-1)$ respectively, then we can bound the determinant by
    \begin{align*}
        \log \Bigl[\frac{\lvert\Mdet (I -2(q-1) \MD \bar \Mu )\rvert^{-1}}{\lvert\Mdet (I +\MD \bar \Mu)\rvert^{2(q-1)}}\Bigr] \lesssim \beta^2 dh^2 q^2 N\,.
    \end{align*}
    almost surely.
    Now, almost surely we can write
    \begin{align*}
        \lim_{\eta \searrow 0} \norm{\bar \Mu}_{\Hb}^2 \leq \sum_{k=1}^N \int_{(k-1)h}^{kh} \frac{\beta^2}{2} \norm{X_t - X_k^+}^2 \, \D t\,.
    \end{align*}
    It follows that as $\tau_k \leq h$, and for $h$ small enough, via the AM--GM inequality,
    \begin{align*}
        \Renyi_q(\mathbf P \mmid \mathbf Q) &\lesssim \beta^2 dh^2 q N +  \frac{1}{q-1} \log \E_{\mathbf P} \exp\Bigl(\sum_{k=1}^N \beta^2 q^2\int_{(k-1)h}^{kh} \norm{X_{t} - X_k^+}^2\,\D t \Bigr) \\
        &\lesssim \beta^2 dh^2 qN + \frac{1}{q} \max_{k \in [N]} \log \E_{\mathbf P} \exp\Bigl(\beta^2 q^2 N\int_{(k-1)h}^{kh} \norm{X_{t} - X_k^+}^2\,\D t \Bigr) \,.
    \end{align*}
\end{proof}

Finally, we can state a sampling corollary, once we analyze the expectation in this result.
\begin{corollary}\label{cor:ormd-error}
    If $\mathbf Q, \mathbf P$ are the laws under which $(X_t)_{t \in [0, T)}$ follows~\eqref{eq:langevin_diff} and~\eqref{eq:M-LMC} respectively, it follows that if the initial distribution is $X_0 \sim \mu_0$, under Assumption~\ref{as:regularity}, and then for $q \geq 2$
    \begin{align*}
        \Renyi_q(\mathbf P \mmid \mathbf Q) \lesssim \beta^2 d h^2 qN + \frac{1}{q}\sup_{k \leq N} \log \E_{\mbb P}\Bigl[\exp\Bigl(\beta^2 h^3 q^2 N\norm{\nabla V(X_{(k-1)h})}^2\Bigr)\Bigr]\,.
    \end{align*}
    In particular, $T \deq Nh$ and $h = \Otilde( \frac{1}{\beta^{3/2} \sqrt{qT}} \wedge \frac{d}{\beta \Renyi_3(\mu_0 \mmid \pi)})$, then
    \begin{align*}
        \Renyi_q(\mu_{Nh} \mmid \tilde \mu_{Nh}) \lesssim \beta^2 d h^2 q N\,.
    \end{align*}
    Finally, suppose $\pi$ satisfies~\eqref{eq:lsi} with constant $1/\alpha$ and we define $\kappa \deq \beta/\alpha$, assume that $\Renyi_q(\mu_0 \mmid \pi) \lesssim d \log \kappa$. This implies that, with $h = \widetilde{\Theta}(\frac{\varepsilon^2}{\beta \kappa d q})$ for $\varepsilon$ sufficiently small, we can output a distribution $\hat \pi$ with $\Renyi_q(\hat \pi \mmid \pi) \lesssim \varepsilon^2$, using the following number of gradient queries,
    \begin{align*}
        N = \widetilde{\Theta}\Bigl(\frac{\kappa^2 d q^2}{\varepsilon^2} \Bigr)\,.
    \end{align*}
\end{corollary}

\begin{remark}
    The assumption $\Renyi_q(\mu_0 \mmid \pi) \lesssim d \log \kappa$ is reasonable in light of~\cite[\S D]{zhang2023improved}. The corollaries above fail to witness the beneficial properties of a randomized midpoint scheme; namely for $\Renyi_2$, it is able to achieve error on the order of $\Otilde(\kappa d^{1/2}/\varepsilon)$ rather than $\Otilde(\kappa^2 d/\varepsilon^2)$ as claimed above. This is because our analysis must first fix the endpoint when computing the paths, and then average on the outside. This essentially forces us to accrue a quantity similar to the strong error in~\cite{scr3}, without any possibility of improved weak error.
\end{remark}

\subsection{DM--ULMC}\label{scn:dm-ulmc}
Before proceeding, let us define the following helpful integrals:
\begin{align}
    \eia(s,t) &\deq \exp(-\gamma(t-s))\,, \label{eq:first-int}\\
    \eib(s,t) &\deq \frac{1}{\gamma}\Bigl(1-\exp(-\gamma(t-s))\Bigr) = \int_s^t \eia(s,r) \, \D r\,, \label{eq:second-int}\\
    \eic(s,t) &\deq \frac{1}{\gamma}\Bigl((t-s) + \frac{1}{\gamma} \bigl(\exp(-\gamma(t-s)) -1\bigr)  \Bigr) = \int_s^t \eib(s,r) \, \D r \label{eq:third-int}\,.
\end{align}
We now consider a new discretization given by the following iterates, for a fixed double sequence of interpolant times $((\tau_1^-, \tau_1^+), \dotsc, (\tau_N^-, \tau_N^+))$ all living in $[0, h) \times [0, h)$:
\begin{align}\label{eq:DM--ULMC-marginal}
\begin{aligned}
    X_{(k+1)h} &= X_{kh} + \eib(0, h) P_{kh} -  \eic(0, h) \nabla V(X^-_k) + \sqrt{2\gamma} \int_{kh}^{(k+1)h} \eib(s, (k+1)h) \, \D B_s\,,\\
    P_{(k+1)h} &= \eia(0, h) P_{kh} - \eib(0, h) \nabla V(X^+_k) + \sqrt{2\gamma} \int_{kh}^{(k+1)h} \eia(s,(k+1)h) \, \D B_s\,,
\end{aligned}
\end{align}
with the interpolants
\begin{align*}
    X^-_k &= X_{kh} + \eib(0, \tau_{k}^-) P_{kh} - \eic(0, \tau_k^-) \nabla V(X_{kh}) + \sqrt{2\gamma}\int_{kh}^{kh + \tau_k^-} \eib(t, kh + \tau_k^-) \, \D B_t\,, \\
    X^+_k &=X_{kh} + \eib(0, \tau_k^+) P_{kh} - \eic(0, \tau_k^+) \nabla V(X_{kh}) + \sqrt{2\gamma} \int_{kh}^{kh + \tau_k^+} \eib(t, kh + \tau_k^+) \, \D B_t\,.
\end{align*}
We use the notation $X_k^+, X_k^-$ because for both deterministic and random algorithms $\tau_k^+ > \tau_k^-$ will be chosen, almost surely or on average respectively. For instance, in the deterministic algorithm, we take $\tau_k^+ = \frac{h}{2}$, while $\tau_k^- = \frac{h}{3}$ for all $k \in [N]$. We note that this algorithm comes from prior work on midpoint discretizations for~\eqref{eq:ULD}, which see improved error either under randomization~\cite{shen2019randomized, he2020ergodicity, yu2023langevin, scr3, scr4} or additional smoothness~\cite{sanz2021wasserstein, paulin2024correction}. 

However, like~\cite{scr4}, when proving results in $\KL$ divergence, it will be necessary to introduce a second midpoint to improve the error in the position coordinate. In the case of random midpoints, we note that it is necessary in order to ensure that the weak errors of both position and momentum are suitably small. See~\cite{scr4} for more justification. In the case of deterministic midpoints, we note that the expansion in Lemma~\ref{lem:dmd-local-error} requires cancellation of the lowest order term, which occurs only for the specific choice of $\tau_k^-, \tau_k^+$ given above.

While the algorithm above only fixes the marginal distributions at each $\{kh\}_{k \in [N]}$, 
we define the following interpolation for $t \in [kh, (k+1)h)$,
\begin{align}\label{eq:DM--ULMC-multistep}\tag{DM--ULMC}
\begin{aligned}
    X_t &\deq X_{kh} + \int_{kh}^t P_s \, \D s \\
    P_t &\deq P_{kh} - \gamma \int_{kh}^t P_s \, \D s - \int_{kh}^t \msf G_s^\opt \, \D s + \sqrt{2\gamma}\int_{kh}^t \D B_s\,,
\end{aligned}    
\end{align}
where now
\begin{align*}
    \msf G_t^\opt = \nabla V(X_t) - \eia(t, (k+1)h) \lambda_{k,1} - \eib(t, (k+1)h) \lambda_{k,2}
\end{align*}
and
\begin{align*}
    \msf G_k^\msp \deq \eib(0, h) \nabla V(X_k^+)\,, \qquad \msf G_k^\msx \deq \eic(0, h) \nabla V(X_k^-)\,,
\end{align*}
with the four coefficients,
\begin{align*}
    \sigma_{11} &= \int_0^h \eia(t, h)^2 \, \D t\,, \quad \sigma_{21} = \sigma_{12} = \int_0^h \eia(t, h)\eib(t,h)\, \D t \\
    \sigma_{22} &= \int_0^h \eib(t,h)^2 \, \D t\,, \quad 
    \Delta_\sigma = \sigma_{11} \sigma_{22} - \sigma_{21}^2\,.
\end{align*}
and
\begin{align*}
    \lambda_{k,1} &= \frac{\sigma_{22}\Bigl[\int_{kh}^{(k+1)h} \eia(t, (k+1)h) \nabla V(X_t) \, \D t - \msf G_k^\msp\Bigr] - \sigma_{21}\Bigl[\int_{kh}^{(k+1)h} \eib(t, (k+1)h) \nabla V(X_t) \, \D t - \msf G_k^\msx \Bigr]}{\Delta_\sigma} \\
    \lambda_{k,2} &= \frac{-\sigma_{12}\Bigl[\int_{kh}^{(k+1)h} \eia(t, (k+1)h) \nabla V(X_t) \, \D t - \msf G_k^\msp \Bigr] + \sigma_{11} \Bigl[\int_{kh}^{(k+1)h} \eib(t, (k+1)h) \nabla V(X_t) \, \D t - \msf G_k^\msx\Bigr]}{\Delta_\sigma}\,,
\end{align*}

As this algorithm uses two midpoints, we call it the underdamped double (deterministic) midpoint discretization, or \textbf{DM--ULMC} for short. We also give an equivalent formulation of~\eqref{eq:underdamped_langevin_diff} for which comparison is easier, for any $t \geq kh$:
\begin{align}\label{eq:ULD}\tag{ULD}
    X_t &\deq X_{kh} + \eib(0,t) P_{kh} - \int_{kh}^t \eib(s, t) \nabla V(X_s) \, \D s + \sqrt{2\gamma} \int_{kh}^{t} \eib(s, t) \, \D \tilde B_s\,, \\
    P_t &\deq \eia(0,t) P_{kh} - \int_{kh}^t \eia(s, t) \nabla V(X_s) \, \D s + \sqrt{2\gamma} \int_{kh}^{t} \eia(s, t) \, \D \tilde B_s\,,
\end{align}
using the fact that
\begin{align*}
    \int_0^t \int_0^s \eia(r, s) \, \D \tilde B_r \, \D s = \int_0^t \eib(s, t) \, \D \tilde B_s\,.
\end{align*}
Here, note that we want $(X_t, P_t)_{t \in [0, T)}$ to solve~\eqref{eq:DM--ULMC-multistep} for $(B_t)_{t \in [0, T)}$, and~\eqref{eq:ULD} for a different process $(\tilde B_t)_{t \in [0, T)}$. To find the transform which turns $(B_t)_{t \in [0, T)}$ into $(\tilde B_t)_{t \in [0, T)}$, we see it should be of the form~\eqref{eq:transformation-generic} with $\Mu$ defined as
\begin{align}\label{eq:difference-DM--ULMC-multistep}
    \bar \Mu_t = -\frac{1}{\sqrt{2\gamma}} \cdot \Bigl\{\eia(t, (k+1)h) \lambda_{k, 1} + \eib(t, (k+1) h) \lambda_{k,2} \Bigr\}\,,
\end{align}
for $t \in [kh, (k+1)h)$, for each $k \in [N]$.

\paragraph{Intuitions}
Consider the processes on $[0, h)$, for a single step of~\eqref{eq:DM--ULMC-multistep}. A na\"ive extension of~\eqref{eq:DM--ULMC-marginal}, obtained for instance by replacing $h$ with $t$ everywhere, is unfavourable when evaluated through Girsanov's theorem and comparing against~\eqref{eq:underdamped_langevin_diff}. Indeed, one will find that the corresponding path measure is mutually singular with respect to $\bf Q$, the measure of $\tilde B_{[0,h]}$. Instead, we seek a process with the correct marginal law, but which minimizes the error which would arise through Girsanov's theorem. This has the form
\begin{align*}
    \operatorname{objective}(\MG_{[0, h)}) \propto \E_{\mathbf P}[\norm{\MG_{[0,h)} - \nabla V(X)_{[0, h)}}^2_\Hb]\,.
\end{align*}
In particular, let $\msf G^\opt_{[0, h)}$ be the process in~\eqref{eq:DM--ULMC-multistep}
subject to the two constraints at time $h$ given by~\eqref{eq:DM--ULMC-marginal}. As our constraints must hold for every fixed realization of the Brownian path $B_{[0, h)}$, our objective is also most naturally formulated pointwise with respect to this path. 

Although $X_t$ itself depends implicitly on $\msf G_t^\opt$, we can reasonably expect that this dependence will generally be lower order (on the scale of $h^2$); furthermore, incorporating this dependence in the optimization leads to a solution that is difficult to express in closed form. Thus, when optimizing, we ignore the differential of $\nabla V(X_t)$ with respect to $\msf G_t^\opt$. This leads to a cost-minimizing velocity
\begin{align*}
    \msf G_t^\opt(X_{[0, h)}) = \nabla V(X_t) - \eia(t, h) \lambda_1(X_{[0, h)}) - \eib(t, h) \lambda_2(X_{[0, h)})\,,
\end{align*}
where the coefficients $\lambda_i$ are chosen to satisfy the marginal constraints.

Now, consider the marginal effective velocities $\MG^\msp, \MG^\msx$ which gives (up to a scaling constant) the amount of gradient we accumulate in the momentum and position respectively at $t = h$. 
Evidently, this causes
\begin{align*}
    \int_0^h \eia(t, h) \msf G_t^\opt(X_{[0, h)}) \, \D t = \msf G^\msp\,, \qquad \int_0^h \int_0^t \eia(s,t) \msf G_s^\opt(X_{[0, h)})  \, \D s \, \D t = \msf G^\msx\,,
\end{align*}
as we can view the matrix $\Sigma$ with coefficients $\sigma_{ij}$ as giving
\begin{align*}
    \int_0^h \eia(t,h) \msf G_t^\opt(X_{[0, h)}) \, \D t = \int_0^h \eia(t, h) \nabla V(X_t)\, \D t -\Bigl[\Sigma \Sigma^{-1} \operatorname{diff}\Bigr]_1\,, \\
    \int_0^h \int_0^t \eia(s,h) \msf G_s^\opt(X_{[0, h)}) \, \D s\, \D t = \int_0^h \eib(t, h) \nabla V(X_t)\, \D t -\Bigl[\Sigma \Sigma^{-1} \operatorname{diff} \Bigr]_2\,,
\end{align*}
where $\operatorname{diff}^\msp = \int_0^h \eia(t, h) \nabla V(X_t)\, \D t -\msf G^\msp$ and $\operatorname{diff}^\msx = \int_0^h \eib(t, h) \nabla V(X_t)\, \D t -\msf G^\msx$, and $\operatorname{diff} = [\operatorname{diff}^\msp, \operatorname{diff}^\msx]^\top$. We note that our choice of $\lambda_1, \lambda_2$ can easily be seen to satisfy both constraints, corresponding to $(\Sigma^{-1} \operatorname{diff})_1, (\Sigma^{-1} \operatorname{diff})_2$ respectively.

Note that the solution for $\MG_t^\opt$ is implicit due to the dependence of $X_t$ itself on $\msf G_t^\opt$. For a fixed realization of $B_{[0, h)}$ with $h$ sufficiently small, and any initial point $(X_0, P_0)$, this implicit equation will have a unique solution, which can be justified by Lemma~\ref{lem:anticipating-SDE-existence}. This implies the bijectivity property for The verification of this will follow from our Malliavin computations below, but again it suffices to see that the implicit part of the dynamics arising from $\msf G_t^\opt$ contributes only higher order terms, compared to the direct dependence through $B_t$. Thus, we can generally disregard the dependence of $\MG_t^\opt$ on $X_{[0, h)}$ in our calculations.

\paragraph{Parts (A)-(C): Computing the relevant Malliavin quantities for DM--ULMC.} It remains to compute the relevant quantities needed for Theorem~\ref{thm:anticipating-girsanov}. As this computation will be significantly more tedious than our earlier results in Lemma~\ref{lem:MD-one-step}, we will only present the leading terms in terms of their asymptotic orders. The full derivations can be found in our appendices.
\begin{lemma}\label{lem:DMD-one-step-uld}
    For a piecewise constant approximation $\Mu^\eta$ for $\bar \Mu$, one can find explicit expressions for $\Mdel \Mu^\eta$, and $\log \det (\msf I+q\MD \Mu^\eta)$ under Assumption~\ref{as:regularity} which are stable as $\eta \searrow 0$, and almost surely we have for integer $q$ and $h \lesssim \frac{1}{\beta^{1/2}q^{1/2}}$ that 
    \begin{align*}
        \log \Mdet (\msf I + q\MD \bar \Mu) = O({\beta^2 dh^4 q^2})\,.
    \end{align*}
    If we extend this to multiple steps, we find instead the bound
    \begin{align*}
        \lim_{\eta \searrow 0} \log \Mdet \Bigl(\msf I + q \MD \bar \Mu\Bigr) = O({\beta^2 dh^4 q^2 N})\,.
    \end{align*}
\end{lemma}
The proof of the multistep formula is essentially equivalent to analysis in Lemma~\ref{lem:MD-multistep}. 

\paragraph{Part (D): Computation of the $\Renyi_q$ divergence.}

\begin{lemma}\label{lem:DMD-renyi-generic}
    Let $\bf Q$, $\bf P$ be the laws under which $(\tilde B_t)_{t \in [0, Nh)}$ and $(B_t)_{t \in [0, Nh)}$ respectively are Brownian. Then, we have, for $q \geq 2$, if $h \lesssim \frac{1}{\beta^{1/2} q^{1/2}}$, under Assumption~\ref{as:regularity}
    \begin{align*}
        \Renyi_q(\mathbf{P} \mmid \mathbf{Q}) 
        &\lesssim {\beta^2 d h^4 q N} + \frac{1}{q} \max_{k \in [N]} \log \E_{\mathbf P} \exp\Bigl(\frac{q^2 N}{\gamma}\int_{(k-1)h}^{kh} \norm{\eia(t,kh) \lambda_{k,1} + \eib(t,kh) \lambda_{k,2}}^2\D t \Bigr) \,.
    \end{align*}
    We also have the $\KL$ bound
    \begin{align*}
        \KL(\mathbf{P} \mmid \mathbf{Q}) 
        &\lesssim {\beta^2 dh^4 N} + \max_{k \in [N]} \E_{\mathbf P} \Bigl[\frac{N}{\gamma}\int_{(k-1)h}^{kh} \norm{\eia(t,kh) \lambda_{k,1} + \eib(t,kh) \lambda_{k,2}}^2\D t \Bigr] \,.
    \end{align*}
\end{lemma}
\begin{proof}
    The proof follows essentially the same outline as Lemma~\ref{lem:renyi-md}. Using Theorem~\ref{thm:anticipating-girsanov} and Lemma~\ref{lem:DMD-one-step-uld}, noting that the dynamics are only well-defined under the condition given on the step-size,
    \begin{align*}
        \Renyi_q(\mathbf P \mmid \mathbf Q)
        &\leq \frac{1}{q-1} \log \E_{\bf P} \Bigl[\Mdet\Bigl(I -2(q-1) \MD \bar \Mu \Bigr) \exp\Bigl(2(q-1) \Mdel \bar \Mu - 2(q-1)^2 \norm{\bar \Mu}^2_{\Hb} \Bigr) \Bigr] \\
        &\qquad +\frac{1}{q-1} \log \E_{\bf P} \biggl[\exp\Bigl(\bigl((q-1) + 2(q-1)^2\bigr)\norm{\bar \Mu}^2_{\Hb} \Bigr) \frac{\det(I + \MD \bar \Mu)^{-2(q-1)}}{\det(I - 2(q-1) \MD \bar \Mu)}\biggr]\,.
    \end{align*}
    Via Corollary~\ref{cor:doleans-dade}, the first term has unit expectation again. Now, we note that we had bounded
    \begin{align*}
         \log \Mdet \Bigl(I + r \MD\bar \Mu \Bigr)
         &\lesssim {\beta^2 d h^4 r^2 N}\,,
    \end{align*}
    almost surely, for $h \lesssim \frac{1}{\beta^{1/2} r^{1/2}}$, and furthermore this holds for all $r \in \R$. If we apply this inequality twice with $r = 1$ and $r = -(q-1)$, then we can bound the determinant by
    \begin{align*}
        \log \biggl[\frac{\Mdet (I -2(q-1) \MD\bar \Mu )^{-1}}{\Mdet (I +\MD\bar \Mu )^{2(q-1)}}\biggr] \lesssim {\beta^2 d h^4 q^2 N}\,, 
    \end{align*}
    almost surely.
    It follows that for $h$ small enough, via the AM--GM inequality,
    \begin{align*}
        \Renyi_q(\mathbf P \mmid \mathbf Q) &\lesssim {\beta^2 dh^4 qN}+  \frac{1}{q-1} \log \E_{\mathbf P} \exp\Bigl(\sum_{k=1}^N \frac{q^2}{\gamma}\int_{(k-1)h}^{kh} \norm{\msf G_t^\opt - \nabla V(X_t)}^2\,\D t \Bigr) \\
        &\lesssim {\beta^2 d h^4 q N} + \frac{1}{q} \max_{k \in [N]} \log \E_{\mathbf P} \exp\Bigl(\frac{q^2 N}{\gamma}\int_{(k-1)h}^{kh} \norm{\eia(t,kh) \lambda_{k,1} + \eib(t,kh) \lambda_{k,2}}^2\D t \Bigr) \,.
    \end{align*}
    Finally, repeating the above analysis for the $\KL$ divergence shows that the required error bound is instead $L^2$ under $\mathbf P$; alternatively, we could have shown this by a more careful analysis combined with a limiting argument.
\end{proof}

The following lemma leverages the above to obtain a crude R\'enyi bound. We will use this in multiple ways in order to obtain a refined $O(d^{1/4})$ dependence.
\begin{lemma}\label{lem:renyi-dmd-coarse}
    Under Assumption~\ref{as:regularity}, with $\mathbf P, \mathbf Q$ denoting paths starting at a measure $\bs \mu_0$ and following~\eqref{eq:DM--ULMC-multistep} and~\eqref{eq:ULD} respectively, $h \lesssim \frac{\gamma^{1/2}}{\beta q^{1/2} T^{1/2}} \wedge \frac{1}{\beta^{1/2} q^{1/2}}$,
    \begin{align*}
        \Renyi_q(\mathbf P \mmid \mathbf Q) \lesssim \frac{\beta^2 {h}^3 q N}{\gamma} \Bigl(d + \Renyi_3(\bs \mu_0 \mmid \bs \pi) \Bigr)\,,
    \end{align*}
\end{lemma}
\begin{remark}
    The above bound does not leverage any notion of higher order smoothness and holds for any choice of midpoint, and hence obtains a rate no better than ULMC.
\end{remark}

\subsubsection{Deterministic double midpoint}
Now, it remains to obtain a refined bound through a higher-order expansion. Assume that $\tau_k^- = \frac{h}{3}$ and $\tau_k^+ = \frac{h}{2}$ identically for all $k \in [N]$. To take advantage of the improved properties of the discretization, we impose the following assumption. First, for a $3$-tensor $A \in \R^{d\times d \times d}$, we define its $\{1,2\}, \{3\}$ norm as follows:
\begin{align*}
    \norm{A}_{\{1,2\}, \{3\}} = \sup_{x \in \R^{d\times d}, y \in \R^d} \Bigl\{A_{ijk}\, x_{ij}\, y_k; \quad \norm{x}_F \leq 1,\, \norm{y} \leq 1 \Bigr\}\,.
\end{align*}
We can also define its $\{1,2,3\}$ norm, or Frobenius norm, with
\begin{align*}
    \norm{A}_{\{1,2,3\}} = \sqrt{\sum_{i,j,k=1}^d A_{i,j,k}^2}\,.
\end{align*}
The following assumption may seen somewhat technical at a glance, but it is crucial for our results.
\begin{assumption}\label{as:chaos-tail}
    $V$ is thrice differentiable and its third derivative tensor has the property that, for $\xi$ a standard Gaussian and $x \in \R^d$ deterministic and arbitrary, for all $\delta \in (0, 1/c)$ for an absolute constant $c$, with probability $1-\delta$, and for some parameters $\beta_\eff, \beta_H \in \R_+$ with $\beta_H = \sup_{x \in \R^d} \norm{\nabla^3 V(x)}_{\{1,2\}, \{3\}}$,
    \begin{align*}
        \norm{\nabla^3 V(x)[\xi, \xi]}^2 \lesssim \beta_H^2 d + \beta_{\eff}^2 d^{1/2} \log \frac{1}{\delta} + \beta_{\eff}^2 \log^2 \frac{1}{\delta}\,.
    \end{align*}
\end{assumption}
The assumption of thrice-differentiability is not strictly necessary; twice-differentiability combined with boundedness of the subdifferential of $\nabla^2 V$ is enough, although we do not work with such an assumption for simplicity. Assumption~\ref{as:chaos-tail} is somewhat non-standard, compared to the assumption $\norm{\nabla^3 V}_{\{1,2\}, \{3\}} \lesssim \beta_H$ which has appeared in~\cite{paulin2024correction, chen2023does}. Nonetheless, we demonstrate that Assumption~\ref{as:chaos-tail} holds with dimension-independent $\beta_\eff$ in two key settings.
\begin{example}[Product measures]
    Suppose $\bs \pi$ is a product measure, so that $(\nabla^3 V)_{i,j,k} = 0$ unless $i = j = k$. Then, this reduces to bounding
    \begin{align*}
        \Bigl\lvert \sum_{i=1}^d\bigl\lvert\partial_i^3 V(x_i) [\xi_i, \xi_i]\bigr\rvert^2- \E \sum_{i=1}^d\bigl\lvert\partial_i^3 V(x_i) [\xi_i, \xi_i]\bigr\rvert^2\Bigr\rvert \lesssim \norm{\nabla^3 V}_{\{1,2\},\{3\} }^2 \abs{\sum_{i=1}^d \xi_i^4 - \E \sum_{i=1}^d \xi_i^4}\,.
    \end{align*}
    noting that $\norm{\nabla^3 V}_{\{1,2\},\{3\} } = \max_{i \in [d]} \abs{\partial_i^3 V}$, and thus that $\beta_H^2$ is dimension-free.

    Now, this is an order-$4$ diagonal Gaussian chaos. Consider the vector $X \in \R^{d/2}$ with $X_j = \xi_{2j}^2 + \xi_{2j+1}^2$, for $j = 1,\dotsc, d$ (assuming $d$ is even for simplicity). Then, each $X_j \sim \msf{exponential}(0, 1/2)$. It follows that as each component of $X_j$ satisfies a Poincar\'e inequality with constant $4$, that the vector $X$ also satisfies a Poincar\'e inequality with the same constant. Finally, this implies that for $t \gtrsim 1$, as $X \mapsto \norm{X}$ is Lipschitz,
    \begin{align*}
        \Pr\{\abs{\norm{X}- \E \norm{X}} \geq t\} \leq 2\exp\Bigl(-\frac{t}{8} \Bigr)\,,
    \end{align*}
    or alternatively, if $\log \frac{1}{\delta} \gtrsim 1$, then with probability $1-\delta$,
    \begin{align*}
        \norm{X} \lesssim \E[\norm{X}] + \log \frac{1}{\delta}\,.
    \end{align*}
    It follows by Jensen's inequality and Young's inequality that
    \begin{align*}
        \norm{X}^2 \lesssim \E[\norm{X}^2] + \E[\norm{X}]\log \frac{1}{\delta} + \log^2 \frac{1}{\delta}\,.
    \end{align*}
    As $\E[\norm{X}^2] \lesssim d$, in this case we have (multiplying by $\beta_H^2$), Assumption~\ref{as:chaos-tail} with $\beta_\eff = \beta_H$. 
\end{example}

\begin{example}[Dimension-free $\beta_T$]
    When $\beta_T = \sup_{x \in \R^d} \norm{\nabla^3 V(x)}_{\{1,2, 3\}}$ does not depend on the dimension, then Lemma~\ref{lem:frobenius-gaussian-tail} tells us that
    \begin{align*}
        \Pr\Bigl\{\Bigl\lvert&\norm{\nabla^3 V(x)[\xi, \xi]}^2 - \E \norm{\nabla^3 V(x)[\xi, \xi]}^2\Bigr\rvert \geq t \Bigr\} \\
        &\leq 2\exp\Bigl(-c \min \Bigl\{\frac{\sqrt{t}}{\norm{\nabla^3 V}_{\{1,2,3\} }}, \frac{t}{\sqrt{d} \norm{\nabla^3 V}_{\{1,2\},\{3\}}^2 + \norm{\nabla^3 V}_{\{1,2,3\}}^2} \Bigr\} \Bigr)\,.
    \end{align*}
    This verifies Assumption~\ref{as:chaos-tail} with constant $\beta_\eff^2 = \max\{\beta_T^2, \beta_H^2\} = \beta_T^2$, since $\beta_T^2$ always upper bounds $\beta_H^2$.
    When $\norm{\nabla^3 V}_{\{1,2,3\}}^2$ is not implicitly dependent on the dimension, then this implies $\beta_\eff$ is dimension-independent, which is needed for the $O(d^{1/4})$ rates. This occurs for instance in some of the examples considered in~\cite{chen2023does}.
\end{example}
Finally, we do not rule out the possibility that the sole assumption $\norm{\nabla^3 V}_{\{1,2\}, \{3\}} \lesssim \beta_H$ may be sufficient to obtain a dimension-free $\beta_\eff$ in Assumption~\ref{as:chaos-tail}. As far as we are aware, there is no such concentration result available in the literature, and a thorough exploration of this subject would be beyond the scope of the present work.

For our final result, we only obtain a single-step bound assuming a crude warm-start, $\Renyi_4(\bs \mu_0 \mmid \bs \pi) \lesssim d^{1/2}$. As we cannot expect to obtain $O(d^{1/4})$ dependence across multiple steps without taking into account weak error, this is the best that can be done using \emph{only} the techniques presented in this paper.
\begin{theorem}[{Sharp one-step regularity of~\eqref{eq:DM--ULMC-multistep}}]\label{thm:dmd-sharp-one-step}
    Suppose Assumptions~\ref{as:regularity} and~\ref{as:chaos-tail} hold. Furthermore, suppose that $\Renyi_4(\bs \mu_0 \mmid \bs \pi) \lesssim d^{1/2}$, and $h \lesssim \frac{1}{\sqrt{\beta}}$, $\gamma \asymp \sqrt{\beta}$. Then, we can guarantee that, for $\bs \mu_h$ the output of~\eqref{eq:DM--ULMC-marginal} after a single step, and $(\tilde{\bs \mu}_t)_{t \in [0, h)}$ the evolution of $\bs \mu_0$ along~\eqref{eq:ULD},
    \begin{align*}
        \KL(\bs \mu_h \mmid \tilde{\bs \mu}_h) \lesssim {\beta^2 d h^4} + \frac{(\beta_H^2 + \beta_\eff^2) d h^5}{\gamma}\,.
    \end{align*}
\end{theorem}

\section{Combining with Shifted Composition}\label{scn:discussion}
In all of the algorithms above, one hopes to take advantage of the \emph{weak error} of the schemes. More precisely, prior work~\cite{scr3} has shown that the quality of the approximation can be understood in terms of the weak and strong errors of the algorithm, quantified on average over $\bs \mu_t$.
\begin{lemma}[KL local error framework for~\eqref{eq:ULD}]\label{lem:local_error}
    Let $\bs P$ denote the Markov kernel for~\eqref{eq:ULD}, run for time $h$, and let ${\bs P^\alg}$ be the kernel of a numerical discretization.
    Assume that for all $x,\tilde x,p,\tilde p\in\R^d$, there are random variables $(X_h, P_h) \sim \delta_{x,p}\bs P$, $(X_h^\alg,P_h^\alg) \sim \delta_{x,p} \bs P^\alg$ such that the following conditions hold, for $\bs \mu_{kh} = \bs \mu_0 (\bs P^{\alg})^k$,
    \begin{enumerate}
        \item (weak error)
        \begin{align*}
            \sup_{k \in [N]}\Bigl\{h^{-1} \E_{(x,p) \sim \bs \mu_{kh}} [\,\norm{\E X^\alg_h - \E X_h}^2] \vee \E_{(x,p) \sim \bs \mu_{kh}} [\norm{\E P^\alg_h - \E P_h}^2]\Bigr\} \le (\bar{\mc E}^{\rm w})^2\,.
        \end{align*}
        \item (strong error)
        \begin{align*}
            \sup_{k \in [N]}\Bigl\{h^{-1} \E_{(x,p) \sim \bs \mu_{kh}} [\,\norm{X^\alg_h - X_h}^2] \vee \E_{(x,p) \sim \bs \mu_{kh}} [\norm{P^\alg_h - P_h}^2]\Bigr\} \le (\bar{\mc E}^{\rm s})^2\,.
        \end{align*}
        \item (cross-regularity) 
        \begin{align*}
            \E_{x,p \sim \bs \mu_{(N-1)h}}[\KL(\delta_{x,p} \bs P^\alg \mmid \delta_{\tilde x,\tilde p} \bs P)] &\lesssim (\gamma h^3)^{-1}\E_{x,p \sim \bs \mu_{(N-1)h}}[\,\norm{x-\tilde x}^2] \\
            &\qquad+ (\gamma h)^{-1}\E_{x,p \sim \bs \mu_{(N-1)h}}[\,\norm{p-\tilde p}^2] + {\bar b}^2\,.
        \end{align*}
    \end{enumerate}
    Let $\alpha > 0$, $\gamma = \sqrt{32\beta}$, $\omega \deq \alpha/(3\gamma)$, and assume additionally that $h \lesssim 1/(\beta^{1/2}\kappa)$, $\kappa \deq \beta/\alpha$.
    Adopt Assumption~\ref{as:regularity}.
    Then, for all $\bs\mu_0 \in \mc P_2(\R^d\times\R^d)$ and $\bs \pi$ the stationary distribution of~\eqref{eq:ULD},
    \begin{align*}
        \KL(\bs\mu_{Nh} \mmid \bs\pi)
        &\lesssim  \Bigl(\frac{1}{\gamma}\Bigl(\frac{\omega}{\exp(c\omega T) - 1} \Bigr)^3 +\frac{\gamma \omega}{\exp(c\omega T) - 1}\Bigr)\,\mc W_2^2(\bs \mu_0,\bs \pi) + \mathtt{Err} + \bar b^2\,.
    \end{align*}
    Here, $\mc W_2$ denotes $2$-Wasserstein distance in the twisted norm $(x,p)\mapsto \sqrt{\norm x^2 + \gamma_0^{-2}\,\norm p^2}$, with    \begin{align*}
        \gamma_0
        &\gtrsim \gamma + \frac{\one_{T\le 1/\omega}}{T}\,.
    \end{align*}
    Then,
        \begin{align*}
            \mathtt{Err}
            &\lesssim \frac{1}{\alpha h^2}\,(\bar{\mc E}^{\rm w})^2 + \Bigl[\frac{1}{\beta^{1/2} h} \log\frac{1}{\omega h}\Bigr]\,(\bar{\mc E}^{\rm s})^2\,.
        \end{align*}
\end{lemma}
An analogous version holds for~\eqref{eq:langevin_diff}, without the need to distinguish between two separate coordinates. See~\cite[Theorem 1.2]{scr3}.

\subsection{Improved cross-regularity guarantees for M--LMC and DM--ULMC}
Both algorithms~\eqref{eq:M-LMC} and~\eqref{eq:DM--ULMC-multistep} can be viewed as deterministic equivalents of the over- and underdamped randomized midpoint algorithms.

In~\cite{scr3, scr4}, a key ingredient in these proofs is the cross-regularity of the diffusion. For~\eqref{eq:M-LMC}, this was proven in a somewhat circuitous way in~\cite[Lemma 6.2]{scr3}, conditioning on the probability that the interpolant time satisfies $\tau_k \leq (1-\delta) h$ for small $\delta$ and then integrating over this distribution. This causes an additional $\log \frac{1}{\beta h}$ to appear in their cross-regularity bounds. Instead, we note that Corollary~\ref{cor:ormd-error} directly provides a bound on the cross-regularity with
\begin{align*}
    \E_{x \sim \bs \mu_{(N-1)h}}[\Renyi_q(\delta_x \bs P^\alg \mmid \delta_y \bs P)] \lesssim \frac{\norm{x-y}^2}{h} + \beta^2 d h^2 q\,,
\end{align*}
and furthermore that this bound has been improved to the $\Renyi_q$ divergence.

For~\eqref{eq:DM--ULMC-multistep}, the work~\cite{scr4} considers a modified algorithm in which all but the final step uses DM--ULMC. They are unable to complete the proof using DM--ULMC alone; instead, one analyzes the composition of this algorithm with a single step of the standard, exponential Euler discretization (ULMC). On the other hand, Lemma~\ref{lem:renyi-dmd-coarse} shows that cross-regularity in Lemma~\ref{lem:local_error} holds with
\begin{align*}
    (\bar b)^2 \lesssim \frac{\beta^2 h^3}{\gamma} \Bigl(d + \Renyi_3(\bs \mu_{(N-1)h} \mmid \bs \pi)\Bigr)\,.
\end{align*}
As generally $\Renyi_3(\bs \mu_{(N-1)h} \mmid \bs \pi) \lesssim d \log \kappa$ for a reasonable initialization in our setting, then this gives a bound of
\begin{align*}
    (\bar b)^2 \lesssim \frac{\beta^2 dh^3}{\gamma} \log \kappa\,,
\end{align*}
which is sufficient for the $O(d^{1/3})$ bounds obtained in that paper. We again note that our cross-regularity is obtained in R\'enyi divergence rather than $\KL$. Since changing the algorithm in the last step is arguably unnatural, this result fills an important gap in their work.

\subsection{Sharper bounds for deterministic DM--ULMC}

For~\eqref{eq:DM--ULMC-multistep} with deterministic midpoints, we claim a completely new bound in $\KL$ divergence which has not yet appeared in the literature. First, we give the local errors under the two kernels.
\begin{lemma}[Local error analysis]\label{lem:dmd-local-error}
If $h \lesssim \frac{1}{\sqrt{\beta}}$, $\Renyi_2(\bs \mu_0 \mmid \bs \pi) \lesssim d^{1/2}$, $\gamma \asymp \sqrt{\beta}$, Assumptions~\ref{as:regularity} and~\ref{as:chaos-tail} hold, then letting $(X_t, P_t)_{t \in [0, h)}, (X_t^\alg, P_t^\alg)_{t \in [0, h)}$ be two synchronously coupled processes following~\eqref{eq:ULD} and~\eqref{eq:DM--ULMC-multistep} respectively, both started at $(x, p)$, we have the bounds
\begin{align}\label{eq:one-step-high-prob}
\begin{aligned}
     \E_{(x,p) \sim \bs\mu_0}\Bigl[h^{-2}\E[\norm{X_h - X_h^\alg}^2] \vee \E[\norm{P_h - P_h^\alg}^2]\Bigr] &\lesssim {\beta^2 \gamma d h^5} + {(\beta_H^2 + \beta_\eff^2)d h^6 } \\
    \E_{(x,p) \sim \bs\mu_0}\Bigl[h^{-2}\norm{\E X_h - \E X_h^\alg}^2 \vee \norm{\E P_h - \E P_h^\alg}^2\Bigr] &\lesssim {\beta^3 d h^6} + {(\beta_H^2 + \beta_\eff^2) d h^6}\,.
\end{aligned}
\end{align}
We also obtain analogous high probability bounds to match each of these expressions. 
\end{lemma}

With all these pieces, we can prove a bound with $\Otilde(d^{1/4})$ dependence through an elaborate two-phase argument. We again assume $\Renyi_3(\bs \mu_0 \mmid \bs \pi) \lesssim d \log \kappa$, which is reasonable in light of~\cite[\S D]{zhang2023improved}.
\begin{theorem}\label{thm:dmd-final}
    If Assumptions~\ref{as:regularity} and~\ref{as:chaos-tail} hold, $\nabla^2 V \succeq \alpha I_d$, $\alpha \geq 0$, $\gamma \asymp \sqrt{\beta}$ and $\Renyi_3(\bs \mu_0 \mmid \bs \pi) \lesssim d \log \kappa$, then if $\varepsilon$ sufficiently small and $h = \widetilde{\Theta}(\frac{\varepsilon^{1/2}}{\beta^{1/2} \kappa^{1/4} d^{1/4}} \wedge \frac{\alpha^{1/4} \varepsilon^{1/2}}{(\beta_H^2 + \beta_\eff^2)^{1/4} d^{1/4}})$, then $\KL(\bs \mu_0 \hat P^N \mmid \bs \pi) \leq \varepsilon^2$ with
    \begin{align*}
        N = \widetilde{\Theta}\Bigl(\frac{\kappa^{5/4} d^{1/4}}{\varepsilon^{1/2}} \vee \frac{(\beta_H^2 + \beta_\eff^2)^{1/4} \kappa^{1/2}  d^{1/4}}{\alpha^{3/4} \varepsilon^{1/2}} \Bigr)\,.
    \end{align*}
\end{theorem}
In particular, if we treat $(\beta_H^2 + \beta_\eff^2) \lesssim \beta^3$, then this rate is $\frac{\kappa^{5/4} d^{1/4}}{\varepsilon^{1/2}}$.

\begin{proof}
\begin{enumerate}
    \item Fix the terminal time horizon $T$ with $\frac{T}{2} \asymp \frac{\gamma}{\alpha} \log \frac{\Renyi_3(\bs \mu_0 \mmid \bs \pi)}{\varepsilon^2}$ based on the time that~\eqref{eq:ULD} needs to converge to $\varepsilon^2$-approximate stationarity, for $\varepsilon^2 \lesssim d^{1/2}$ by Lemma~\ref{lem:uld_regularity}.

    \item Now, apply Lemma~\ref{lem:renyi-dmd-coarse} with our choice of step-size and initialization $\Renyi_3(\bs \mu_0 \mmid \bs \pi) \lesssim d \log \kappa$. We have for the choice of $T$ fixed above that
    \begin{align*}
        \Renyi_8(\mathbf P \mmid \mathbf Q) \lesssim d^{1/2}\,.
    \end{align*}
    In particular, by the R\'enyi triangle inequality (Lemma~\ref{lem:renyi-triangle}), this implies that if we consider the marginal $\bs \mu_{kh}$ for any $kh > T/2$, we have $\Renyi_4(\bs \mu_{kh} \mmid \bs \pi) \lesssim d^{1/2}$.

    \item Now, considering $T/2$ as the ``starting time'' in Lemma~\ref{lem:dmd-local-error}, with the coarse guarantee $\Renyi_4(\bs \mu_{kh} \mmid \bs \pi) \lesssim d^{1/2}$ for any $k \geq \frac{T}{2h}$. For our choice of step-size, we have local errors of order
    \begin{align*}
        (\bar{\mc E}^{\mathrm s})^2 &\lesssim {\beta^2 \gamma d h^5} + {(\beta_H^2 + \beta_\eff^2) d h^6} \\
    (\bar{\mc E}^{\mathrm w})^2 &\lesssim {\beta^3 d h^6} + {(\beta_H^2 + \beta_\eff^2) d h^6}\,.
    \end{align*}
    Note that these do not hold pointwise, but only in expectation under the coarse warm-start from Lemma~\ref{lem:renyi-dmd-coarse}. Furthermore, we can bound $\bar b$ in the cross-regularity by
    \begin{align*}
        \bar b^2 \lesssim \beta^2 dh^4 + \frac{(\beta_H^2 + \beta_\eff^2) d h^5}{\gamma}\,,
    \end{align*}
    as we have via joint convexity of the $\KL$-divergence and Lemma~\ref{lem:renyi-triangle},
    \begin{align*}
        \E_{x, p \sim \bs \mu_{(N-1)h}} \Bigl[\KL(\delta_{x,p} \bs P^\alg \mmid \delta_{\tilde x, \tilde p} \bs P)\Bigr] &\lesssim \KL(\bs \mu_{(N-1)h} \bs P^\alg \mmid \bs \mu_{(N-1)h} \bs P) \\
        &\qquad+ \E_{x,p \sim \bs \mu_{(N-1)h}} \Bigl[\Renyi_2(\delta_{x, p} {\bs P} \mmid \delta_{\tilde x, \tilde p} \bs P)\Bigr]\,.
    \end{align*}
    The first term is handled
    via Theorem~\ref{thm:dmd-sharp-one-step}, while the second is handled via Lemma~\ref{lem:uld_regularity}. This is the same argument as in~\cite[Lemma 4.2]{scr4}.

    \item Finally, appealing to the local error framework in Lemma~\ref{lem:local_error}, we see that our chosen step-size is indeed sufficient to guarantee our final $\KL$ error bounds.
\end{enumerate}
\end{proof}

As far as we know, this is the first end-to-end rate which achieves $O(d^{1/4})$ dependence in $\KL$-divergence under generic assumptions, which implies $O(d^{1/4} \varepsilon^{-1/2})$ rates in the total variation metric as well. For the very specific case of ridge-separable potentials, this rate was obtained (with worse $\kappa$ dependence) in~\cite{mou2021high} using a different algorithm, but their rate does not generalize well due to their inability to exactly simulate the algorithm in other settings.

As our guarantee is in $\KL$ and not $\Renyi_q$, it is not yet sufficient to warm start the algorithm in~\cite{chen2023does} and achieve $O(d^{1/4} \operatorname{polylog} 1/\varepsilon)$ complexity. Furthermore, our results also come under a slightly stronger assumption (Assumption~\ref{as:chaos-tail}). The inability to warm start is because we need to ensure that bad events which have exponentially small probability under $\bs \pi$ should have equivalently small probabilities under the law of the output of the algorithm $\bs \mu_{Nh}$. Guarantees in say $\Renyi_2$ would be sufficient for this purpose, or one could strive to bound these tail probabilities directly. Nonetheless, we see our as a crucial first step to attaining such a warm start, and leave the further development of this for future work.

\section*{Acknowledgements}
We would like to thank Jason Altschuler and Valentin de Bortoli for their stimulating discussions, and furthermore we would like to especially thank Sinho Chewi for providing thorough guidance and feedback. MSZ was supported by NSERC through the CGS-D program. 

\newpage
\appendix
\input{app}

\printbibliography

\end{document}

%% file: _macros.tex
\usepackage{comment,url,algorithm,algorithmic,graphicx,subcaption,relsize}
\usepackage{amssymb,amsfonts,amsmath,amsthm,amscd,dsfont,mathrsfs,mathtools,microtype,nicefrac,pifont}
\usepackage{float,psfrag,epsfig,color,url,hyperref}
\usepackage{upgreek}
\usepackage[dvipsnames]{xcolor}
\usepackage{epstopdf,bbm,mathtools,enumitem}
\usepackage[toc,page]{appendix}
\usepackage[mathscr]{euscript}
\usepackage[giveninits=true, style=alphabetic]{biblatex}

\usepackage[top=1in, bottom=1in, left=1in, right=1in]{geometry}

\def\balign#1\ealign{\begin{align}#1\end{align}}
\def\baligns#1\ealigns{\begin{align*}#1\end{align*}}
\def\balignat#1\ealign{\begin{alignat}#1\end{alignat}}
\def\balignats#1\ealigns{\begin{alignat*}#1\end{alignat*}}
\def\bitemize#1\eitemize{\begin{itemize}#1\end{itemize}}
\def\benumerate#1\eenumerate{\begin{enumerate}#1\end{enumerate}}

\newenvironment{talign*}
 {\csname align*\endcsname}
 {\endalign}
\newenvironment{talign}
 {\csname align\endcsname}
 {\endalign}

\def\balignst#1\ealignst{\begin{talign*}#1\end{talign*}}
\def\balignt#1\ealignt{\begin{talign}#1\end{talign}}

\let\originalleft\left
\let\originalright\right
\renewcommand{\left}{\mathopen{}\mathclose\bgroup\originalleft}
\renewcommand{\right}{\aftergroup\egroup\originalright}

\def\tinycitep*#1{{\tiny\citep*{#1}}}
\def\tinycitealt*#1{{\tiny\citealt*{#1}}}
\def\tinycite*#1{{\tiny\cite*{#1}}}
\def\smallcitep*#1{{\scriptsize\citep*{#1}}}
\def\smallcitealt*#1{{\scriptsize\citealt*{#1}}}
\def\smallcite*#1{{\scriptsize\cite*{#1}}}

\def\mbf#1{\mathbf{#1}}

\def\<{\left\langle} 
\def\>{\right\rangle}

\newcommand{\inner}[1]{{\left\langle #1 \right\rangle}} 

\DeclareSymbolFont{rsfs}{U}{rsfs}{m}{n}
\DeclareSymbolFontAlphabet{\mathscrsfs}{rsfs}

\newcommand{\KL}{\mathsf{KL}}

\providecommand{\tr}{\mathop\mathrm{tr}}

\ifdefined\nonewproofenvironments\else
\ifdefined\ispres\else
\newtheorem{theorem}{Theorem}
\newtheorem{lemma}[theorem]{Lemma}
\newtheorem{corollary}[theorem]{Corollary}

\renewenvironment{proof}{\noindent\textbf{Proof.}\hspace*{.3em}}{\qed \vspace{.1in}}
\newenvironment{proof-sketch}{\noindent\textbf{Proof Sketch}
  \hspace*{1em}}{\qed\bigskip\\}
\newenvironment{proof-idea}{\noindent\textbf{Proof Idea}
  \hspace*{1em}}{\qed\bigskip\\}
\newenvironment{proof-of-lemma}[1][{}]{\noindent\textbf{Proof of {#1}}
  \hspace*{1em}}{\qed\\}
  \newenvironment{proof-of-proposition}[1][{}]{\noindent\textbf{Proof of Proposition {#1}}
  \hspace*{1em}}{\qed\\}
\newenvironment{proof-of-theorem}[1][{}]{\noindent\textbf{Proof of Theorem {#1}}
  \hspace*{1em}}{\qed\\}
\newenvironment{proof-attempt}{\noindent\textbf{Proof Attempt}
  \hspace*{1em}}{\qed\bigskip\\}

\newtheorem*{remark*}{Remark}
\newenvironment{remark}{\noindent\textbf{Remark.}
  \hspace*{0em}}{\smallskip}

\fi

\newtheorem{assumption}{Assumption}
\theoremstyle{definition}
\newtheorem{definition}[theorem]{Definition}
\newtheorem{example}[theorem]{Example}
\fi
\makeatletter
\@addtoreset{equation}{section}
\makeatother

\hypersetup{
  colorlinks,
  linkcolor={red!50!black},
  citecolor={blue!50!black},
  urlcolor={blue!80!black}
}

\renewcommand{\Pr}[1]{\mathbb{P}\left( #1 \right)}

\newcommand{\Otilde}{\widetilde{O}}

\newcommand{\eia}{\mc E_1}
\newcommand{\eib}{\mc E_2}
\newcommand{\eic}{\mc E_3}

\newcommand{\TV}{\msf{TV}}

%% file: app.tex
\section{Sampling results}
\subsection{Convergence in continuous time}\label{scn:cvg}
For self-containedness, we summarize the convergence of the various samplers in continuous time. These results are well-known in the literature~\cite{chewisamplingbook}. First, all the convergence results require the target measure to satisfy an isoperimetric inequality, given in Definition~\ref{def:lsi}.

\begin{lemma}[Convergence of the Langevin diffusion]\label{lem:overdamped-cvg}
    Suppose $\pi \propto \exp(-V)$ satisfies~\eqref{eq:lsi} with constant $1/\alpha$. Then, for any measure $\mu_0 \in \mc P(\R^d)$, and $\tilde \mu_t$ its evolution along the Langevin diffusion~\eqref{eq:langevin_diff} up to time $t$, with $q \geq 2$,
    \begin{align*}
        \Renyi_q(\tilde \mu_t \mmid \pi) \leq e^{-\frac{2\alpha t}{q}} \Renyi_q(\mu_0 \mmid \pi)\,.
    \end{align*}
\end{lemma}
Another result gives exponential convergence in R\'enyi under convexity conditions.
\begin{lemma}[{Harnack inequality for the underdamped Langevin dynamics, adapted from~\cite[Theorem 3.2]{scr4}}]\label{lem:uld_regularity}
    Let $\alpha I_d \preceq \nabla^2 V \preceq \beta I_d$, $\alpha \geq 0$, and let $\bs \mu_T, \bar{\bs \mu}_T$ denote the measures corresponding to running~\eqref{eq:ULD} for time $T > 0$ from starting points $(x,p)$ and $(\bar x, \bar p)$ respectively.
    Then, for all $q\ge 1$ and all $x,\bar x,p,\bar p\in\R^d$, $\gamma \gtrsim \sqrt{\beta}$,
    \begin{align*}
        \Renyi_q(\bs \mu_T \mmid \bar{\bs \mu}_T)
        &\le 
            q\Bigl\{\frac{1}{\gamma}\Bigl(\frac{\omega}{\exp(c\omega T) - 1} \Bigr)^3 +\frac{\gamma \omega}{\exp(c\omega T) - 1}\Bigr\}\,\Bigl\{\norm{x-\bar x}^2 + \frac{1}{\gamma_0^2}\,\norm{p-\bar p}^2\Bigr\}\,,
    \end{align*}
    where $\gamma_0, \omega$ are defined in Lemma~\ref{lem:local_error}. Here, $c > 0$ is a universal constant ($c=1/48$ suffices).
\end{lemma}

\subsection{Error analysis of M--LMC}

\begin{proof-of-lemma}[Corollary~\ref{cor:ormd-error}]
    First, let us argue conditional on $\ms F_{k {h}}$ (the filtration at iteration $k$). Using crude bounds, we have under $\mathbf P$,
    \begin{align*}
        \int_{(k-1)h}^{kh} \norm{X_t - X_k^+}^2 \, \D t &\lesssim \int_{(k-1)h}^{kh} \Bigl\{(t-(k-1)h)^2 \norm{\nabla V(X_k^+) - \nabla V(X_{(k-1)h})}^2 \\
        &\qquad\qquad + (t-(k-1)h - \tau_k)^2_+ \norm{\nabla V(X_k^+)}^2 \\
        &\qquad\qquad + ((k-1)h + \tau_k - t)^2_+ \norm{\nabla V(X_{(k-1)h})}^2 + \norm{B_{(k-1)h+\tau_k} - B_t}^2 \Bigr\} \, \D t\\
        &\lesssim h \sup_{t\in [0,h)} \norm{B_{(k-1)h+t} - B_{(k-1)h}}^2 + h^3 \max \{\norm{\nabla V(X_k^+)}^2, \norm{\nabla V(X_{(k-1)h})}^2\}\,.
    \end{align*}
    We can write
    \begin{align*}
        \norm{\nabla V(X_k^+)}^2 &\leq \norm{\nabla V(X_{(k-1)h})}^2 + \beta^2\norm{X_{(k-1)h} - X_k^+}^2 \\
        &\lesssim \norm{\nabla V(X_{(k-1)h})}^2 + \beta^2\Bigl(h^2\E[\norm{\nabla V(X_{(k-1)h})}^2] + \sup_{t\in [0,h)} \norm{B_{(k-1)h+t} - B_{(k-1)h}}^2\Bigr)\,.
    \end{align*}
    Then, as long as $h \lesssim \frac{1}{\beta}$, we can bound everything by
    \begin{align*}
         \int_{(k-1){h}}^{k{h}} \norm{X_t - X_k^+}^2 \, \D t \leq{h} \sup_{t\in [0,{h})} \norm{B_{(k-1){h}+t} - B_{(k-1){h}}}^2 + {h}^3 \norm{\nabla V(X_{(k-1){h}})}^2\,.
    \end{align*}
    Now, via Cauchy--Schwarz, if we consider the paths up to iteration $k$, we only need to bound for an absolute constant $C$
    \begin{align*}
        \frac{1}{q}\log \E_{\bf P} &\exp\Bigl(N\int_{(k-1){h}}^{k{h}} \beta^2 q \norm{X_t - X_k^+}^2 \, \D t \Bigr) \\
        &\lesssim \frac{1}{q}\log \E_{\bf P} \exp\Bigl(C\beta^2 hqN \sup_{t\in [0,{h})} \norm{B_{(k-1){h}+t} - B_{(k-1){h}}}^2 \Bigr) \\
        &\qquad + \frac{1}{q}\log \E_{\bf P} \Bigl[\exp\Bigl(C\beta^2{h}^3 qN\norm{\nabla V(X_{(k-1){h}}}^2 \Bigr)\Bigr]  \\
        &\lesssim \beta^2 d{h}^2qN + \frac{1}{q}\sup_{j \leq k} \log\E_{\bf P}\Bigl[\exp\Bigl(C\beta^2 {h}^3q^2N \norm{\nabla V(X_{(j-1){h}})}^2\Bigr)\Bigr]  \,,
    \end{align*}
    using the standard bounds on the supremum of a Gaussian process.
    
    We constrain $k{h} \leq N{h} \eqqcolon T$. Now, we have the bound, using Lemmas~\ref{lem:subgsn-score} and~\ref{lem:renyi-change-measure},
    \begin{align*}
        \frac{1}{q}\sup_{k \in [N]} \log\E_{\bf P}\Bigl[\exp\Bigl(C\beta^2 {h}^3 q^2 N\norm{\nabla V(X_{(j-1){h}})}^2\Bigr)\Bigr] \lesssim \beta^3 d{h}^3 q N + \beta^3 {h}^3 q N\sup_{k \in [N]} \Renyi_2(\mu_{(k-1){h}}\mmid \pi)\,.  
    \end{align*}
    It will be helpful for the recursion to assume that $q \geq 4$ is sufficiently large. As increasing $q$ from $2$ to $3$ affects our bounds by only a constant factor, this will not hinder our proof.
    This then uses the R\'enyi triangle inequality (Lemma~\ref{lem:renyi-triangle}) as well as Lemma~\ref{lem:overdamped-cvg}, to assert that
    \begin{align*}
        \Renyi_2(\mu_{N{h}}\mmid \pi) \lesssim \Renyi_{4}(\mu_{N{h}} \mmid \tilde \mu_{N{h}}) + \Renyi_{3}( \tilde \mu_{N{h}} \mmid \pi)\,.
    \end{align*}
    Thus, as $\pi$ is stationary under~\eqref{eq:langevin_diff}, via the data-processing inequality we find
    \begin{align*}
        \sup_{k \in [N]} \Renyi_2(\mu_{(k-1){h}}\mmid \pi) \lesssim \Renyi_{3}(\mu_0 \mmid \pi) +\beta^2 d{h}^2 N + \sup_{k \in [N]} \log\E_{\bf P}\Bigl[\exp\Bigl(16C \beta^2 {h}^3 N \norm{\nabla V(X_{(j-1){h}})}^2\Bigr)\Bigr]\,. 
    \end{align*}
    Defining $T \deq N{h}$, for ${h} \lesssim \frac{1}{\beta^{3/2} \sqrt{qT}}$ and $q \geq 4$, we can absorb the final term into the left side of our earlier expression for the subgaussian expectation of $\nabla V$. This leaves 
    \begin{align*}
        \frac{1}{q}\sup_{k \in [N]} \log\E_{\bf P}\Bigl[\exp\Bigl(C\beta^2 {h}^3 q^2 N \norm{\nabla V(X_{(j-1){h}})}^2\Bigr)\Bigr] \lesssim \beta^3 {h}^3 q N \Bigl(d + \Renyi_3(\mu_0 \mmid \pi)\Bigr)\,.
    \end{align*}
    Plugging this into our earlier bound concludes the proof of the second equality.
    
    The final complexity is obtained by combining $T \asymp \frac{q}{\alpha} \log \frac{\Renyi_q(\mu_0 \mmid \pi)}{\varepsilon^2}$ (with $\Renyi_q(\mu_0 \mmid \pi) \lesssim d \log \kappa$ by assumption) through Lemma~\ref{lem:overdamped-cvg}, and then choosing ${h}$ so that the process error $\Renyi_{2q}(\mu_{N{h}} \mmid \tilde \mu_{N{h}}) \lesssim \varepsilon^2$. 
    The total number of gradient queries is then obtained by taking $N = T/{h}$.
\end{proof-of-lemma}

\subsection{Error analysis of DM--ULMC}
\begin{proof-of-lemma}[{Lemma~\ref{lem:DMD-one-step-uld}}]

\paragraph{Part (A): Discretization.}
    Due to the implicit nature of our drift, we will need to again argue by approximation. First, work with a single step, taking $\lambda_1, \lambda_2 = \lambda_{0, 1}, \lambda_{0, 2}$, and $\tau^+, \tau^- = \tau_1^+, \tau_1^-$, etc.
    Impose the standard discretization onto intervals $[(k-1)\eta, k\eta)$ for $k = 1, \dotsc, h/\eta$. 
    The validity of these approximations is discussed in Lemma~\ref{lem:anticipating-SDE-existence}. Define again,
    \begin{align}\label{eq:DM--ULMC-approx}
        \hat X_{k\eta} = X_0 - \eib(0, k\eta) P_0 - \sum_{j = 0}^{k-1} \eta \eib(j\eta, k\eta) \hat\MG_{j\eta}^\opt + \sqrt{2\gamma} \sum_{j = 0}^{k-1} \eib(j\eta, k\eta) (B_{(j+1)\eta} - B_{j\eta})\,, 
    \end{align}
    and, define $\hat \MG_t^\opt$ similarly to $\MG_t^\opt$, but with parameters $\hat \lambda_1, \hat \lambda_2$ given by
    \begin{align*}
        \hat \lambda_1 &= \frac{\sigma_{22}\Bigl[\sum_{j = 0}^{ h/\eta-1} \eta \eia(j\eta, h) \nabla V(\hat X_{j\eta}) -\hat \MG^\msp \Bigr] - \sigma_{21} \Bigl[\sum_{j = 0}^{ h/\eta-1} \eib(j\eta, h) \nabla V(\hat X_{j\eta}) - \hat \MG^\msx \Bigr]}{\Delta_\sigma}\,,
    \end{align*}
    and similarly for $\hat \lambda_2$, where we assume that $\eta$ evenly divides $\tau_1^+, \tau_1^-$. We define
    \begin{align*}
        \hat \MG^\msp \deq \eib(0, h) \nabla V(\hat X^+)\,,\qquad \hat \MG^\msx \deq \eic(0, h) \nabla V(\hat X^-)\,,
    \end{align*}
    only with $\hat X^+$ defined as
    \begin{align*}
        \hat X^+ &= X_0 + \eib(0, \tau_1^+) P_0 -  \eic(0, \tau_1^+) \nabla V(X_0) + \sqrt{2\gamma} \sum_{k=1}^{\tau_1^+/\eta} \eib((k-1)\eta, \tau_1^+) (B_{k\eta} - B_{(k-1)\eta})\,.
    \end{align*}
    $\hat X^-$ is similarly defined but with $\tau_k^-$ in place of $\tau_k^+$. This allows us to write for $i \in \{1, \dotsc, h/\eta\}$,
    \begin{align}\label{eq:malliavin-ulmc}
    \begin{aligned}
        \MD_i \hat X_{j\eta} &= \sqrt{2\gamma} \eta \eib(i\eta, j\eta) \one_{i < j} I_d - \eta \sum_{j' =0}^{j-1} \eib(j'\eta, j\eta) \MD_i \hat \MG_{j'\eta}^\opt\,,
    \end{aligned}
    \end{align}
    and
    \begin{align*}
        \MD_i \nabla V(\hat X_{j\eta}) &= \nabla^2 V(\hat X_{j\eta}) \MD_i \hat X_{j\eta} = \nabla^2 V(\hat X_{j\eta}) \cdot \Bigl[\sqrt{2\gamma} \eta \eib(i\eta, j\eta) \one_{i < j} - \eta \sum_{j' =0}^{j-1} \eib(j'\eta, j\eta) \MD_i \hat \MG_{j'\eta}^\opt \Bigr]\,.
    \end{align*}
    Furthermore,
    \begin{align*}
        \MD_i \hat \lambda_1 \!&=\! \frac{\sigma_{22}}{\Delta_\sigma} \Bigl[\eta \!\!\sum_{j = 0}^{ h/\eta-1} \eia(j\eta, h)\nabla^2 V(\hat X_{j\eta}) \MD_i \hat X_{j\eta} \!-\! \MD_i \hat \MG^\msp \Bigr]\! -\! \frac{\sigma_{21}}{\Delta_\sigma} \Bigl[\eta \!\!\sum_{j = 0}^{ h/\eta-1} \eib(j\eta, h)\nabla^2 V(\hat X_{j\eta}) \MD_i \hat X_{j\eta} \!-\! \MD_i \hat \MG^\msx \Bigr]\,.
    \end{align*}
    $\MD_i \hat \lambda_2$ is similarly defined, with $-\sigma_{12}$ in place of $\sigma_{22}$ and $\sigma_{11}$ in place of $\sigma_{21}$. Finally,
    \begin{align*}
        \MD_i \hat \MG^\msx = \sqrt{2\gamma} \eic(0, h) \nabla^2 V(\hat X^-) \eib(i\eta, \tau_1^-) \cdot \eta \one_{i < \tau_1^-/\eta}\,, \\
        \MD_i \hat \MG^\msp = \sqrt{2\gamma}\eib(0, h) \nabla^2 V(\hat X^+)  \eib(i\eta, \tau_1^+) \cdot  \eta \one_{i < \tau_1^+/\eta}\,.
    \end{align*}

    \paragraph{Part (B): Computing the discretized quantities.}
    Given the above, the proof proceeds as follows. We first show that the most troublesome term ($\MD \hat \MG^\opt$) in the expansion can be ignored. We can observe that $\norm{\MD \hat \MG^\msx} + h\norm{\MD \hat \MG^\msp}   \lesssim \sqrt{\gamma} \beta h^3$.
    \begin{enumerate}[label=(\roman*)]
        \item Noting that $\frac{\sigma_{22}}{\Delta_\sigma} \asymp \frac{1}{h}$, $\frac{\sigma_{21}}{\Delta_\sigma} \asymp \frac{1}{h^2}$, $\frac{\sigma_{11}}{\Delta_\sigma} \asymp \frac{1}{h^3}$, we can observe from the above that
        \begin{align*}
            \norm{\MD \hat \lambda_1} \lesssim \beta \norm{\MD \hat X} + \norm{\MD \hat \MG^\msp} + \frac{1}{h} \norm{\MD \hat \MG^\msx} \lesssim \beta \norm{\MD \hat X}+ \sqrt{\gamma} \beta h^2\,.
        \end{align*}
        Likewise, $h\norm{\MD \hat \lambda_2} \lesssim \beta \norm{\MD \hat X} + \sqrt{\gamma} \beta h^2$.

        \item On the other hand, we can read from~\eqref{eq:malliavin-ulmc} that
        \begin{align*}
            \MD \hat X = A^{(1)} - A^{(2)} \MD \hat \MG^\opt\,, 
        \end{align*}
        where $\norm{A^{(1)}} \lesssim \sqrt{\gamma} h^2$, $\norm{A^{(2)}} \lesssim h^2$, and so
        \begin{align*}
            \MD \nabla V(\hat X) = A^{(3)} - A^{(4)} \MD \hat \MG^\opt
        \end{align*}
        with $\norm{A^{(3)}} \lesssim \sqrt{\gamma} \beta h^2$, $\norm{A^{(4)}} \lesssim \beta h^2$.

        \item Putting this all together, we use that
        \begin{align*}
            \norm{\MD \hat \MG} \lesssim  \norm{\MD \nabla V(\hat X)} +  \norm{\MD \hat \lambda_1} + h\norm{\MD \hat \lambda_2}\,.
        \end{align*}
        First, we can stitch all these terms into an implicit equation, which has a solution for $h \lesssim \frac{1}{\sqrt{\beta}}$ sufficiently small. With a sufficiently small implied constant above, we can conclude that $\norm{\MD \hat \MG^\opt} \lesssim \sqrt{\gamma} \beta h^2$.

        \item Note that in the equations for $\norm{\MD \hat \lambda_1}$, the $\norm{\MD \hat \MG^\opt}$ term already appears with a prefactor of $\beta h^2$ from $\beta A^{(2)}$. Thus, it will be negligible when compared to the leading terms arising from $\beta A^{(1)}$ and $\norm{\MD \hat \MG^\msp}$, if again $h \lesssim \frac{1}{\sqrt{\beta}}$, and we can always ignore it up to a cost of $\sqrt{\gamma} \beta^2 h^4$.
    \end{enumerate}
    
    In discrete terms, defining the temporal vectors $\eia, \eib \in \R^{h/\eta}$ with entries
    \begin{align*}
        (\eia)_j = \eia((j-1)\eta, h)\,, \qquad (\eib)_j = \eib((j-1)\eta, h)\,,
    \end{align*}
    then we can express
    \begin{align*}
        \MD \Mu^\eta = -\frac{1}{\sqrt{2\gamma}}\Bigl( \MD \hat \lambda_1 \eia^\top + \MD \hat \lambda_2 \eib^\top \Bigr)
    \end{align*}
    where kernels corresponding to $\MD \hat \lambda_1, \MD \hat \lambda_2$ are in $3$-tensors inhabiting $\R^{d\times d \times h/\eta}$.
    Using our expressions for $\MD \hat \lambda_i$ and ignoring $\MD \hat \MG^\opt$, which we have argued is not the leading order term, we can now obtain both the Malliavin derivative matrix and the Skorohod adjoint. 

    \paragraph{Leading asymptotics:}
    If we are only interested in the $\KL$ or $\Renyi_q$ divergences, then we do not need to be concerned with the Skorohod adjoint, as it will always be removed in the R\'enyi computations using Cauchy--Schwarz via Corollary~\ref{cor:doleans-dade}. 
    
    We compute the determinant and trace of the derivative matrix. Recalling our familiar formula for the Carleman--Fredholm determinant,
    \begin{align*}
        \log \Mdet(I + q\MD \bar \Mu) &= \lim_{\eta \searrow 0} \log \det \Bigl(I -\frac{q}{\sqrt{2\gamma}} (\MD \hat \lambda_1^\eta (\eia^\eta)^\top + \MD \hat \lambda_2^\eta (\eib^\eta)^\top)\Bigr) \\
        &\qquad\qquad\qquad +\frac{q}{\sqrt{2\gamma}}\tr (\MD \hat \lambda_1^\eta (\eia^\eta)^\top + \MD \hat \lambda_2^\eta (\eib^\eta)^\top)\,.
    \end{align*}
    Omit for the moment the superscript in $\eta$.
    As $\norm{\nabla^2 V}\lesssim \beta$, we see that $\frac{1}{\sqrt{\gamma}} (\MD \hat \lambda_1 \eia^\top + \MD \hat \lambda_2 \eib^\top)$ has eigenvalues on the order of $O(\beta h^2)$. More precisely, we had inferred that
    \begin{align*}
        \norm{\MD \hat \lambda_1 \eia^\top} \lesssim \beta \norm{\MD \hat X} + \sqrt{\gamma} \beta h^2 \lesssim \sqrt{\gamma} \beta h^2 + O(\sqrt{\gamma} \beta^2 h^4)\,,
    \end{align*}
    and a similar bound holds for $\norm{\MD \hat \lambda_2 \eib^\top}$.
    
    Expanding the log-determinant of the matrix in terms of trace of its powers, it suffices to compute
    \begin{align*}
        \log \Mdet(I + q\MD \Mu^\eta) = -\frac{q^2}{4\gamma} \tr\Bigl((\MD \hat \lambda_1 \eia^\top + \MD \hat \lambda_2 \eib^\top)^2 \Bigr) + O\Bigl(\beta^3 d h^6 q^3\Bigr)\,.
    \end{align*}
    By cyclicity of the trace, we note that we can instead write
    \begin{align*}
        \tr\Bigl((\MD \hat \lambda_1 \eia^\top + \MD \hat \lambda_2 \eib^\top)^2\Bigr) = \tr\Bigl((\eia^\top \MD \hat \lambda_1 + \eib^\top \MD \hat \lambda_2)^2\Bigr)\lesssim \beta^2 dh^4 q^2\,,
    \end{align*} 
    by summing over the $d$ possible eigenvalues. 

    \paragraph{Part (C): Computing the limiting quantities.}
    It remains to take the limit as $\eta \searrow 0$ to conclude. This is fairly trivial as the bounds above are $\eta$-free. 
    Finally, the multistep bound follows from summing the determinants in each iteration, noting that the total Malliavin derivative matrix will have the same lower triangular form as in Lemma~\ref{lem:MD-multistep}. We also provide some more detailed computations below.

    \paragraph{Detailed computations:}
    Now, we provide some detailed descriptions for the leading terms of the Skorohod adjoint and Malliavin derivative matrices, in case they may be needed. First, for the Skorohod adjoint, we have
    \begin{align*}
        \Mdel \bar \Mu &=  -\frac{1}{\sqrt{\gamma}} \int_0^h \inner{\lambda_1 \eia(t, h) + \lambda_2 \eib(t, h), \D B_t} + \lim_{\eta \to 0} \tr \MD \Mu^\eta\,.
    \end{align*}
    This trace is
    \begin{align*}
        \tr \MD \Mu^\eta = -\tr\Bigl(\mathrm{I} + \mathrm{II} + \mathrm{III} + \mathrm{IV}
        \Bigr) + O(\beta^2 dh^4)\,,
    \end{align*}
    where we separate into terms as the computation is somewhat unwieldy, with
    \begin{align*}
        \mathrm{I} &\deq \sum_{i=0}^{h/\eta - 1} \frac{\sigma_{22} \eia(i\eta, h)}{\Delta_\sigma} \biggl[  \eta^2 \sum_{j=0}^{h/\eta - 1} \eia(j\eta, h) \nabla^2 V(\hat X_{j\eta})\eib(i\eta, j\eta) \one_{i < j} \\
        &\qquad\qquad\qquad\qquad\qquad\qquad\qquad - \eta\eib(0,h) \nabla^2 V(X^+) \eib(i\eta, \tau_1^+) \one_{i < \tau_1^+/\eta} \biggr] \\
        &\to  \frac{\sigma_{22}}{\Delta_\sigma}\biggl[\int_0^h \int_0^t \eia(s, h)\eia(t,h)\eib(s,t) \nabla^2 V(X_t)  \, \D s \, \D t \\
        &\qquad\qquad\qquad\qquad\qquad\qquad\qquad- \int_0^{\tau_1^+} \eia(s, h) \eib(0, h)  \eib(s, \tau_1^+) \nabla^2 V(X^+) \, \D s \, \biggr]\,.
    \end{align*}
    For the remaining terms, we only give their limiting forms, with the discrete versions obtained in a very similar way:
    \begin{align*}
        \mathrm{II}
        &\to  -\frac{\sigma_{21}}{\Delta_\sigma}\biggl[\int_0^h \int_0^t {\eia(s, h)\eib(t,h)\eib(s,t)}\nabla^2 V(X_t)  \, \D s \, \D t  \\
        &\qquad \qquad \qquad - \int_0^{\tau_1^-} {\eia(s, h) \eic(0, h)  \eib(s, \tau_1^-)} \nabla^2 V(X^-) \, \D s \, \biggr]\,.\\
        \mathrm{III}
        &\to  -\frac{\sigma_{12}}{\Delta_\sigma}\biggl[\int_0^h \int_0^t {\eib(s, h)\eia(t,h)\eib(s,t)} \nabla^2 V(X_t)  \, \D s \, \D t  \\
        &\qquad \qquad \qquad - \int_0^{\tau_1^+} {\eib(s, h) \eib(0, h)  \eib(s, \tau_1^+)} \nabla^2 V(X^+) \, \D s \, \biggr]\,.\\
        \mathrm{IV}
        &\to \frac{\sigma_{11}}{\Delta_\sigma}\biggl[\int_0^h \int_0^t {\eib(s, h)\eib(t,h)\eib(s,t)} \nabla^2 V(X_t)  \, \D s \, \D t  \\
        &\qquad \qquad \qquad - \int_0^{\tau_1^-} {\eib(s, h) \eic(0, h)  \eib(s, \tau_1^-)} \nabla^2 V(X^-) \, \D s \, \biggr]\,.
    \end{align*}
    Thus, we can compute now
    \begin{align*}
        \frac{1}{\sqrt{\gamma}}\eia^\top \MD \hat \lambda_1 = \mathrm{I} + \mathrm{II}\,, \qquad \frac{1}{\sqrt{\gamma}}\eia^\top \MD \hat \lambda_2 = \mathrm{III} + \mathrm{IV}\,,
    \end{align*}
    where the terms $\mathrm{I}, \dotsc, \mathrm{IV}$ are precisely as defined earlier. This can be used to obtain a more refined expression for the $\log \Mdet (I + q\MD \Mu)$.
\end{proof-of-lemma}

\begin{proof-of-lemma}[Lemma~\ref{lem:renyi-dmd-coarse}]
For simplicity, it suffices to take $k = 0$ and work with a single step. For this, we need only bound, with $(\lambda_1, \lambda_2) \deq (\lambda_{0,1}, \lambda_{0,2})$, and $\Lambda^2 \deq \norm{\lambda_1}^2 + {h}^2 \norm{\lambda_2}^2$
\begin{align*}
    \int_0^{h} \norm{\eia(t, {h}) \lambda_1 + \eib(t, {h}) \lambda_2}^2 \, \D t 
    \lesssim h \Lambda^2 &\lesssim \Bigl(\frac{{h} \sigma_{22}^2}{\Delta_\sigma^2} + \frac{{h}^3 \sigma_{12}^2}{\Delta_\sigma^2} \Bigr) \norm{\int_0^{h} \eia(t, {h}) \nabla V(X_t)\, \D t - \MG^\msp_0}^2  \\
    &\qquad + \Bigl(\frac{{h} \sigma_{21}^2}{\Delta_\sigma^2} + \frac{{h}^3 \sigma_{11}^2}{\Delta_\sigma^2} \Bigr)\norm{\int_0^{h} \eib(t, {h}) \nabla V(X_t)\, \D t - \MG^\msx_0}^2\,,
\end{align*}
noting that these inequalities hold almost surely. For ${h} \lesssim 1$, it follows that $\sigma_{22} \lesssim {h}^3$, $\sigma_{12} = \sigma_{21} \lesssim {h}^2$, $\sigma_{11} \lesssim {h}$ and $\Delta_\sigma \gtrsim {h}^4$, so that
\begin{align*}
    \int_0^{h} \norm{\eia(t, {h}) \lambda_1 + \eib(t, {h}) \lambda_2}^2 \, \D t &\lesssim \frac{1}{{h}} \norm{\int_0^{h} \eia(t, {h}) \nabla V(X_t)\, \D t -\MG^\msp_0}^2 \\
    &\qquad +\frac{1}{{h}^3} \norm{\int_0^{h} \eib(t, {h}) \nabla V(X_t)\, \D t - \MG^\msx_0}^2 \,.
\end{align*}
We call the first term the momentum error and the second term the position error. 

\paragraph{Crude analysis:}
We begin with the weaker bound which does not leverage higher order smoothness. Conditional on our initial point $(X_0, P_0)$,
\begin{align*}
    \norm{\int_0^{h} \eia(t, {h}) \nabla V(X_t)\, \D t - \MG^\msp_0}^2 &\lesssim \beta^2 {h}^2 \Bigl\{\sup_{t \in [0, {h})} \norm{X_t - X_{\tau_1^+}}^2  +\norm{X_{\tau_1^+} - X^+}^2 \Bigr\}\,.
\end{align*}
For the first term, we use $\sup_{t \in [0, h)} \norm{\MG_t^\opt}^2 \lesssim \sup_{t \in [0, h)} \norm{\nabla V(X_t)}^2 + \Lambda^2$,
\begin{align*}
    \beta^2 &{h}^2\sup_{t \in [0, {h})} \norm{X_t - X_{\tau_1^+}}^2\lesssim {\beta^2 {h}^6} \sup_{t \in [0, {h})} \norm{\MG_t^\opt}^2 + {\beta^2 {h}^4} \norm{P_0}^2 + {\beta^2 \gamma  {h}^4} \sup_{t\in[0, {h})} \norm{B_t}^2 \\
    &\lesssim {\beta^2 {h}^6} \sup_{t \in [0, {h})} \norm{\nabla V(X_t)}^2 + {\beta^2 {h}^6 \Lambda^2} +{\beta^2 {h}^4} \norm{P_0}^2 + { \beta^2 \gamma {h}^4}\sup_{t\in[0, {h})} \norm{B_t}^2\,.
\end{align*}
For the gradient term, we have
\begin{align*}
    \norm{\nabla V(X_t)}^2 &\lesssim \norm{\nabla V(X_0)}^2 + \beta^2 \norm{X_t - X_0}^2 \\
    &\lesssim \norm{\nabla V(X_0)}^2 + {\beta^2 {h}^4} \sup_{t \in [0, {h})}\norm{\nabla V(X_t)}^2  +{\beta^2 {h}^4  \Lambda^2}  +{\beta^2 {h}^2} \norm{P_0}^2 + {\beta^2 \gamma {h}^2} \sup_{t\in[0, {h})} \norm{B_t}^2\,.
\end{align*}
Taking supremum over $t \in [0, {h})$, this shows after rearranging, if ${h} \lesssim \frac{1}{\sqrt{\beta}}$, that we have
\begin{align*}
    \beta^2 &{h}^2\sup_{s \in [0, {h})} \norm{X_t - X_{\tau_1^+}}^2 \lesssim {\beta^2 {h}^6} \norm{\nabla V(X_0)}^2 + {\beta^2 {h}^6  \Lambda^2} +{\beta^2 {h}^4} \norm{P_0}^2 + {\beta^2 \gamma {h}^4} \sup_{t\in[0, {h})} \norm{B_t}^2\,.
\end{align*}
For the other term, we have
\begin{align*}
    \beta^2 {h}^2 \norm{X_{\tau_1^+} - X^+}^2
    &\lesssim {\beta^2 {h}^2} \norm{\int_0^{\tau_1^+}\int_0^t \eia(s, t)\, \MG_s^\opt \, \D s \, \D t - \eic(0, \tau_1^+) \nabla V(X_0)}^2 \\
    &\lesssim {\beta^2 {h}^6 \Lambda^2}  + {\beta^2{h}^2}\norm{\int_0^{\tau_1^+} \int_0^t \eia(s,t) \{\nabla V(X_s) - \nabla V(X_0)\} \, \D s \, \D t}^2\,.
\end{align*}
We can then bound
\begin{align*}
    \norm{\int_0^{{h}/2} &\int_0^t \eia(s,t) \{\nabla V(X_s) - \nabla V(X_0)\} \, \D s \, \D t}^2 \\
    &\lesssim \beta^2 {h}^4 \sup_{t \in [0, {h})}  \norm{X_s - X_0}^2 \lesssim {\beta^2 {h}^8} \sup_{t \in [0, {h})} \norm{\MG_t^\opt}^2  + \beta^2 {h}^6 \norm{P_0}^2 + \beta^2 \gamma {h}^6 \sup_{t \in [0, {h})} \norm{B_t}^2  \,.
\end{align*}
Combining this with our earlier bounds on $\norm{\MG_t^\opt}^2$, we have for $h \lesssim \frac{1}{\beta^{1/2}}$,
\begin{align}\label{eq:same-time-interpolant}
\begin{aligned}
    \beta^2 {h}^2 \norm{X_{\tau_1^+} - X^+}^2 \lesssim {\beta^2 {h}^6  \Lambda^2}  + {\beta^4{h}^{10}} \norm{\nabla V(X_0)}^2 + {\beta^4 {h}^8} \norm{P_0}^2 + {\beta^4 \gamma  {h}^8} \sup_{t \in [0, {h})} \norm{B_t}^2\,.
\end{aligned}
\end{align}
We see that the bound on $\norm{X_{\tau_1^+} - X^+}^2$ will yield only higher order terms (in ${h}$) relative to the bounds we obtained for $\norm{X_t - X_{\tau_1^+}}^2$. Putting these two terms together for ${h} \lesssim \frac{1}{\sqrt{\beta}}$ gives
\begin{align*}
    \norm{\int_0^{h} &\eia(t, {h}) \nabla V(X_t)\, \D t - \MG^\msp_0}^2 \lesssim {\beta^2 {h}^4} \Bigl\{{{h}^2 \norm{\nabla V(X_0)}^2} + {{h}^2\Lambda^2} + {\norm{P_0}^2} + \gamma \sup_{t \in [0, {h})} \norm{B_t}^2\Bigr\}\,.
\end{align*}
The position error is handled in almost the exact same way, but gains two orders of ${h}$ due to $\eib(t,{h})$ being present in place of $\eia(t, {h})$,
\begin{align*}
    \norm{\int_0^{h} \eib(t, {h}) \nabla V(X_t)\, \D t - \MG^\msx_0}^2 &\lesssim \beta^2 {h}^4 \Bigl\{\sup_{t \in [0, {h})} \norm{X_t - X_{\tau_1^-}}^2  +\norm{X_{\tau_1^-} - X^-}^2 \Bigr\}\,,
\end{align*}
and both of these terms are handled in the exact same way.
Thus, we can bound
\begin{align*}
    \Lambda^2 \lesssim {\beta^2 {h}^4} \norm{\nabla V(X_0)}^2 + {\beta^2 {h}^2} \norm{P_0}^2 + {\beta^2 \gamma {h}^2} \sup_{t \in [0, {h})} \norm{B_t}^2\,.
\end{align*}
Substituting this into our Girsanov bound and using the standard supremum bounds on the Brownian motion (Lemma~\ref{lem:brownian-subgaussian}) tells us that, using AM--GM (or equivalent Cauchy--Schwarz), for some absolute constant $C > 0 $,
\begin{align}\label{eq:dmd-one-step-renyi-bd}
\begin{aligned}
    \log &\exp\Bigl( \int_0^{h} \frac{q^2}{\gamma} \norm{\eia(t, {h}) \lambda_1 + \eib(t, {h}) \lambda_2}^2 \, \D t\Bigr) \\
    &\lesssim \log \E_{\mathbf P} \exp\Bigl(\frac{C \beta^2 {h}^5 q^2}{\gamma} \norm{\nabla V(X_0)}^2 \Bigr) +  \log \E_{\mathbf P} \exp\Bigl(\frac{C \beta^2 {h}^3 q^2}{\gamma} \norm{P_0}^2 \Bigr) + {\beta^2 d{h}^4 q^2 }\,.
\end{aligned}
\end{align}

\paragraph{Multistep bounds}
Using~\eqref{eq:dmd-one-step-renyi-bd}, we obtain the equation for the full paths over $N$ iterations, using Lemmas~\ref{lem:subgsn-score} and~\ref{lem:renyi-change-measure} (as well as the subgaussianity of $\norm{P}$ under $\pi$),
\begin{align*}
\begin{aligned}
    \Renyi_q(\mathbf P \mmid \mathbf Q)
    &\lesssim \frac{1}{q} \sup_{k \in [N]} \log \E_{\mathbf P} \exp\Bigl(\frac{C \beta^2 {h}^5 q^2 N}{\gamma} \norm{\nabla V(X_{k{h}})}^2 \Bigr) \\
    &\qquad + \frac{1}{q} \sup_{k \in [N]} \log \E_{\mathbf P} \exp\Bigl(\frac{C \beta^2 {h}^3 q^2 N}{\gamma} \norm{P_{k{h}}}^2 \Bigr) + \beta^2 d {h}^4 q N \\
    &\lesssim \frac{\beta^3 {h}^5 q N}{\gamma}\Bigl(d + \sup_{k \in [N]} \Renyi_{2}(\bs \mu_{k{h}} \mmid \bs \pi)  \Bigr) + \frac{ \beta^2 {h}^3 q N}{\gamma}\Bigl(d + \sup_{k \in [N]} \Renyi_{2}(\bs \mu_{k{h}} \mmid \bs \pi) \Bigr) + {\beta^2 {h}^4 d q N}\,.
\end{aligned}
\end{align*}
As we can bound, using Lemma~\ref{lem:renyi-triangle}, the stationarity of $\bs \pi$ under $\mathbf Q$ and the data-processing inequality,
\begin{align*}
    \Renyi_2(\bs \mu_t \mmid \bs \pi) &\lesssim \Renyi_4(\bs \mu_t \mmid \tilde{\bs \mu}_t) + \Renyi_3(\tilde {\bs \mu}_t \mmid \bs \pi) \lesssim \Renyi_4(\mathbf P \mmid \mathbf Q) + \Renyi_3(\bs \mu_0 \mmid \bs \pi)\,,
\end{align*}
Again, if we assume $q \geq 4$, and ${h} \leq \frac{1}{\beta^{1/2}} \wedge \frac{\gamma^{1/2}}{\beta q^{1/2} T^{1/2}}$, then we can absorb the first term $\Renyi_4(\mathbf P \mmid \mathbf Q)$ into the left side of our bound. This then allows us to bound, noting that the $\beta^2 d h^4 q N$ term will now be higher order,
\begin{align*}
    \Renyi_q(\mathbf P \mmid \mathbf Q) \lesssim \frac{\beta^2 {h}^3 q N}{\gamma} \Bigl(d + \Renyi_3(\bs \mu_0 \mmid \bs \pi) \Bigr)\,.
\end{align*}

\end{proof-of-lemma}

\begin{proof-of-lemma}[Lemma~\ref{lem:dmd-local-error}]
Throughout the present proof, we assume that $\delta \in (0, \frac{1}{c})$ for some absolute constant $c$. Recall that if $(X_t)_{t \in [0, {h})}$ is obtained from~\eqref{eq:ULD} after time ${h}$, and $X^\alg_{h}$ is obtained from a single step of~\eqref{eq:DM--ULMC-multistep}, then 
\begin{align*}
    \norm{X_{h} - X_{h}^\alg}^2 &\lesssim \norm{\int_0^{h} \eib(t, {h}) \{\nabla V(X^-) - \nabla V(X_t)\} \, \D t}^2\,, \\
    \norm{P_{h} - P_{h}^\alg}^2 &\lesssim \norm{\int_0^{h} \eia(t, {h}) \{\nabla V(X^+) - \nabla V(X_t)\} \, \D t}^2\,,
\end{align*}
Here, we note that $X^-$ and $X^+$ can be determined using the same initial point and Brownian motion. We refer to the first term as the position error, and the second as the momentum error, opting to handle them separately.
\paragraph{Momentum error:}
We first expand once using It\^o--Taylor,
\begin{align}\label{eq:expansion-momentum}
    \int_{0}^{h} &\eia(t,{h}) \nabla V(X_t) \, \D t = \int_0^{{h}} \eia(t, {h}) \nabla V(X_{{h}/2}) \, \D t + \int_0^{h} \int_{{h}/2}^t \eia(t,{h}) \nabla^2 V(X_s) P_s \, \D s \, \D t\,.
\end{align}
Now, we have,
\begin{align*}
    \norm{\int_0^{h} &\int_{{h}/2}^t \eia(t,{h}) \nabla^2 V(X_s) P_s \, \D s \, \D t}^2 \\
    &\lesssim \norm{\int_{0}^{h} \int_{{h}/2}^t \nabla^2 V(X_s) P_s \, \D s \, \D t}^2  + \norm{\int_{0}^{h} \int_{{h}/2}^t (\eia(t, {h}) - 1)\nabla^2 V(X_s) P_s \, \D s \, \D t}^2\,.
\end{align*}
We can bound, noting that $\abs{\eia(t, {h}) - 1} \lesssim {h}$, and $\norm{\nabla^2 V} \lesssim \beta$, 
\begin{align*}
    \norm{\int_{0}^{h} \int_{{h}/2}^t (\eia(t, {h}) - 1)\nabla^2 V(X_s) P_s \, \D s \, \D t}^2 &\lesssim {h}^2 \int_0^{h} \int_{{h}/2}^t \norm{(\eia(t, {h}) - 1)\nabla^2 V(X_s) P_s}^2 \, \D s \, \D t \\
    &\lesssim \beta^2 {h}^4 \int_0^{h} \int_{{h}/2}^t \norm{P_s}^2 \, \D s \, \D t\,. 
\end{align*}
With probability $1-\delta$, we know that if $\bs \mu_0$ has $\Renyi_q(\bs \mu_0 \mmid \bs \pi) \lesssim d^{1/2}$ for $q \geq 2$, then under $\tilde{\bs \mu}_t$ for all $t \geq 0$, we have with probability $1-\delta$ via Lemma~\ref{lem:high-prob-change},
\begin{align*}
    \beta^2 {h}^4 \int_0^{h} \int_{{h}/2}^t \norm{P_s}^2 \, \D s \, \D t \lesssim \beta^2 d{h}^6 + \beta^2 {h}^6 \log \frac{1}{\delta}\,.
\end{align*}
By symmetrizing the integral, we find for the other term,
\begin{align*}
    \norm{\int_{0}^{h} \int_{{h}/2}^t \nabla^2 V(X_s) P_s \, \D s \, \D t}^2
    &=\norm{\int_{{h}/2}^{h} \int_{{h}/2}^t \{\nabla^2 V(X_s) P_s - \nabla^2 V(X_{{h}-s}) P_{{h}-s}\} \, \D s \, \D t}^2\\
    &= \norm{\int_{{h}/2}^{h} \int_{{h}/2}^t \int_{{h} -s}^s \,\D \{ \nabla^2 V(X_r) P_r\} \, \D s \, \D t}^2 \\
    &\lesssim {h}^2 \int_{{h}/2}^{h} \int_{{h}/2}^t \norm{\int_{{h} -s}^s  \,\D \{ \nabla^2 V(X_r) P_r\}}^2  \, \D s \, \D t  \,. 
\end{align*}
Now, using It\^o's lemma applied to $\nabla^2 V(X_s) P_s$ under~\eqref{eq:ULD},
\begin{align*}
    \norm{\int_{{h} -s}^s\,\D \{ \nabla^2 V(X_s) P_s\}}^2 &= \norm{\int_{{h} -s}^s\{\nabla^3 V(X_r)[P_r, P_r] \,\D r + \nabla^2 V(X_r) \, \D P_r\}}^2 \\
    &\lesssim {{h}}\int_{{h}-s}^{s}\norm{\nabla^3 V(X_r)[P_r, P_r]}^2 \,\D r  + \gamma^2 {h}  \int_{{h}-s}^{s} \norm{\nabla^2 V(X_r)P_r}^2 \,\D r \\
    &\qquad + {{h}} \int_{{h}-s}^{s}\norm{\nabla^2 V(X_r)  \nabla V(X_r)}^2 \, \D r + \gamma \norm{\int_{{h}-s}^{s}\nabla^2 V(X_r)\,\D B_r}^2 \\
    &\eqqcolon \mathrm{I} + \mathrm{II} + \mathrm{III} + \mathrm{IV}\,.
\end{align*}
Here, $B_t$ is a standard Brownian motion. For $\mathrm{I}$, we use Assumption~\ref{as:chaos-tail} to claim that under $\bs \pi$, where $P$ is a standard Gaussian uncorrelated from $X$, with probability $1-\delta$,
\begin{align*}
     \norm{\nabla^3 V(X)[P, P]}^2\lesssim \beta_H^2 d + \beta_\eff^2 d^{1/2} \log \frac{1}{\delta} + \beta_\eff^2 \log^2 \frac{1}{\delta}\,.
\end{align*}
Using Lemma~\ref{lem:high-prob-change}, under our assumption on the initial R\'enyi divergence, with probability $1-\delta$ we have
\begin{align*}
    \norm{\nabla^3 V(X)[P, P]}^2\lesssim \beta_H^2 d + \beta_\eff^2 d + \beta_\eff^2 d^{1/2} \log \frac{1}{\delta} + \beta_\eff^2 \log^2 \frac{1}{\delta}\,,
\end{align*}
Note that this implies after integration that, uniformly in $s$,
\begin{align*}
    \E_{\mathbf Q}[\mathrm{I}] \lesssim  (\beta_\eff^2 + \beta_H^2)dh^2 \,.
\end{align*}
We can immediately bound using $\norm{\nabla^2 V} \lesssim \beta$, the subgaussianity of $\norm{P}$ under $\bs \pi$, and Lemma~\ref{lem:high-prob-change} with the same initial condition on the R\'enyi, with probability $1-\delta$,
\begin{align*}
    \mathrm{II} \lesssim \beta^2 \gamma^2 d {h}^2 + \beta^2 \gamma^2 {h}^2 \log \frac{1}{\delta}\,.
\end{align*}
Furthermore, Lemmas~\ref{lem:subgsn-score} and~\ref{lem:high-prob-change} give a similar bound for the third term of
\begin{align*}
    \mathrm{III} \lesssim {\beta^3 d {h}^2 } + {\beta^3 {h}^2 \log \frac{1}{\delta}} \,.
\end{align*}
For the final term $\mathrm{IV}$, Lemma~\ref{lem:brownian-subgaussian} applied directly yields a bound with probability $1-\delta$ of
\begin{align*}
    \mathrm{IV} \lesssim \beta^2\gamma  d {h} + \beta^2 \gamma {h} \log \frac{1}{\delta}\,.
\end{align*}
Putting all of this together, we find that if ${h} \lesssim \frac{1}{\sqrt{\beta}}$, $\gamma \asymp \sqrt{\beta}$, with probability $1-\delta$,
\begin{align}\label{eq:momentum-expansion-high-prob}
\begin{aligned}
    \norm{\int_0^{h} \int_{{h}/2}^t \eia(t,{h}) \nabla^2 V(X_s) P_s \, \D s \, \D t}^2 &\lesssim \beta^2 \gamma d {h}^5 + {(\beta_H^2 + \beta_\eff^2) d {h}^6} \\
    &\qquad+ \Bigl({\beta_\eff^2 d^{1/2} {h}^6} + \beta^2 \gamma {h}^5\Bigr) \log \frac{1}{\delta} + {\beta_\eff^2 {h}^6} \log^2 \frac{1}{\delta}\,.
\end{aligned}
\end{align}
For the remaining term in~\eqref{eq:expansion-momentum}, we obtain
\begin{align*}
    \norm{\int_0^{{h}} \eia(t, {h}) \nabla V(X_{{h}/2}) \, \D t - \int_0^{{h}} \eia(t, {h}) \nabla V(X^+) \, \D t}^2 \lesssim \beta^2{h}^2 \norm{X_{{h}/2} - X^+}^2\,,
\end{align*}
and we have
\begin{align*}
    \norm{X_{{h}/2} - X^+}^2 &\lesssim \norm{\int_0^{{h}/2} \eib(t, \frac{{h}}{2}) \nabla V(X_t) \, \D t - \eic(0, \frac{{h}}{2}) \nabla V(X_0)}^2 \lesssim {\beta^2 {h}^3} \int_0^{{h}/2} \norm{X_t - X_0}^2 \, \D t \\
    &\lesssim {\beta^2 {h}^3} \int_0^{{h}/2} \Bigl\{{{h}} \int_0^t \norm{\eia(s,t) P_s \, \D s}^2 + {{h}} \int_0^t \norm{\eib(s,t) \nabla V(X_s) \, \D s}^2 \\
    &\qquad\qquad\qquad \qquad\qquad + \gamma \norm{\int_0^t \eib(s,t) \, \D B_s}^2 \Bigr\} \, \D t\,.
\end{align*}
For the first two terms, we use the same change of measure argument (Lemma~\ref{lem:high-prob-change} combined with the concentration in Lemma~\ref{lem:subgsn-score}) to bound them, while the last term uses Lemma~\ref{lem:brownian-subgaussian}. This allows us to obtain for ${h} \lesssim \frac{1}{\sqrt{\beta}}$ the bound with probability $1-\delta$,
\begin{align*}
     \norm{\int_0^{{h}} \eia(t, {h}) \nabla V(X_{{h}/2}) \, \D t - \int_0^{{h}} \eia(t, {h}) \nabla V(X^+) \, \D t}^2 \lesssim {\beta^4 {h}^8} \Bigl(d  +\log \frac{1}{\delta} \Bigr)\,.
\end{align*}
Thus, this term will be negligible.
Putting all of this together tells us that the momentum strong error is bounded with probability $1-\delta$ as
\begin{align}\label{eq:momentum-one-step-high-prob}
\begin{aligned}
    \norm{\int_0^{h} \eia(t, {h}) \{\nabla V(X^+) - \nabla V(X_t)\} \, \D t}^2 &\lesssim {\beta^2 \gamma {h}^5} \Bigl(d + \log \frac{1}{\delta} \Bigr) + {(\beta_H^2 + \beta_\eff^2) d {h}^6} \\
    &\qquad + {\beta_\eff^2 d^{1/2} {h}^6} \log \frac{1}{\delta} + {\beta_\eff^2 {h}^6} \log^2 \frac{1}{\delta}\,.
\end{aligned}
\end{align}
For the weak errors, all the above remains true if we insert the expectation into inside the expressions. However, term IV is zero since $X_r$ is adapted under~\eqref{eq:underdamped_langevin_diff}, 
\begin{align*}
    \E[\nabla^2 V(X_r) \, \D B_r] = 0\,.
\end{align*}
Thus, we default to the next term in our bound, which for $\gamma \asymp \sqrt{\beta}$ is, with probability $1-\delta$,
\begin{align}\label{eq:momentum-one-step-high-prob-weak}
\begin{aligned}
    \norm{\int_0^{h} \eia(t, {h}) \{\E \nabla V(X^+) - \E \nabla V(X_t)\} \, \D t}^2 &\lesssim {\beta^3 {h}^6} \Bigl(d + \log \frac{1}{\delta} \Bigr) + {(\beta_H^2 + \beta_\eff^2) d {h}^6}\\
    &\qquad + {\beta_\eff^2 d^{1/2} {h}^6} \log \frac{1}{\delta} + {\beta_\eff^2 {h}^6}\log^2 \frac{1}{\delta}\,.
\end{aligned}
\end{align}

\paragraph{Position error:} 
An It\^o--Taylor expansion gives
\begin{align*}
    \int_{0}^{h} &\eib(t,{h}) \nabla V(X_t) \, \D t = \int_0^{{h}} \eib(t, {h}) \nabla V(X_{{h}/3}) \, \D t + \int_0^{h} \int_{{h}/3}^t \eib(t,{h}) \nabla^2 V(X_s) P_s \, \D s \, \D t\,.
\end{align*}
Now, noting that $\abs{\eib(t, {h}) - ({h} -t)} \lesssim \gamma {h}^2$, and $\norm{\nabla^2 V} \lesssim \beta$, then we have by Cauchy--Schwarz,
\begin{align*}
    \norm{\int_0^{h} &\int_{{h}/3}^t \eib(t,{h}) \nabla^2 V(X_s) P_s \, \D s \, \D t}^2 \\
    &\lesssim \norm{\int_{0}^{h} \int_{{h}/3}^t ({h} -t)\nabla^2 V(X_s) P_s \, \D s \, \D t}^2  + \norm{\int_{0}^{h} \int_{{h}/3}^t (\eib(t, {h}) - ({h}-t))\nabla^2 V(X_s) P_s \, \D s \, \D t}^2\,.
\end{align*}
We can bound
\begin{align*}
    \norm{\int_{0}^{h} &\int_{{h}/3}^t (\eib(t, {h}) - ({h}-t))\nabla^2 V(X_s) P_s \, \D s \, \D t}^2 \\
    &\lesssim {h}^2 \int_0^{h} \int_{{h}/3}^t \norm{(\eib(t, {h}) - ({h}-t))\nabla^2 V(X_s) P_s}^2 \, \D s \, \D t 
    \lesssim \beta^2 \gamma^2 {h}^6 \int_0^{h} \int_{{h}/3}^t \norm{P_s}^2 \, \D s \, \D t\,. 
\end{align*}
With probability $1-\delta$, we know that if $\bs \mu_0$ has $\Renyi_q(\bs \mu_0 \mmid \bs \pi) \lesssim d^{1/2}$ for $q \geq 2$, then under $\tilde{\bs \mu}_t$ for all $t \geq 0$, we have with probability $1-\delta$ via Lemma~\ref{lem:renyi-change-measure},
\begin{align*}
    \beta^2 \gamma^2 {h}^6 \int_0^{h} \int_{{h}/3}^t \norm{P_s}^2 \, \D s \, \D t \lesssim \beta^2 \gamma^2 d{h}^8 + \beta^2 \gamma^2  {h}^8 \log \frac{1}{\delta}\,.
\end{align*}
By symmetrizing the integral, we find, noting that
\begin{align*}
    \int_{0}^{h} \int_{{h}/3}^t ({h}-t) \nabla^2 V(X_{h/3}) P_{h/3} \, \D s \, \D t = 0\,,
\end{align*}
then we have, expanding $\nabla^2 V(X_s) P_s$ around ${h}/3$,
\begin{align*}
    \norm{\int_{0}^{h} \int_{{h}/3}^t ({h}-t)\nabla^2 V(X_s) P_s \, \D s \, \D t}^2
    &=\norm{\int_{0}^{h} \int_{{h}/3}^t \int_{{h}/3}^s ({h} - t)\, \D[\nabla^2 V(X_r) P_r] \, \D s \, \D t}^2\\
    &\lesssim {h}^4 \int_0^{h} \int_{{h}/3}^t \norm{\int_{{h}/3}^s  \,\D [ \nabla^2 V(X_r) P_r]}^2  \, \D s \, \D t  \,. 
\end{align*}
We now again apply It\^o's lemma and change expectations to $\bs \pi$. All the terms are handled in a similar way to momentum error, gaining an extra two powers of ${h}$ as we have $\eib(t, {h})$ in place of $\eia(t, {h})$. We omit the calculations as they are straightforward. However, we note that we obtain the following intermediate result,
\begin{align}\label{eq:position-expansion-high-prob}
\begin{aligned}
    \norm{\int_0^{h} \int_{{h}/3}^t \eib(t,{h}) \nabla^2 V(X_s) P_s \, \D s \, \D t}^2 &\lesssim \beta^2 \gamma d {h}^7 + {(\beta_H^2 + \beta_\eff^2) d {h}^8} \\
    &\qquad+ \Bigl({\beta_\eff^2 d^{1/2} {h}^8} + \beta^2 \gamma {h}^7\Bigr) \log \frac{1}{\delta} + {\beta_\eff^2 {h}^8} \log^2 \frac{1}{\delta}\,.
\end{aligned}
\end{align}
Furthermore, we have
\begin{align*}
    \norm{\int_0^{{h}} \eib(t, {h}) \nabla V(X_{{h}/3}) \, \D t - \int_0^{{h}} \eib(t, {h}) \nabla V(X^-) \, \D t}^2 \lesssim \beta^2{h}^4 \norm{X_{{h}/3} - X^-}^2\,.
\end{align*}
This is bounded entirely analogous to the momentum errors. Thus, we find for the position strong errors
\begin{align}\label{eq:position-one-step-high-prob}
\begin{aligned}
    \norm{\int_0^{h} \eib(t, {h}) \{\nabla V(X^-) - \nabla V(X_t)\} \, \D t}^2 &\lesssim \beta^2 \gamma h^7 \Bigl(d + \log \frac{1}{\delta} \Bigr) + {(\beta_H^2 + \beta_\eff^2) d {h}^8} \\
    &\qquad + {\beta_\eff^2 d^{1/2} {h}^8} \log \frac{1}{\delta} + {\beta_\eff^2 {h}^8} \log^2 \frac{1}{\delta}\,.
\end{aligned}
\end{align}
The position weak errors are similarly handled, giving
\begin{align}\label{eq:position-one-step-high-prob-weak}
\begin{aligned}
    \norm{\int_0^{h} \eib(t, {h}) \{\E \nabla V(X^-) - \E \nabla V(X_t)\} \, \D t}^2 &\lesssim {\beta^3 {h}^8} \Bigl(d + \log \frac{1}{\delta} \Bigr) + {(\beta_H^2 + \beta_\eff^2) d {h}^8} \\
    &\qquad + {\beta_\eff^2 d^{1/2} {h}^8} \log \frac{1}{\delta} + {\beta_\eff^2 {h}^8} \log^2 \frac{1}{\delta}\,.
\end{aligned}
\end{align}
\end{proof-of-lemma}

\begin{proof-of-lemma}[{Theorem~\ref{thm:dmd-sharp-one-step}}]
Under $\mathbf Q$, $X^+$ will does not follow a straight-forward equation. Because of this, we first consider expansions under $\mathbf P$. This gives
\begin{align*}
    \norm{\int_0^{h} \eia(t, {h}) \nabla V(X_t)\, \D t - G^\msp_0}^2 &\lesssim \beta^2 {h}^2 \sup_{t \in [0, {h})} \norm{\int_0^{h} \{\nabla V(X_t) - \nabla V(X^+)\} \, \D t}^2 \\
    &\lesssim \beta^2 {h}^2\norm{X_{{h}/2} - X^+}^2 + \norm{\int_0^{h} \int_{{h}/2}^t\nabla^2 V(X_s) P_s\, \D s \, \D t}^2\,.
\end{align*}
Again, define $\Lambda^2 = \E_{\mathbf P}[\norm{\lambda_1}^2] + h^2 \E_{\mathbf P}[\norm{\lambda_2}^2]$. The first term is bounded by~\eqref{eq:same-time-interpolant}, which gives
\begin{align*}
    \beta^2 {h}^2 \E_{\mathbf P}[\norm{X_{{h}/2} - X^+}^2]
    &\lesssim {\beta^2 {h}^6} \Lambda^2 + \beta^4{h}^{10} \E_{\mathbf P}[\norm{\nabla V(X_0)}^2]  + \beta^4 {h}^8 \E_{\mathbf P}[\norm{P_0}^2] + \beta^4 \gamma d {h}^9\,.
\end{align*}
Using the initial R\'enyi bound and Lemma~\ref{lem:renyi-change-measure}, we can bound this by, for ${h} \lesssim \frac{1}{\sqrt{\beta}}$,
\begin{align*}
    \beta^2 {h}^2 \E_{\mathbf P}[\norm{X_{{h}/2} - X^+}^2]
    &\lesssim {\beta^2 {h}^6}\Lambda^2 + {\beta^4 d{h}^8 }\,.
\end{align*}
The second term is bounded with probability $1-\delta$ under $\mathbf P$, from~\eqref{eq:momentum-expansion-high-prob} and Lemma~\ref{lem:high-prob-change},
\begin{align*}
    \norm{\int_0^{h} \int_{{h}/2}^t \nabla^2 V(X_s) P_s\, \D s \, \D t}^2 &\lesssim \beta^2 \gamma d {h}^5 + {(\beta_H^2 + \beta_\eff^2) d{h}^6 } \\
    &\qquad+ \Bigl({\beta_\eff^2  d^{1/2}{h}^6} + \beta^2 \gamma {h}^5\Bigr) \log \frac{1}{\delta} + \beta_\eff^2 {h}^6 \log^2 \frac{1}{\delta}\,.
\end{align*}
We can integrate this expression to bound
\begin{align*}
    \E_{\mathbf P}\bigl[\norm{\int_0^{h} \eia(t, {h}) \nabla V(X_t)\, \D t - G^\msp_0}^2\bigr] &\lesssim \beta^2\gamma d {h}^5+ {(\beta_H^2 + \beta_\eff^2)d {h}^6} + {\beta^2 {h}^6} \Lambda^2 \,.
\end{align*}

\paragraph{Position errors:} This will end up being very similar in flavour to the momentum errors. 
\begin{align*}
    \norm{\int_0^{h} \eib(t, {h}) \nabla V(X_t)\, \D t - G^\msx_0}^2 &\lesssim \beta^2 {h}^4 \sup_{t \in [0, {h})} \norm{\int_0^{h} \{\nabla V(X_t) - \nabla V(X^-)\} \, \D t}^2 \\
    &\lesssim \beta^2 {h}^4\norm{X_{{h}/3} - X^-}^2 + {{h}^2}\norm{\int_0^{h} \int_{{h}/3}^t\nabla^2 V(X_s) P_s\, \D s \, \D t}^2\,.
\end{align*}
Both terms are now bounded similarly, and we obtain a similar bound as for the momentum errors but with two additional powers of ${h}$.

\paragraph{Completing the bounds:}
Putting these together, we find that
\begin{align*}
    {h} \Lambda^2 \lesssim {\beta^2 \gamma d {h}^4}+ {(\beta_H^2 + \beta_\eff^2)d {h}^5} + {\beta^2 {h}^5}\Lambda^2\,.
\end{align*}
For our choice of ${h}$, we can absorb the $\lambda$'s on the right side into the left, and so we obtain the bound
\begin{align*}
     \int_0^{h} \E \Bigl[\norm{\eia(t, {h}) \lambda_1 + \eib(t, {h}) \lambda_2}^2\Bigr] \, \D t \lesssim {\beta^2 \gamma d {h}^4} + {(\beta_H^2 + \beta_\eff^2)d {h}^5}\,.
\end{align*}
Substituting this into the $\KL$ bounds in Lemma~\ref{lem:DMD-renyi-generic} concludes the proof. 
\end{proof-of-lemma}

\subsection{Existence and approximation for DM--ULMC}\label{app:malliavin}

We state an existence result for our application of interest.
\begin{lemma}\label{lem:anticipating-SDE-existence}
    Consider the implicit equation which gives $X_{[0,h]} \in L^2_{\bf P}([0, h]; \R^d)$ in terms of $B_{[0, h]}$ through~\eqref{eq:DM--ULMC-multistep} with $N = 1$, $k = 0$.
    For the discrete approximation~\eqref{eq:DM--ULMC-approx}, the transform $\msf T$ is invertible if $h \lesssim \frac{1}{\beta^{1/2}}$ for a sufficiently small absolute constant, and the approximate solution converges to a solution $X_t$ of~\eqref{eq:DM--ULMC-multistep} in $L^2_{\mathbf P}([0, h]; \R^d)$, in the sense that $\norm{X - \hat X^\eta}^2_{L^2_{\mathbf P}([0, h]; \R^d)} \to 0$. Furthermore, a solution to~\eqref{eq:DM--ULMC-multistep} exists and is unique. 
\end{lemma}
\begin{proof}
    We will keep the realizations of $X_0 = x_0$ and $B_{[0, h]}$ fixed throughout this argument. First, we show that~\eqref{eq:DM--ULMC-multistep} has a unique solution using generic principles. Define $F(X_{[0, h]})$ to be the right side of~\eqref{eq:DM--ULMC-multistep}, or more precisely, let $F(X_{[0,h]}) \in \Hb$ with
    \begin{align*}
        F_t(X_{[0,h]}) = X_0 + \eib(0, t) P_0 - \int_0^t \eib(s, t) \MG_s^\opt(X_{[0, h]}) \, \D s + \sqrt{2\gamma} \int_0^t \eib(s,t) \,\D B_s\,. 
    \end{align*}
    Indeed, if $h$ is sufficiently small, consider two candidate solutions $X_{[0, h]}$, $\bar X_{[0, h]}$ for the fixed point equation  with the same initial starting point, $X_0 = \bar X_0$. Letting $P_{[0, h]}, \bar P_{[0, h]}$ be the corresponding momentum processes, recall that $P$ should solve its own fixed point equation (viewing $P_{[0, h]}$ as deterministic given $X_{[0, h]}$ and $B_{[0, h]}$),
    \begin{align*}
        P_t = \eia(0, t)P_0 - \int_0^t \eia(s, t) \MG_s^\opt(X_{[0, h]}) \, \D s + \sqrt{2\gamma} \int_0^t \eia(s,t) \,\D B_s\,. 
    \end{align*}
    we have, using that
    \begin{align*}
        \norm{\MG_{[0,h]}^\opt(X_{[0,h]}) - \MG_{[0,h]}^\opt(\bar X_{[0,h]})}^2_\Hb \lesssim \norm{\nabla V(X)_{[0,h]} - \nabla V(\bar X)_{[0,h]}}^2_{\Hb}\,,
    \end{align*}
    from which we get
    \begin{align*}
        \norm{P_{[0, h]} - \bar P_{[0, h]}}^2_{\Hb} &\lesssim {h^2} \norm{\nabla V(X)_{[0, h]} - \nabla V(\bar X)_{[0, h]}}^2_\Hb \\
        &\lesssim {\beta^2 h^2} \norm{X_{[0, h]} - \bar X_{[0, h]}}^2_\Hb\,.
    \end{align*}
    Furthermore,
    \begin{align*}
        \norm{X_{[0, h]} - \bar X_{[0, h]}}^2_{\Hb} &\lesssim {h^2} \norm{P_{[0, h]} - \bar P_{[0, h]}}^2_{\Hb}\,.
    \end{align*}
    Putting these together says that if $h\lesssim \frac{1}{\beta^{1/2}}$ for a small enough implied constant, then $F$ is a contraction on $\R^{d}$ with constant $c < 1/2$, and furthermore, via a fixed point theorem, that the solution must be unique.
    As a result, denote by $(X^*_{[0, h]}, P^*_{[0, h]})$ the solution to~\eqref{eq:DM--ULMC-multistep}.

    We also show the following stability result: for $X_{[0, T]} \in L^2_{\mathbf P}([0, h]; \R^d)$ with $X_0 = x_0$, we have
    \begin{align*}
        \norm{F(X_{[0, h]}) - X_{[0, h]}}^2_\Hb &\geq \frac{1}{2}\norm{X_{[0, h]} - X_{[0, h]}^*}^2_\Hb - \norm{F(X_{[0, h]}) - F(X_{[0, h]}^*)}^2_{\Hb} \\
        &\geq \frac{1}{2} \norm{X_{[0, h]} - X_{[0, h]}^*}^2_\Hb - c\norm{X_{[0, h]} - X_{[0, h]}^*}^2_{\Hb}\,,
    \end{align*}
    where $c < 1/2$ is the implied contraction constant from earlier. As a result, we obtain 
    \begin{align*}
        \norm{F(X_{[0, h]}) - X_{[0, h]}}^2_\Hb &\gtrsim \norm{X_{[0, h]} - X_{[0, h]}^*}^2_\Hb\,.
    \end{align*}
    
    For the second part of the claim, we proceed by a direct analysis. Define $F_\eta(X_{[0, h]})$ to be the right side of the approximation in~\eqref{eq:DM--ULMC-approx}. For processes $\hat X_{[0, h]}^\eta$ which are constants on intervals $[k\eta, (k+1)\eta)$, by comparison of~\eqref{eq:DM--ULMC-approx} and~\eqref{eq:DM--ULMC-multistep}, we get
    \begin{align}\label{eq:f-diff}
    \begin{aligned}
        \norm{F(\hat X^\eta_{[0, h]}&) - F_\eta(\hat X^\eta_{[0, h]})}^2_\Hb \\
        &\lesssim {\gamma h^3} \sup_{k \in [h/\eta]} \sup_{s \in [0, \eta]} \norm{B_{k\eta + s} - B_{k\eta}}^2 + {h^2} \cdot \eta^2 \norm{P_0}^2 + {h^3} \cdot \eta^2 \norm{\nabla V(\hat X^\eta)_{[0, h]}}^2_\Hb  \,.
    \end{aligned}
    \end{align}
    It remains to ensure that $\norm{\nabla V(\hat X^\eta)_{[0, h]}}^2_{\Hb}$ is bounded as $\eta \to 0$.

    We now prove this bound.
    \begin{align}\label{eq:x-approx}
        \norm{\hat X_{[0, h]}^\eta}^2_\Hb \lesssim h\norm{X_0}^2 + {h^3}\norm{P_0}^2 + {h^4}\norm{\hat \MG_{[0, h]}^\opt(\hat X_{[0, h]}^\eta)}^2_\Hb + \gamma h^3\sup_{t \in [0, h]} \norm{B_t}^2\,.
    \end{align}
    Furthermore, we can bound very crudely (along the lines of the proof of Lemma~\ref{lem:DMD-one-step-uld})
    \begin{align*}
        \norm{\hat \MG_{[0, h]}^\opt(\hat X_{[0, h]}^\eta)}^2_\Hb &\lesssim \norm{\nabla V(\hat X^\eta)_{[0, h]}}^2_{\Hb} \\
        &\lesssim \norm{\nabla V(\hat X^\eta)_{[0, h]} - \nabla V(0)_{[0, h]}}^2_{\Hb} + \norm{\nabla V(0)_{[0, h]}}^2_\Hb \\
        &\lesssim \beta^2\norm{\hat X_{[0, h]}^\eta}^2_{\Hb} + \norm{\nabla V(0)_{[0, h]}}^2_\Hb \,.
    \end{align*}
    Absorbing $\norm{\hat X_{[0, h]}^\eta}^2_\Hb$ into the left side of~\eqref{eq:x-approx} for $h \lesssim \frac{1}{\beta^{1/2}}$, we find that all the quantities on the right remain bounded above by some constant as we take $\eta \to 0$, so we can ensure that $\norm{\hat X_{[0, h]}^\eta}^2_\Hb = O(h)$ in this limit, which also implies bounds on $\norm{\nabla V(\hat X^\eta)_{[0, h]}}_\Hb^2$.
    
    Thus, we have, using stability and invariance of $\hat X^\eta$ under $F_\eta$, as well as the temporal approximation result to get
    \begin{align*}
        \norm{\hat X^\eta_{[0, h]} - X^*_{[0, h]}}^2_\Hb &\lesssim \norm{F(\hat X^\eta_{[0, h]}) - \hat X^\eta_{[0, h]}}^2_\Hb \\
        &\lesssim \norm{F(\hat X^\eta_{[0, h]}) - F_\eta(\hat X^\eta_{[0, h]})}^2_\Hb\,.
    \end{align*}
    It remains to show that this goes to $0$ as $\eta \to 0$ in appropriate sense. For $L^2_{\mathbf P}$ convergence, we note that
    \begin{align*}
        &\E_{\mathbf P}[\sup_{k \in [h/\eta]} \sup_{s \in [0, \eta]} \norm{B_{k\eta + s} - B_{k\eta}}^2] = \widetilde O(\eta)\,, \\
        &\E_{\mathbf P}[\norm{\nabla V(\hat X^\eta)_{[0, h]}}_\Hb^2] \lesssim {\gamma \beta^2 h^3} \E[\sup_{t \in [0, h]} \norm{B_t}^2] + \operatorname{const.} \lesssim {\gamma \beta^2 h^4 d} + \operatorname{const.}
    \end{align*}
    and thus the limit can be strengthened to $L^2_{\mathbf P}([0, T]; \R^d)$. Here, the constant terms depend on $\norm{X_0}, \norm{P_0}, \norm{\nabla V(0)}$ but will not affect the asymptotics as $\eta \to 0$. Substituting this into~\eqref{eq:f-diff} shows the convergence as $\eta \to 0$.

    Finally, we discuss invertibility of $\msf T$. Earlier, we showed boundedness of the operator norm $q \MD \Mu^\eta$, and so we need only show properness of the transform $\msf T$ induced by $\Mu^\eta$. This follows as, using the analysis of Lemma~\ref{lem:DMD-one-step-uld} when $h \lesssim \frac{1}{\beta^{1/2} q^{1/2}}$,
    \begin{align*}
        \norm{\tilde{\bs \xi}}^2 &\geq \norm{\bs \xi}^2 - q^2 \cdot O\Bigl(\norm{\hat \MG^\opt}^2 \Bigr)
        \geq \norm{\bs \xi}^2 - q^2 \cdot O\Bigl(\beta^2 h^4 \norm{\bs \xi}^2\Bigr) - \operatorname{const.},
    \end{align*}
    where the $\operatorname{const.}$ terms are independent of $\norm{\bs \xi}$. This shows that $\msf T$ is proper, and therefore invertible as required.
\end{proof}

One can fortify the lemma above in various ways, for instance by generalizing to other types of sufficiently smooth drifts, or incorporating a time-dependent diffusion coefficient. See~\cite[\S3.3.3]{nualart2006malliavin} for a more thorough discussion.

We also prove a useful corollary. It is a generalization of the exponential martingale property or Dol\'eans--Dade exponential in the standard It\^o calculus.
\begin{corollary}\label{cor:doleans-dade}
    Let $\bar \Mu \in L^2_{\bf P}([0, T]; \R^d)$, interpreted in~\eqref{eq:transformation-generic}, yield a well-defined, invertible transformation of $\omega$. Then,
    \begin{align*}
        \E_{\bf P} \Bigl[\abs{\Mdet(I - \MD \bar \Mu)(\omega)}\exp\Bigl(\Mdel \bar \Mu - \frac{1}{2} \norm{\bar \Mu}_{\Hb}^2 \Bigr) \Bigr] = 1\,.
    \end{align*}
\end{corollary}
\begin{proof}
    This follows from Theorem~\ref{thm:anticipating-girsanov}, as under invertibility, we have $\mathbf Q \ll \mathbf P$, where $\mathbf Q$ results from $\msf T_{\#} \mathbf P$, so the object in the expectation is a genuine Radon--Nikodym derivative. 
\end{proof}

\section{Helper lemmas}

The following can be found as an intermediate bound in the proof of~\cite[Lemma 13]{chen2023does}.
\begin{lemma}[{Moments of a Gaussian chaos}]\label{lem:frobenius-gaussian}
    Suppose $H$ is a 3-tensor with operator norm bounded by, for $\beta_H, \beta_T \in \R_+$,
    \begin{align*}
        \norm{H}_{\{1,2\}, \{3\} }\leq \beta_H\,, \qquad \norm{H}_{\{1,2,3\}} \deq \sqrt{\sum_{i,j,k=1}^d H_{i,j,k}^2} \leq \beta_T
    \end{align*}
    Then, if $\xi$ is a standard Gaussian, $Z \deq \norm{H[\xi, \xi]}_2^2$, then we have
    \begin{align*}
        \E[\abs{Z-\E Z}^p]^{\frac{1}{p}} \lesssim  \beta_T^2p^2 + (\beta_T^2 + \beta_H^2 d^{1/2}) p\,,
    \end{align*}
    where the implicit constant does not depend on $p, \beta_H, \beta_T, d$.
\end{lemma}

We upgrade this into a bound on the tail probabilities of the Gaussian chaos.
\begin{lemma}[Concentration of a Gaussian chaos]\label{lem:frobenius-gaussian-tail}
    Suppose $H$ is a 3-tensor with norms bounded as
    \begin{align*}
        \norm{H}_{\{1,2\}, \{3\} }\leq \beta_H\,, \qquad \norm{H}_{\{1,2,3\}} \leq \beta_T\,.
    \end{align*}
    Then, if $\xi \sim N(0, I_d)$ is a standard Gaussian, we have for an absolute constant $c > 0$,
    \begin{align*}
        \Pr\{\abs{\norm{H[\xi, \xi]}^2 - \E \norm{H[\xi, \xi]}^2} \geq t \} \leq 2\exp\Bigl(-c\min \Bigl\{\sqrt{\frac{t}{\beta_T^2}}, \frac{t}{\beta_H^2 d^{1/2} + \beta_T^2} \Bigr\} \Bigr)\,.
    \end{align*}
\end{lemma}
\begin{proof}
    Let $Z = \norm{H[\xi, \xi]}^2$. Using Chebyshev's inequality and Lemma~\ref{lem:frobenius-gaussian}, we find that,
    \begin{align*}
        \Pr\{\abs{Z - \E Z} \geq t\} \leq \frac{(C \max\{(\beta_H^2 d^{1/2} + \beta_T^2)p,  \beta_T^2 p^2\})^p }{t^p}\,.
    \end{align*}
    Here, $C$ is an absolute constant which we can assume $\geq 1$ without loss of generality.
    When $t \leq \frac{2C(\beta_T^2 + \beta_H^2 \sqrt{d})^2}{\beta_T^2}$, take $p = \frac{t}{2C(\beta_T^2 + \beta_H^2 \sqrt{d})}$ and use the first term in the maximum to find the bound
    \begin{align*}
        \Pr\{\abs{Z - \E Z} \geq t\} \leq \exp\Bigl(- \frac{ct}{\beta_H^2 d^{1/2} + \beta_T^2} \Bigr)\,.
    \end{align*}
    Otherwise, if $t > \frac{2C(\beta_T^2 + \beta_H^2 \sqrt{d})^2}{\beta_T^2}$, setting $p = \frac{\sqrt{t}}{\sqrt{2C}\beta_T}$ we use the second term in the maximum to find
    \begin{align*}
        \Pr\{\abs{Z - \E Z} \geq t\} \leq \exp\Bigl( -\frac{c\sqrt{t}}{\beta_T}\Bigr)\,.
    \end{align*}
    Putting these two together yields the bound.
\end{proof}

The following lemmas are a useful toolkit for analysis in R\'enyi divergence, and come from~\cite{chewisamplingbook}.
\begin{lemma}[{Suprema of Brownian motion, adapted from~\cite[Lemma 6.2.4]{chewisamplingbook}}]\label{lem:brownian-subgaussian}
    Let $(B_t)_{t \geq 0}$ be a $\bf P$-Brownian motion. For $\lambda, t > 0$ such that $\lambda < \frac{1}{2t}$,
    \begin{align*}
        \E \exp\Bigl(\lambda \sup_{s \in [0, t]} \norm{B_s}^2 \Bigr) \leq \Bigl(\frac{1+2\lambda t}{1-2\lambda t}\Bigr)^d\,.
    \end{align*}
\end{lemma}

\begin{lemma}[{R\'enyi weak triangle inequality, adapted from~\cite[Lemma 6.2.5]{chewisamplingbook}}]\label{lem:renyi-triangle}
    For $q > 1$, $\lambda \in (0,1)$, and measures $\mu, \nu, \pi$, we have
    \begin{align*}
        \Renyi_q(\mu \mmid \pi) \leq \frac{q-\lambda}{q-1} \Renyi_{q/\lambda}(\mu \mmid \nu) + \Renyi_{(q-\lambda)/(1-\lambda)}(\nu \mmid \pi)\,,
    \end{align*}
    and in particular we have
    \begin{align*}
        \Renyi_q(\mu \mmid \pi) \leq \frac{q-1/2}{q-1} \Renyi_{2q}(\mu \mmid \nu) + \Renyi_{2q-1} (\nu \mmid \pi)\,.
    \end{align*}
    For $\KL$, we have
    \begin{align*}
        \KL(\mu \mmid \pi) \leq 2\KL(\mu \mmid \nu) + \log \bigl(1+\chi^2(\nu \mmid \pi)\bigr)\,.
    \end{align*}
\end{lemma}

\begin{lemma}[{Score concentration, adapted from~\cite[Lemma 6.2.7]{chewisamplingbook}}]\label{lem:subgsn-score}
    Let $\pi \propto \exp(-V)$ and assume that $\nabla V$ is $\beta$-Lipschitz. Then $\nabla V$ is $\sqrt{\beta}$-subgaussian under $\pi$, such that for all $v \in \R^d$,
    \begin{align*}
        \E_\pi \exp \inner{\nabla V, v} \leq \exp\Bigl(\frac{\beta \norm{v}^2}{2} \Bigr)\,.
    \end{align*}
    In particular, for $0 < \delta < 1/2$, with probability at least $1-\delta$ under $\pi$,
    \begin{align*}
        \norm{\nabla V}^2 \lesssim \beta d + \beta \log \frac{1}{\delta}\,.
    \end{align*}
\end{lemma}

\begin{lemma}[{R\'enyi change of measure, adapted from~\cite[Lemma 6.2.8]{chewisamplingbook}}]\label{lem:renyi-change-measure}
    Suppose that under two measures $\mu, \pi$, and a test function $\phi: \Omega \to \R$, we have
    \begin{align*}
        \pi\{\phi \geq {h}\} \leq \psi({h})\,,
    \end{align*}
    for some growth function $\psi$. Then,
    \begin{align*}
        \mu\{\phi \geq {h}\} \leq \psi({h})^{(q-1)/q} \exp\Bigl(\frac{q-1}{q} \Renyi_q(\mu \mmid \pi) \Bigr)\,.
    \end{align*}
\end{lemma}

We will make repeated use of the following helper lemma.
\begin{lemma}[Sub-Weibull change of measure]\label{lem:high-prob-change}
    Suppose $\bs \mu_0$ satisfies $\Renyi_q(\bs \mu_0 \mmid \bs \pi) \lesssim d^{1/2}$ for $q \geq 2$. Then, if under $(X, P) \sim \bs \pi$, we have for some measurable function $f: \R^{2d} \to \R$ with probability $1-\delta$,
    \begin{align*}
        f(X, P) \lesssim \msf C_1 d + \msf C_2 d^{1/2} \log \frac{1}{\delta} + \msf C_3 \log^2 \frac{1}{\delta}\,,
    \end{align*}
    Then under $(X, P) \sim \tilde{\bs \mu}_t$ following~\eqref{eq:ULD} starting from $\bs \mu_0$ with target $\bs \pi$ for all $t \leq {h}$, we have with the same probability
    \begin{align*}
        f(X, P) \lesssim \max\{\msf C_1, \msf C_2, \msf C_3\} d + \msf C_2 d^{1/2} \log \frac{1}{\delta} + \msf C_3 \log^2 \frac{1}{\delta}\,.
    \end{align*}
\end{lemma}
\begin{proof}
    This is straightforward from Lemma~\ref{lem:renyi-change-measure}, noting that $\Renyi_q(\tilde{\bs \mu}_t \mmid \bs \pi) \leq \Renyi_q(\bs \mu_0 \mmid \bs \pi)$ via the data-processing inequality, as both arguments evolve along~\eqref{eq:underdamped_langevin_diff} and $\bs \pi$ is stationary under~\eqref{eq:underdamped_langevin_diff}.
\end{proof}

\begin{lemma}[Linearization of $\log \det$]\label{lem:log-det}
    Consider $M \in \R^{d\times d}$ with $\norm{M}_{\op} < 1$. Then, we have the expansion
    \begin{align*}
        \log \det(I_d + M) = \tr(M) - \frac{1}{2} \tr(M^2) + O(\norm{M}^3_{\op}\, d)\,.
    \end{align*}
\end{lemma}
\begin{proof}
    This can be obtained by Taylor expansion of $t \mapsto \log \det(I_d +tM)$, which is the logarithm of a polynomial in $t$ and hence analytic.
\end{proof}